\documentclass[12pt]{article}

\RequirePackage[colorlinks,citecolor=blue,urlcolor=blue]{hyperref}

\usepackage{amsmath,amssymb,amsthm}
\usepackage{graphicx}
\usepackage{subfigure}
\usepackage{hyperref}
\usepackage{multirow}

\makeatletter
\def\singlespace{\def\baselinestretch{1}\@normalsize}

\@addtoreset{equation}{section}
\renewcommand{\baselinestretch}{1.412}

\renewcommand{\theequation}{\arabic{section}.\arabic{equation}}

\newcommand{\vv}{\boldsymbol{v}}
\newcommand{\veee}{\tilde{\boldsymbol{e}}}
\newcommand{\vw}{\boldsymbol{w}}

\newcommand{\ve}{\boldsymbol{e}}

\newcommand{\vY}{\boldsymbol{Y}}
\newcommand{\vm}{\boldsymbol{m}}

\newcommand{\vx}{\boldsymbol{x}}
\newcommand{\vy}{\boldsymbol{y}}
\newcommand{\vu}{U}
\newcommand{\vbeta}{\boldsymbol{\beta}}
\newcommand{\valpha}{\boldsymbol{\alpha}}

\newcommand{\Hess}{\text{Hess}}
\newcommand{\hpca}{h_{\text{pca}}}

\newcommand{\NN}{\mathbb{N}}

\newcommand{\RR}{\mathbb{R}}

\newcommand{\EE}{\mathbb{E}}
\newcommand{\XX}{\mathbb{X}}
\newcommand{\WW}{\mathbb{W}}
\newcommand{\MM}{\text{M}}
\newcommand{\tr}{\mbox{tr}}
\newcommand{\Ric}{\mbox{Ric}}

\newcommand{\II}{\textup{II}}
\newcommand{\ud}{\textup{d}}
\newcommand{\supp}{\text{supp}}
\newcommand{\diag}{\text{diag}}

\newcommand{\argmin}{\operatornamewithlimits{argmin}}

\newcommand{\SSS}{\mathfrak{S}}
\newcommand{\QQQ}{\mathfrak{Q}}
\newcommand{\HHH}{\mathfrak{H}}
\newcommand{\EEE}{\mathfrak{E}}

\newtheorem{thm}{Theorem}[section]
\newtheorem{lem}[thm]{Lemma}
\newtheorem{col}{Corollary}[section]

\begin{document}

\pagestyle{empty}
\begin{center}
{\Large \bf
Local Linear Regression on Manifolds and its Geometric Interpretation
}
\vspace{0.1in}

 \centerline{\bf Ming-Yen Cheng and Hau-Tieng Wu}

\end{center}

\begin{abstract} 
High-dimensional data analysis has been an active area, and the main focuses have been variable selection and dimension reduction. In practice, it occurs often that the variables are located on an unknown, lower-dimensional nonlinear manifold. Under this manifold assumption, one purpose of this paper is regression and gradient estimation on the manifold, and another 
is developing a new tool for manifold learning.   To the first aim, we suggest directly reducing the dimensionality to the intrinsic dimension $d$ of the manifold, and performing the popular local linear regression (LLR) on a tangent plane estimate. An immediate consequence is a 
dramatic reduction in the computation time
when the ambient space dimension $p\gg d$. We provide rigorous 
theoretical justification of the convergence of the proposed regression and gradient estimators by carefully analyzing the curvature, boundary, and non-uniform sampling effects.  A bandwidth selector that can handle heteroscedastic errors is proposed.  
To the second aim, we analyze carefully the behavior of our regression estimator both in the interior and near the boundary of the manifold, and make explicit its relationship with manifold learning, in particular estimating the Laplace-Beltrami operator of the manifold. In this context, we also make clear that it is important to  use a smaller bandwidth in the tangent plane estimation than in the LLR. 
Simulation studies and the Isomap face data example are used to illustrate the computational speed and estimation accuracy of our methods. 
\end{abstract}


\noindent {\small \bf SHORT TITLE}: Manifold Adaptive Regression And Manifold Learning

\begin{singlespace}
\begin{footnotetext}
{ Ming-Yen Cheng is Professor, Department of Mathematics, National Taiwan University, Taipei 106, Taiwan (Email: cheng@math.ntu.edu.tw). Hau-Tieng Wu is Postdoctoral Research Associate, Program in Applied and Computational Mathematics, Princeton University, Princeton, NJ 08544, USA (Email: hauwu@math.princeton.edu). Cheng's research was supported in part by the National Science Council grant NSC97-2118-002-001-MY3 and the Mathematics Division, National Center of Theoretical Sciences (Taipei Office). The authors like to thank Professor Peter Bickel for instructive comments.}
\end{footnotetext}
\end{singlespace}

\newpage

\pagestyle{plain}

\section{Introduction}

High-dimensional data arise frequently in many fields of the contemporary science. 
In addition, it is common that the sample size is small relative 
to the dimensionality of the data. Such intrinsically complex data structure introduces new challenges in
statistical analysis and inference, and requires innovative methods and theories \cite{fan_lv:2008,  hall_marron_neeman:2005}.
In this context, we focus on the regression problem, which plays an important role in understanding the relationship between the response variable and the predictors.
Conventionally, the probability density function (p.d.f.) of the predictor vector is assumed to be non-degenerate. In this case, variable selection and dimension reduction are fundamental issues and have been extensively studied 
\cite{fan_li:2001, fan_peng:2004, zhang_jiang:2010, fan_lv:2008, fan_song:2010, li_liang:2008,  xia:2007, xia:2008}.
However, these problems remain difficult in the nonparametric regression setting, because commonly the models are built in the ambient space and 
the curse of dimensionality is a serious issue \cite{lafferty_wasserman:2008, fan_feng_song:2011,zhu_li:2011}.

Recently, it has been noticed that, in practice, the predictor vector often takes on values in a lower-dimensional, nonlinear manifold.
More specifically, in the cryo Electron Microscopy problem \cite{frank:2006}, the images are located on the $3$-dimensional manifold $SO(3)$; in the radar signal example the data can be modeled as being sampled from the Grassmannian manifold \cite{chikuse:2003};  natural images are argued to be lying on a Klein bottle \cite{carlsson_Tigran:2008}; the general manifold model for image and signal analysis is considered in \cite{Gabriel:2009}; and  spherical, circular and oriental data are distributed on special types of manifolds \cite{mardia_jupp:2000}; to name but a few.
Based on the manifold assumption, in the past few years,  
numerous papers have been devoted to  learning the manifold, or more generally the underlying 
structure \cite{coifman_lafon:2006, lerman_zhang:2010, vdm}, and a few have addressed regression on manifolds \cite{pelletier:2006, bickel_li:2007, aswani_bickel:2011}.

In the manifold learning literature, the Nadaraya-Watson kernel regression estimator has been used to construct an estimator of the Laplace-Beltrami operator of the manifold; however, to avoid the boundary  blowup  problem, Neuman's boundary condition is required \cite{coifman_lafon:2006}. When the $p$-dimensional predictor is non-degenerate in $\RR^p$, it is well known that the asymptotic bias of the traditional LLR in the Euclidean setup is related to the Laplacian of the regression function  and that it alleviates the boundary effect \cite{ruppert_wand:1994}. Thus, it is interesting to see if these properties still hold for some properly constructed  LLR in the manifold setup, as it will enable us to obtain a new estimator for the Laplace-Beltrami operator of the manifold with a different boundary condition.

Besides, due to the rich geometric structure, when the predictors are concentrated on a manifold, regression models that taking into account the geometric structure of the manifold are intuitively appealing. 
In \cite{pelletier:2006,loubes_pelletier:2008} the kernel regression estimator is constructed directly on the manifold, using the true geodesic distance both in determining the nearest neighbors and in constructing the kernel weights. 
Another approach is to employ the usual LLR in the ambient space $\RR^p$ with  regularization imposed on the coefficients in the directions perpendicular to a tangent plane estimate \cite{aswani_bickel:2011}.
However, there are several interesting and important issues left unsolved.
First, although the idea of constructing kernel estimators on the manifold in \cite{pelletier:2006,loubes_pelletier:2008} is appealing, it is unrealistic to make use of the geodesic distance. It is non-trivial to construct LLR on the manifold without knowing the manifold structure.
Second, it remains unknown whether the methods in \cite{aswani_bickel:2011} alleviate the boundary effect, and it is not obvious whether the asymptotic biases have any connections with the Laplace-Beltrami operator of the manifold.
Third, when $p$ is large, fitting LLR in $\RR^p$ as in \cite{aswani_bickel:2011} can be computationally expensive even if regularization has been imposed. 
Fourth,  in \cite{aswani_bickel:2011} the bandwidth used in the tangent plane estimation is the same as the one employed in the LLR. It is unclear if we can benefit from using different bandwidths in these two steps.  
Fifth, the quantity ``exterior derivative $d_xf|_{x_0}$'' in \cite[(4.5)]{aswani_bickel:2011} is subtle and the details are missing. 
Furthermore, the topology of the embedded manifold, in particular, the condition number \cite{NSW}, is another important issue that needs to be taken care of.

Motivated by the above observations, in this paper, we explore further the Riemannian geometric structure of the manifold, in particular the tangent bundle structure, and construct the LLR directly on an estimate of the tangent plane to the manifold, without knowing the geodesic distance and manifold structure. 
Specifically, we first estimate the intrinsic dimension $d$, and deal with the condition number issue when determining the nearest neighbors using the Euclidean distance. Subsequently,  we obtain an estimate of the embedded tangent plane based on 
local principal 
component analysis (PCA). Finally, we construct the LLR on the tangent plane estimate using the coordinates of the nearest neighbors with respect to the orthonormal basis. 
We call our approach the Manifold Adaptive Local Linear Estimator for the Regression (MALLER).
In addition, we suggest a procedure for selecting the bandwidth in the regression step that can handle heteroscedastic errors, which arise often in practice. A consequence of the proposed MALLER is an estimator for the gradient and the Laplace-Beltrami operator of the manifold.

Throughout this paper the dimension $p$ is kept as a fixed number and we assume the predictors are observed without any noise. Thus, if the sample size $n$ is large enough compared to the intrinsic dimension $d$, the tangent plane can be estimated accurately so that the dimensionality of the data can be reduced from $p$ to $d$. Under this circumstance, the first consequence is a much more computationally efficient scheme when $p$ is large and $p\gg d$, since all the computations in the regression step depend only on $d$. Another consequence is the ability to handle the practical situations where $n$ is less than $p$, 
in which case no sparsity conditions like those in \cite{aswani_bickel:2011} are needed for MALLER to work. The isomap face data analysis illustrates these points. 

We provide detailed theoretical justification of the convergence of MALLER by carefully analyzing the curvature, non-uniform sampling and boundary effects. In particular, the MALLER and gradient estimators achieve the respective optimal rates of convergence pertaining to nonparametric regression on $d$-dimensional manifolds. 
In addition, the subtle relationship between the bandwidth used in the tangent plane estimation and the one  used in the LLR is made explicit: it is crucial that the former should be of a smaller order than the latter, otherwise larger biases are introduced in the LLR on the tangent plane estimate and in the Laplace-Beltrami estimator mentioned below.  This issue is particularly important when estimating the Laplace-Beltrami operator.
Moreover, 
MALLER enjoys both the automatic boundary correction and the design adaptive properties possessed by the LLR in the $\RR^d$ setup \cite{ruppert_wand:1994}. These properties have strong implications in manifold learning. 
In particular, if the manifold has a smooth boundary, the Laplace-Beltrami operator estimated by our method MALLER is different from the one estimated by employing the Nadaraya-Watson kernel method, in the sense that the two are under different boundary conditions.  
Since the main focus of this paper is regression on manifolds, further theoretical properties and applications of the new estimator of the Laplace-Beltrami operator are left as a future work.

 
The rest of this paper is organized as follows. 
The proposed MALLER algorithm and a bandwidth selection procedure are  introduced in Sections \ref{algorithm} and \ref{bandwidth} respectively. Asymptotic results for the conditional mean squared errors of MALLER and the gradient estimator in both the interior and boundary of the manifold are given in Section \ref{theory}. In Section \ref{numerics} we examine finite sample performance of MALLER and compare it with those of \cite{aswani_bickel:2011} through one simulation study and application to the isomap face dataset, and we demonstrate the efficacy of our gradient estimator via a simulated example. Section \ref{diffusionmap} gives a brief introduction of the diffusion map framework and discusses application of MALLER to estimating the Laplace-Beltrami operator of the manifold.
In Section \ref{discussion}, besides addressing the relationship between MALLER and the NEDE algorithm in \cite[(4.6)]{aswani_bickel:2011}, we discuss various 
related open questions and future directions in both regression on manifolds and manifold learning. Proofs of the theoretical results can be found in the Supplementary, which also contains a brief introduction to the exterior derivative, covariant derivative and gradient of a function on the manifold.

\section{Model and Estimation Procedure}\label{algorithm}

Let  $Y$ denote the scalar response variable and let $X$ be a $p$-dimensional random vector. 
Assume that the distribution of $X$ is concentrated on a $d$-dimensional compact, smooth Riemannian manifold $\MM$ embedded in $\RR^p$ via $\iota:\MM\hookrightarrow \RR^p$, where $\MM$ may have boundary. 
We consider the following regression model 
\begin{equation}\label{model1}
Y=m(\iota^{-1}(X))+\sigma(\iota^{-1}(X))\,\epsilon,
\end{equation}
where $\epsilon$ is a random error independent of $X$ with $\EE (\epsilon)=0$ and $\operatorname{Var}(\epsilon)=1$, and both the regression function $m$ and the conditional variance function  $\sigma^2$ are defined on $\MM$. 

Let  $\{(X_l,Y_l)\}_{l=1}^n$ denote a random sample observed from model (\ref{model1}) 
with $\mathcal{X}:=\{X_l\}_{l=1}^n$ being sampled from $X$. 
Then, given $x\in\MM$, the problem is to estimate nonparametrically
$m({x})$, and its higher order covariant derivatives at ${x}$ if $m$ is smooth enough, based on $\{(X_l,Y_l)\}_{l=1}^n$. Here, $x$ may or may not belong to $\mathcal{X}$. {
For the sake of clearness, we should distinguish between the point $x\in\iota(\MM)$ and the point $\iota^{-1}(x)\in \MM$. However, to simplify the notation, for the rest of this paper we use the same symbol $x$ to denote $x\in\iota(\MM)$ or $\iota^{-1}(x)\in\MM$ and use $X$ to denote $X\in\iota(\MM)$ or $\iota^{-1}(X)\in\MM$ unless there is any ambiguity in the context.} In addition, throughout this paper we assume that the sample size $n\gg d$ {
and $X$ is not contaminated by error}.  
In the following subsections we discuss the steps in the MALLER algorithm : (1) estimating the intrinsic dimension $d$ of the manifold, (2) determining the true nearest neighbors of $x$ on $\MM$ using the Euclidean distance, (3) estimating the embedded tangent plane by local PCA, and (4) constructing LLR on the embedded tangent plane estimate. 
Before going into the details, the MALLER algorithm is summarized below. 

\noindent {\bf The MALLER Algorithm:}
\begin{enumerate}
\item Calculate the MLE intrinsic dimension estimate $\hat{d}$ in \cite{levina_bickel:2005}, and treat it as $d$.
\item For the given $x$, $\hpca$ and $h$ determine $\mathcal{N}^{\text{true}}_{x,\hpca}$ and $\mathcal{N}^{\text{true}}_{x,h}$, the two sets of estimates of the true nearest neighbors of $x$ on $\MM$ within a Euclidean ball of radius $\sqrt{\hpca}$ and $\sqrt{h}$ respectively, which are defined by (\ref{estimate:neighbors}). 
\item Employ the local PCA based on the points in $\mathcal{N}^{\text{true}}_{x,\hpca}$ to get an orthonormal basis $\{\vu_k(x)\}_{k=1}^d$ for the embedded tangent plane estimate at $x$, thus obtaining $\{\vx_{l}\}_{l=1}^n$, the coordinates of the projections of $\{X_l-x\}_{l=1}^n$ onto the affine space spanned by $\{\vu_k(x)\}_{k=1}^d$ with respect to this basis. See Section \ref{section:pca} for the details.
\item For given kernel $K$ and bandwidth $h$, obtain $\hat{\vbeta}_x$ by the LLR (\ref{functionalmfd}) based on $\left \{\vx_{l}: X_l \in \mathcal{N}^{\text{true}}_{x,h}\right\}$. Then we can compute the regression, embedded gradient and covariant derivative estimators defined in (\ref{algo:estimator:mhat}), (\ref{algo:estimator:CovDeri}) and (\ref{algo:estimator:Dmhat}) respectively.
\end{enumerate}

\subsection{Intrinsic dimension estimation}

Given the manifold assumption, in general the intrinsic dimension $d$ of the manifold $\MM$ is unknown a priori and needs to be estimated based on the sample $\mathcal{X}$. There exist many methods for estimating the intrinsic dimension 
and we have picked the maximum likelihood estimation (MLE)  method introduced in \cite{levina_bickel:2005} to estimate $d$ and denote the estimated dimension by $\hat{d}$. Since $d\ll n$, we assume the estimated dimension $\hat{d}$ is correct and hence will not distinguish between $d$ and $\hat{d}$.

\subsection{Determining the nearest neighbors}

{
Numerically determining the neighbors of $x\in\MM$ using the Euclidean distance is problematic due to the embedding structure of the manifold, that is, the condition number of the embedded manifold \cite{NSW}.} 
The reach of $\MM$ is defined as the largest number $\tau\geq0$ so that for every $0 \leq r<\tau$, the open normal bundle of $\MM$ of radius $r$ is still embedded in $\RR^p$. Since $\MM$ is assumed to be compact, we know $\tau>0$. The quantity $1/\tau$ is referred to as the ``condition number'' of $\MM$ \cite{NSW}. For the  given $x\in\MM$ and any $\delta>0$, denote respectively the set of Euclidean $\sqrt{\delta}$-neighbors of $x$ from $\mathcal{X}$ and the set of geodesic $\sqrt{\delta}$-neighbors of $x$ from $\mathcal{X}$ as
$$
\mathcal{N}^{\RR^p}_{x,\delta} = \big\{X_j\in \mathcal{X}:  \|X_j-x\|_{\mathbb{R}^p} < \sqrt{\delta} \big\}\mbox{ and } \mathcal{N}^\MM_{x,\delta} = \big\{X_j\in\mathcal{X}:  d(X_j,x) < \sqrt{\delta} \big\}, 
$$
where $d(\cdot,\cdot)$ is the geodesic distance. 
When $\delta$ is small enough, it is shown in Lemma \ref{lemma4} in the Supplementary that $\mathcal{N}^{\RR^p}_{x,\delta}$ is roughly the same as $\mathcal{N}^\MM_{x,\delta}$, which is the main fact rendering the whole algorithm feasible. However, when $\sqrt{\delta}$ exceeds $2\tau$, $\mathcal{N}^\MM_{x,\delta}$ might be a strict subset of $\mathcal{N}^{\RR^p}_{x,\delta}$. See Figure \ref{fig:condition}. 
This fact combined with the lack of a priori knowledge of $\MM$, in particular, the geodesic distance and the condition number $1/\tau$, lead to the problem. 
Since the manifold structure is our main concern, we need to learn $\mathcal{N}^\MM_{x,\delta}$. The problem is thus reduced to determining which points in $\mathcal{N}^{\RR^p}_{x,\delta}$ are in $\mathcal{N}^\MM_{x,\delta}$ and which are not.
To cope with this problem, we apply the ``self-tuning spectral clustering'' algorithm \cite{Zelnik-Manor_Perona:2004} to the set $\mathcal{N}^{\RR^p}_{x,\delta}$. We denote
\begin{equation}\label{estimate:neighbors}
\mathcal{N}^{\text{true}}_{x,\delta}:=\big\{X_j\in \mathcal{N}^{\RR^p}_{x,\delta}:  X_j\mbox{ is in the same cluster as }x \big\}.
\end{equation}
Then, according to Lemma \ref{lemma4} in the Supplementary, $\mathcal{N}^{\text{true}}_{x,\delta}$ is an accurate estimate of $\mathcal{N}^{\MM}_{x,\delta}$. 
\begin{figure}[h]	
\begin{center}
\subfigure{
\includegraphics[width=0.7\textwidth]{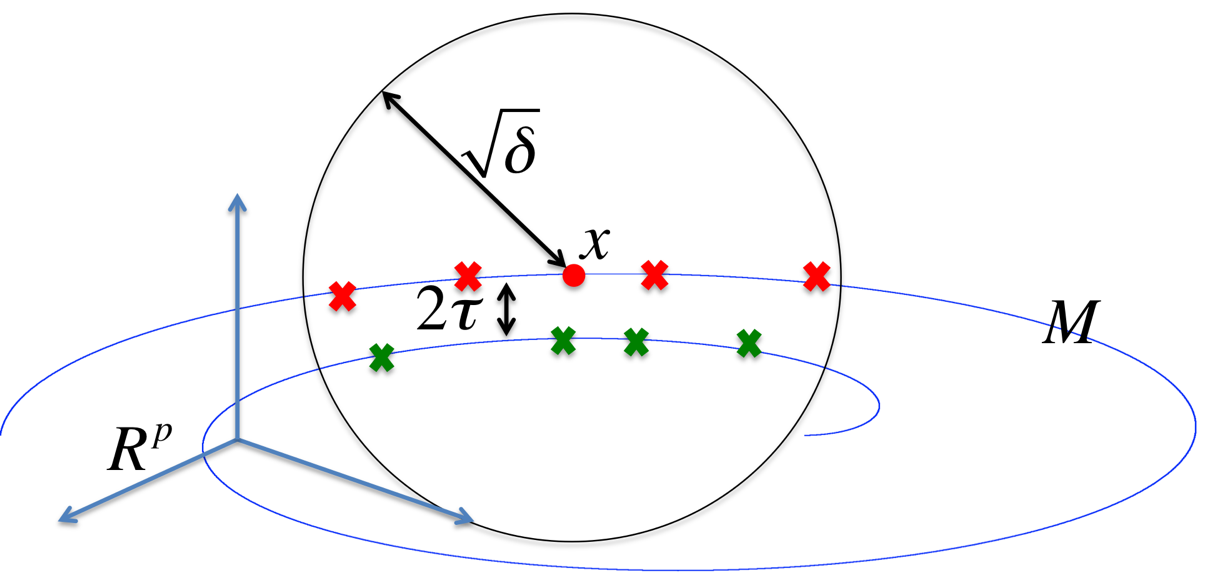}}
\end{center}
\vspace{-0.6cm}
\caption{\small Condition number. A $1$-dim manifold $\MM$ (blue curve) is embedded in $\RR^p$ with the condition number $1/\tau$. For the fixed $x\in\MM$, the black circle is of radius $\sqrt{\delta}$ and is centered at $x$. The Euclidean $\sqrt{\delta}$-neighbors of $x$,  $\mathcal{N}^{\RR^p}_{x,\delta}$, consists of both the red and green crosses. However, the geodesic $\sqrt{\delta}$-neighbors (true neighbors) of $x$, $\mathcal{N}^{\MM}_{x,\delta}$, consists of only the red crosses but not the green crosses. 
}
\label{fig:condition}
\vspace{-0.3cm}
\end{figure}

\subsection{Embedded tangent plane estimation} \label{section:pca}
Write the tangent plane of the manifold at ${x}\in\MM$ as $T_{{x}}\MM$. Denote by $\iota_*$ the total differential of $\iota$ 
and by $\iota_*T_{{x}}\MM$ the embedded tangent plane in $\RR^p$. Note that $\iota_*T_{{x}}\MM$ is a $d$-dimensional affine space inside $\RR^p$ which is tangential to $\MM$ at $x$. 
Next, we find an orthonormal basis of an approximation to the embedded tangent plane $\iota_*T_{{x}}\MM$. 
Fix $\hpca> 0$. Assume that there are $N_x$ points in $\mathcal{N}^{\text{true}}_{x,\hpca}$ and rewrite them as
$\mathcal{N}^{\text{true}}_{x,\hpca} = \{X_{x_1},\ldots,X_{x_{N_x}} \}.$
Let 
$$
\Sigma_x=\frac{1}{n}\sum_{l=1}^{N_x}\big(X_{x_l}-\mu_x\big)\big(X_{x_l}-\mu_x\big)^T
$$ 
be the sample covariance matrix of $\mathcal{N}^{\text{true}}_{x,\hpca}$, where $\mu_x$ is the sample mean of $\mathcal{N}^{\text{true}}_{x,\hpca}$.
Denote by $\{\vu_k(x)\}_{k=1}^d$ the eigenvectors corresponding to the $d$ largest eigenvalues of $\Sigma_x$, where $\vu_k(x)$ is a $p\times 1$ unit length column vector and $d$ is the dimension of the manifold $\MM$,  
and define a $p\times d$ matrix
\begin{equation}\label{algorithm:BX}
B_x:=\big[
\begin{array}{ccc}
\vu_{1}(x) & \ldots & \vu_{d}(x).
\end{array}
\big]
\end{equation}
Let $\vx_{l}=(\vx_{l,1},~\ldots,~\vx_{l,d})^T:=B^T_x(X_{l}-x)$, for $l=1,\ldots,n$. 

\subsection{Local linear regression on the tangent plane }
Choose a kernel function $K:[0,\infty]\to \RR$ so that $K|_{[0,1]}\in C^1([0,1])$ and $K|_{(1,\infty]}=0$ and a 
bandwidth $h>0$. Notice that $h$ is different from $\hpca$.
We solve the regression problem (\ref{model1}) at ${x}$ via considering the following local linear least squares fitting on the estimated tangent plane:
\begin{equation}\label{functionalmfd}
\hat{\vbeta}_x=\argmin_{\vbeta\in\RR^{d+1}} \sum_{l=1}^n\Big(Y_{l}-\beta_0-\sum_{k=1}^d\beta_k\vx_{l,k}\Big)^2\mbox{I}_{\mathcal{N}^{\text{true}}_{x,h}}(X_l)K_h(X_{l},x),
\end{equation}
where $\vbeta=(\beta_0,\beta_1,\ldots,\beta_d)^T$, $K_h(X_{l},x):=h^{-d/2}K\big(\|X_{l}-x\|_{\RR^p}\big/\sqrt{h}\big)$, and $\mbox{I}$ is the indicator function.
 Denote
\begin{equation}\label{def:Yandm}
\vY=\left(Y_{1},\ldots,Y_{n}\right)^T\quad\mbox{and}\quad\vm=\left(m(\iota^{-1}(X_{1})),\ldots,m(\iota^{-1}(X_{n}))\right)^T.
\end{equation}
Denote by $\XX_x$ the $n\times (d+1)$ design matrix related to $x$:
\begin{equation}\label{design}
\XX_x=\bigg[
\begin{array}{ccc}
1 & \dots & 1 \\
\vx_{1} & \dots & \vx_{n} \\
\end{array}
\bigg]^T,
\end{equation}
and $\WW_{x}$ the kernel weight matrix:
\begin{equation}\label{weighted}
\WW_{x}=\diag\left(K_h(X_{1},x)\mbox{I}_{\mathcal{N}^{\text{true}}_{x,h}}(X_1),\ldots,K_h(X_{n},x)\mbox{I}_{\mathcal{N}^{\text{true}}_{x,h}}(X_n)\right),
\end{equation}
which is a diagonal matrix of size $n \times n$. Then (\ref{functionalmfd}) can be written as
\begin{equation}\label{functional2}
\hat{\vbeta}_x=\argmin_{\vbeta\in\RR^{d+1}} (\vY-\XX_x\beta)^T\WW_{x}(\vY-\XX_x\beta).
\end{equation}
It is straightforward to show that the minimizer in (\ref{functional2}) is
$$
\hat{\vbeta}_x=(\XX_x^T\WW_{x}\XX_x)^{-1}\XX^T_x\WW_{x}\vY
$$ 
if $(\XX^T_x\WW_{x}\XX_x)^{-1}$ exists. The invertibility of $\XX^T_x\WW_{x}\XX_x$ will be shown in the Supplementary.
Our estimator of $m({x})$ MALLER is given by 
\begin{equation}\label{algo:estimator:mhat}
\hat{m}({x},h):=\vv_1^T\hat{\vbeta}_x=\vv_1^T(\XX_x^T\WW_{x}\XX_x)^{-1}\XX^T_x\WW_{x}\vY,
\end{equation}
where $\vv_k\in\RR^{d+1}$ is a $(d+1)\times 1$ unit vector with the $k$-th entry being 1. If the interest is to estimate the embedded gradient of $m$ at ${x}$, the following estimator is considered:{
\begin{equation}\label{algo:estimator:CovDeri}
\widehat{\iota_*\mbox{\tt{grad}} m({x})}:=\sum_{i=1}^d\widehat{\nabla_{\partial_i(x)}m}({x},h)U_i(x).
\end{equation}
where $\mbox{\tt{grad}}$ denotes the gradient, }
\begin{equation}\label{algo:estimator:Dmhat}
\widehat{\nabla_{\partial_i({x})}m}({x},h):=\vv_{i+1}^T\hat{\vbeta}_x,
\end{equation}
and $\{\partial_i({x})\}_{i=1}^d$ is the orthonormal basis of $T_{{x}}\MM$ closest to the estimated orthonormal basis $\{\vu_k(x)\}_{k=1}^d$ in the sense described in Lemma \ref{lemma6} in the Supplementary. {
We mention that the gradient on the manifold is closely related to the covariant derivative and the exterior derivative. The relationship between these quantities is summarized in the Supplementary.}

From (\ref{design}) and (\ref{functional2}) we can see that the key ingredient in the estimators (\ref{algo:estimator:mhat}), (\ref{algo:estimator:CovDeri}) and (\ref{algo:estimator:Dmhat}) is finding the coordinate of a given point related to a chosen basis and approximate locally the regression function by a linear function of that coordinate. 
A consequence of this fact is dimension reduction. Indeed, since $d$ may be much smaller than $p$, having obtained $\{\vx_{l}\}_{l=1}^n$, locally at $x$ we convert the $p$-dimensional regression problem to a $d$-dimensional one, by paying the price of additional sampling error coming from the tangent plane approximation and the curvature of the manifold.  Nonetheless, it is shown in Section \ref{theory} and Section \ref{numerics} 
that the effect of this extra sampling error on the MALLER is negligible and does not contribute to the leading term in the estimation error, provided that 
$\hpca$ is smaller than $h$.

\section{Bandwidth Selection}\label{bandwidth}

Selection of the local PCA bandwidth $\hpca$ is a less important problem than choosing the bandwidth $h$ in the regression step, as it is discussed in  Section \ref{theory} that $\hpca$ should be smaller than $h$ and of a smaller order than the optimal order of $h$.   We refer to \cite{vdm} for selection of $\hpca$. 
Suppose that for a given choice of $\hpca$, the tangent plane estimate has been obtained. The aim is finding the optimal value of $h$ so as to minimize the asymptotic conditional MSE  of the MALLER, which is provided in (\ref{thm:interior:cond_mse}). 
When the random errors are homoscedastic, the modified generalized cross-validation (mGCV) suggested in \cite{bickel_li:2007} can be used.
Specifically, let $\mathcal{H}_{\text{mGCV}}=\{\lambda_1,\ldots, \lambda_B\}$ be a set of candidate bandwidths, where $\lambda_i>0$, $i=1,\ldots,B$, and $B\in\NN$, 
and for each point $x$ we choose a block of data points $\{(X_{j}, Y_{j})\}_{j\in \mathcal{J}}$. For each $h\in \mathcal{H}_{\text{mGCV}}$, define the mGCV of $h$ by
$$
\text{mGCV}(h)=\Big(1+2\text{atr}_{\mathcal{J}}(h)\Big)\frac{1}{n_1}\sum_{j\in\mathcal{J}} \Big(Y_j-\hat{m}(X_j,h)\Big)^2,
$$
where 
$
\text{atr}_{\mathcal{J}}(h):=\frac{1}{n_1}\sum_{j\in\mathcal{J}}\vv_1^T(\XX^T_{X_j}\WW_{X_j}\XX_{X_j})^{-1}\vv_1h^{-d/2}K(0),
$ 
$n_1$ is the number of points in $\mathcal{J}$, 
and $\hat{m}(X_j,h)$ is the MALLER (\ref{algo:estimator:mhat}) of $m(X_j)$ based on bandwidth $h$.
Then $h_{\text{mGCV},\hat{m}}$ is chosen as the value of $h$ in $\mathcal{H}_{\text{mGCV}}$ which minimizes $\text{mGCV}(h)$.

In the presence of heteroscedastic random errors, we adopt the following additional step to deal with the bandwidth selection problem. Note that the optimal bandwidth has to balance between the conditional bias and the conditional variance, which depends on $\sigma^2(x)$.
Thus, with the pilot mGCV bandwidth $h_{\text{mGCV},\hat{m}}$ we get the first estimate of $m(X_l)$ by the MALLER, denoted as $\hat{m}(X_l,h_{\text{mGCV},\hat{m}})$, $l=1,\ldots,n$, and we apply the method suggested in \cite{chen_cheng_peng:2009} to estimate $\sigma^2(x)$. We choose this method since the random error $\epsilon$ might have a heavy tailed distribution. Defining the residuals as 
$$
\hat{r}_l:=\Big(Y_l-\hat{m}(X_l,h_{\text{mGCV},\hat{m}})\Big)^2,\, l=1,\ldots,n,
$$
we evaluate the following minimization problem
$$
(\hat{\alpha}_0(x),\hat{\valpha}(x))=\argmin_{\alpha_0\in\RR,\valpha\in\RR^d} \sum_{X_l\in\mathcal{N}^{\text{true}}_{x,h_{\text{mGCV},\hat{r}}}}\hspace{-20pt}\big( \log(\hat{r}_l+1/n)-\alpha_0-\valpha^TB_x^T(X_l-x)\big)^2K_{h_{\text{mGCV},\hat{r}}}(X_l,x),
$$
where $h_{\text{mGCV},\hat{r}}$ is the bandwidth determined by minimizing the mGCV upon the data set $\{(X_l, \log(\hat{r}_l+1/n))\}_{l=1}^n$. The estimated value of $\sigma^2(x)$ is then defined as
$$
\hat{\sigma}^2({x}):=e^{\hat{\alpha}_{0}(x)}\bigg[\frac{1}{n}\sum^n_{l=1}\hat{r}_le^{-\hat{\alpha}_{0}(x)}\bigg]^{-1}.
$$
Finally we select the bandwidth for MALLER given in (\ref{algo:estimator:mhat}) at ${x}\in\MM$. Denote the optimal bandwidth at ${x}$ as $h_{\text{opt}}(x)$. Fix a candidate bandwidths set $\mathcal{H}_{\text{opt}}=\{\lambda_1,\ldots, \lambda_B\}$, which  may be different from $\mathcal{H}_{\text{mGCV}}$, where $B\in\NN$ and $\lambda_i>0$, $i=1,\ldots,B$. For each $h\in\mathcal{H}_{\text{opt}}$, estimate the conditional bias  and the conditional variance of $\hat{m}({x},h)$ respectively by 
$$
\hat{b}({x},h) = 2[\hat{m}({x},h)-\hat{m}({x},h/2)],
$$
which is based on the asymptotic bias expression given in (\ref{rslt:interior:mfd}) of the Supplementary
and (\ref{col:boundary:bias}), and
$$
\hat{v}({x},h) = \vv_1^T(\XX_x^T\WW_{x}\XX_x)^{-1}\XX_x^T\WW_{x}\hat{\SSS}_x\WW_{x}\XX_x(\XX_x^T\WW_{x}\XX_x)^{-1}\vv_1,
$$
which is based on the finite sample variance expression given in (\ref{mfd:interior:var}) of the Supplementary, where $\hat{\SSS}_x$ is a $n\hspace{-2pt}\times\hspace{-2pt} n$ diagonal matrix 
$\hat{\SSS}_x=\diag\{\hat{\sigma}^2(X_1),\ldots, \hat{\sigma}^2(X_n)\}$.
The conditional MSE of $\hat{m}({x},h)$ is then estimated by 
$$
\widehat{\text{MSE}}({x},h):=\hat{b}({x},h)^2+\hat{v}({x},h).
$$
The value of $h\in\mathcal{H}_{\text{opt}}$, denoted as $\hat{h}_{\text{opt}}(x)$, which minimizes $\widehat{\text{MSE}}({x},h)$ is then used to approximate $h_{\text{opt}}(x)$. 
With $\hat{h}_{\text{opt}}(x)$, we can evaluate $\hat{m}({x}, \hat{h}_{\text{opt}}(x))$.
We do not claim the optimality of the bandwidth selection in this algorithm. For example, when the point ${x}$ is near the boundary of the manifold, the bandwidth should be chosen differently. We choose this bandwidth selection scheme since it is commonly used and is easy to implement \cite{ruppert:1997, fan_gijbels:1996}. Further study on the bandwidth selection problem in the manifold setup is an important and open problem and is out of the scope of this paper. 



\section{Theory}\label{theory}

Before stating the main theorems describing the behaviors of the proposed MALLER given in Section \ref{algorithm}, we set up more notation. 
Recall the assumption in Section 2 that $\MM$ is a $d$-dimensional compact smooth Riemannian manifold embedded in $\RR^p$ via $\iota$.
Let the metric $g$ on $\MM$ be the one induced from the canonical metric of the ambient space $\RR^p$. 
The exponential map at ${x}\in\MM$ is denoted as $\exp_{{x}}$. Denote by $d({x},{y})$ the distance between ${x},{y}\in\MM$. 
The volume form on $\MM$ induced from $g$ is denoted as $\ud V$.
Given $\delta\geq 0$, denote the set of points close to the boundary $\partial\MM$ with distance less than $\delta$ as
\begin{equation}\label{conditions:statement:Csqrthx}
\MM_{\delta}=\big\{{x}\in\MM:~\min_{{y}\in\partial\MM}d({x},{y})\leq \delta\big\}.
\end{equation}
When $\delta>0$ is small enough, we denote the geodesic ball with radius $\delta$ and center ${x}\in\MM$ as $B^\MM_{\delta}({x})$. Denote $B^{\RR^q}_{\delta}(x)$ as the ball in $\RR^q$, $q\in\NN$, with radius $\delta$ and center $x\in\RR^q$ and $S^{q-1}$ as the standard $q-1$ sphere embedded in $\RR^q$ with the induced metric. Define 
\begin{equation}\label{approximate_geodesic_ball}
{
\tilde{B}^{\MM}_\delta({x}):=\iota^{-1}\left(B^{\RR^p}_\delta(x)\cap \iota(\MM)\right)\subset \MM,
}
\end{equation} 
which is an approximate of the geodesic ball $B^\MM_\delta({x})$. Denote by $\nabla$ the Levi-Civita connection, $\Delta$ the Laplace-Beltrami operator and $\Hess$ the Hessian operator of $(\MM,g)$. Denote by $\Ric$ the Ricci curvature of $(\MM,g)$. The second fundamental form of the embedding $\iota$ at ${x}$ is denoted by $\II_{{x}}$. 

\subsection{Assumptions}
Let the random vector $X:\Omega \rightarrow \RR^p$ be a measurable function with respect to the probability space $(\Omega,\mathcal{F},P)$. {
To make the definition clear, in this paragraph we make clear the role of $\iota$ to distinguish between $x\in\MM$ and $\iota(x)\in\iota(\MM)$. }Suppose the range of $X$ is supported on $\iota(\MM)$. In this case, the p.d.f. of $X$ is not well-defined  as a function on $\RR^p$ if the intrinsic dimension $d$ of $\MM$ is less than $p$. To define properly the p.d.f. of $X$, let $\tilde{\mathcal{B}}$ be the Borel sigma algebra of $\iota(\MM)$, and denote by $\tilde{P}_X$ the probability measure of $X$, defined on $ \tilde{\mathcal{B}}$, induced from $P$. Assume that $\tilde{P}_{X}$ is absolutely continuous with respect to the volume measure on $\iota(\MM)$, that is, $\ud \tilde{P}_{X}(x)=f(\iota^{-1}({x}))\iota_*\ud V(x)$, where $f\in C^2(\MM)$. Thus, for an integrable function $\zeta:\iota(\MM)\rightarrow \RR$, we have
\begin{eqnarray}
\hspace{-8pt}&&\hspace{-8pt}\EE \zeta(X)=\int_{\Omega}\zeta(X(\omega))\ud P(\omega)=\int_{\iota(\MM)} \zeta(x)\ud \tilde{P}_{X}(x)\nonumber\\
&=&\int_{\MM} \zeta(x) f(\iota^{-1}({x})) \iota_*\ud V(x)
=\int_{\MM} \zeta(\iota(y)) f(y) \ud V(y),\label{definition:expectation:manifold}
\end{eqnarray}
where the second equality follows from the fact that $\tilde{P}_X$ is the induced probability measure, and the last one comes from the change of variable $x=\iota(y)$.
In this sense we interpret $f$ as {\it the p.d.f. of $X$ on $\MM$}. 

 
The kernel function $K:[0,\infty]\rightarrow \RR$ used in the proposed MALLER is assumed to be compactly supported in $[0,1]$ so that $K|_{[0,1]}\in C^1([0,1])$. Denote 
$$
\mu_{i,j}:=\int_{B^{\RR^d}_1(0)} K^i(\|u\|_{\RR^d})\|u\|_{\RR^d}^j\ud u
$$ 
and we normalize $K$ so that $\mu_{1,0}=1$.
Note that we can also consider more general kernel functions. For example, any $C^1(\RR)$ function with proper decaying property can be chosen. More general bandwidth like a positive definite symmetric bandwidth matrix $H$ considered in \cite{ruppert_wand:1994} can also be considered. Since the analysis under these more general conditions is the same except for the wrinkle caused by the extra error terms, we focus on the above setup to make the analysis clear. 

We make the following assumptions in the analysis.
\begin{itemize}
\item[(A1)] $h\to 0$ and  $nh^{d/2}\to \infty$ as $n\rightarrow \infty$.
\item[(A2)] $f$ belongs to $C^2(\MM)$ and satisfies 
\begin{equation}\label{conditions:statement:fcond}
0<\inf_{{x}\in\MM} f({x})\leq \sup_{{x}\in\MM}f({x})<\infty.
\end{equation}
\item[(A3)] For every given $h>0$ and every point ${x}\in\MM_{\sqrt{h}}$, the set 
$B^\MM_{\sqrt{h}}({x})\cap \MM$ 
contains a non-empty interior set. 
The purpose of this assumption is to avoid the potential degeneracy near the boundary. 
\item[(A4)] 
Assume that $\hpca^{1/2}< \min(2\tau,\text{inj}(\MM))$ and $h^{1/2}< \min(2\tau,\text{inj}(\MM))$, where $\text{inj}(\MM)$ is the injectivity radius of $\MM$ and $1/\tau$ is the condition number of $\MM$ \cite{NSW}. Please see step 2 of the algorithm for precise definition of $\tau$.
\end{itemize}

\subsection{Main Theory}

We state our main theorems here and postpone the proofs to the Supplementary. 

\begin{thm}\label{thm:interior}
Suppose $\hpca\asymp n^{-2/(d+1)}$ and $h\geq\hpca$. When ${x}\in\MM\backslash\MM_{\sqrt{h}}$, the conditional mean square error (MSE) of the estimator $\hat{m}({x},h)$ is
\begin{equation}\label{thm:interior:cond_mse}
\begin{split}
&\text{MSE}\{\hat{m}({x},h)| \mathcal{X}\}
=h^2\frac{\mu_{1,2}^2}{4d^2}(\Delta m({x}))^2+\frac{1}{nh^{d/2}}\frac{\mu_{2,0}\sigma^2({x})}{f({x})}\\
&\qquad\qquad+O(h^3+h^2\hpca^{3/4})+O_p\Big(\frac{1}{n^{1/2}h^{d/4-2}}+\frac{1}{nh^{d/2-1}}+\frac{1}{n^{3/2}h^{3d/4}}\Big).
\end{split}
\end{equation}
\end{thm}
Next, we consider the case when ${x}$ is close to the boundary. To ease the notation, for ${x}\in\MM_{\sqrt{h}}$ and $h>0$, define a $(d+1)\times(d+1)$ matrix $\nu_{i,x}$:
\begin{equation}
\nu_{i,x}:=\,\left[\begin{array}{cc}
\nu_{i,x,11} & \nu_{i,x,12}\\
\nu^T_{i,x,12} & \nu_{i,x,22}
\end{array}
\right]
:=\,\left[\begin{array}{cc}
\int_{\frac{1}{\sqrt{h}}\mathfrak{D}({x})}K^i(\|u\|)\ud u & \int_{\frac{1}{\sqrt{h}}\mathfrak{D}({x})}K^i(\|u\|)u^T\ud u\\
\int_{\frac{1}{\sqrt{h}}\mathfrak{D}({x})}K^i(\|u\|)u\ud u & \int_{\frac{1}{\sqrt{h}}\mathfrak{D}({x})}K^i(\|u\|)uu^T\ud u
\end{array}
\right],\label{thm1:statement:nuixdef}
\end{equation}
where for $i=1,2$, $\nu_{i,x,11}\in\RR$, $\nu_{i,x,12}$ is a $1\times d$ matrix, $\nu_{i,x,22}$ is a $d\times d$ matrix and
\begin{eqnarray}
\mathfrak{D}({x}):= \exp_{{x}}^{-1}(B^\MM_{\sqrt{h}}({x})\cap \MM)\subset T_{{x}}\MM.\label{thm1:statement:frakDdef}
\end{eqnarray}
We also define
\begin{equation}\label{thm2:C}
C:=\bigg[\begin{array}{cc}
1 & 0\\
0 & h^{\frac{1}{2}}I_{d}
\end{array}\bigg].
\end{equation} 
Here, $I_{k}$ denotes the $k\times k$ identity matrix for any $k\in\NN$.

\begin{thm}\label{thm:boundary}
Suppose ${x}\in\MM_{\sqrt{h}}$, $\hpca\asymp n^{-2/(d+1)}$ and $h\geq \hpca$.
The conditional MSE of the estimator $\hat{m}({x},h)$ is
\begin{eqnarray}
&&\hspace{-24pt}\text{MSE}\{\hat{m}({x},h)| \mathcal{X}\}= \frac{h^2}{4}\frac{[\tr \big(\Hess m({x})\nu_{1,x,22}\big)]^2}{\nu^2_{1,x,11}}+\frac{\vv^T_1\nu_{1,x}^{-1}\nu_{2,x}\nu_{1,x}^{-1}\vv_1}{nh^{\frac{d}{2}}} \frac{\sigma^2({x})}{f({x})}\label{thm:bdry:cond_mse}\\
&&\hspace{-32pt}+O_p\Big(\hpca^{3/4}h^{3/2}+\hpca^{1/2}h^2\Big)+O_p\Big(\frac{1}{n^{1/2}h^{d/4-2}}+\frac{1}{nh^{d/2-1/2}}+\frac{1}{n^{3/2}h^{3d/4}}\Big)\nonumber
\end{eqnarray}
\end{thm}
Notice that in both Theorem \ref{thm:interior} and \ref{thm:boundary}, the minimum of the conditional MSE is achieved when $h\asymp n^{-2/(d+4)}$, which is strictly larger than $\hpca$.

\begin{col}\label{col:smoothbdry}
Suppose $\partial\MM$ is smooth, ${x}\in\MM_{\sqrt{h}}$, $\hpca\asymp n^{-2/(d+1)}$ and $h\geq \hpca$. Then the conditional bias of $\hat{m}({x},h)$ is asymptotically a linear combination of the second order covariant derivative of $m$:
\begin{equation}\label{col:boundary:bias}
\EE\{\hat{m}({x},h)-m({x})| \mathcal{X}\}=\frac{h}{2}\sum_{k=1}^dc_k({x})\nabla^2_{\partial_k,\partial_k}m({x})+O_p(h^{\frac{1}{2}}\hpca^{3/4}+h\hpca^{1/2})+O_p\Big(\frac{1}{n^{\frac{1}{2}}h^{\frac{d}{4}-1}}\Big),
\end{equation}
where $\{\partial_k\}_{k=1}^d$ is a normal coordinate determined in Lemma \ref{lemma6} of the Supplementary and $c_k({x})$ is uniformly bounded for all $k=1,\ldots,d$.
\end{col}

Recall that when the p.d.f. of the random vector $X$ is well-defined on $\RR^p$,
denoted as $f$, so that $\supp f$ satisfies some weak conditions, it is shown in \cite{ruppert_wand:1994} that the conventional LLR is unbiased up to the second order term even when $x$ is close to the boundary. Additionally, the LLR is design adaptive, that is, the asymptotic bias does not depend on $f$. These properties render the LLR popular in applications. In the degenerate case i.e. $X$ lies on the manifold $\MM$, we can see from the proofs of Theorem \ref{thm:interior} and Theorem \ref{thm:boundary} that MALLER also processes these nice properties. There properties of MALLER have important implications from the manifold learning viewpoint, which will be discussion in Section \ref{diffusionmap}. 

\subsection{Gradient and Covariant Derivative Estimate}

When the p.d.f. $f$ of $X$ is non-degenerate on $\RR^p$, it is well known that the traditional LLR provides an estimate of the gradient of $m$ \cite{ruppert_wand:1994, fan_gijbels:1996}. In the manifold setup, the notion of differentiation is generalized naturally to the ``covariant derivative'', and hence the gradient if the manifold is Riemannian. A brief introduction of the notion of covariant derivative, gradient, exterior derivative and their relationship is provided in the Supplementary \ref{appendix:background}. In this subsection, we show that MALLER provides an estimate of the covariant derivative of $m$.

\begin{thm}\label{thm:interior:cov}
Suppose ${x}\in\MM\backslash\MM_{\sqrt{h}}$, $\hpca\asymp n^{-2/(d+1)}$ and $h\geq\hpca$. The conditional MSE for the estimator $\widehat{\nabla_{\partial_i({x})}m}({x},h)$ given in (\ref{algo:estimator:Dmhat}) is
\begin{equation}
\begin{split}
&\text{MSE}\{\widehat{\nabla_{\partial_i({x})}m}({x},h)|\mathcal{X}\}= h^2\Bigg[\frac{\mu_{1,2}}{d}\frac{\nabla_{\partial_i} f({x})}{f({x})}\Delta m({x})-\frac{\mu_{1,2}d\int_{S^{d-1}}\theta^T\Hess m({x})\theta \theta \nabla_\theta f({x})\ud \theta}{|S^{d-1}|f({x})}\Bigg]^2\nonumber\\
&+\frac{1}{nh^{\frac{d}{2}+1}}\frac{d\mu_{2,2}\sigma^2({x})f({x})}{\mu^2_{1,2}}+O_p(h^{\frac{5}{2}}+h^{\frac{3}{2}}\hpca^{\frac{3}{4}})+O_p\Big(\frac{1}{n^{\frac{1}{2}}h^{\frac{d}{4}-\frac{3}{2}}}+\frac{1}{nh^{\frac{d}{2}}}+\frac{1}{n^{\frac{3}{2}}h^{\frac{3d}{4}+1}}\Big),\nonumber
\end{split}
\end{equation}
where $\{\partial_i({x})\}_{i=1}^d$ is an orthonormal basis of $T_{{x}}\MM$ described in Lemma \ref{lemma6} of the Supplementary.
\end{thm}

\begin{thm}\label{thm:boundary:cov}
Suppose ${x}\in\MM_{\sqrt{h}}$, $\hpca\asymp n^{-2/(d+1)}$ and $h\geq \hpca$. The conditional  MSE for the estimator $\widehat{\nabla_{\partial_i({x})}m}({x},h)$ given in (\ref{algo:estimator:Dmhat}) is
\begin{eqnarray}
\lefteqn{\text{MSE}\{\widehat{\nabla_{\partial_i({x})}m}({x},h)| \mathcal{X}\}=h\bigg(\frac{\vv^T_{i+1}\nu^{-1}_{1,x}}{2}\int_{\frac{1}{\sqrt{h}}\mathfrak{D}({x})}K(\|u\|)u^T\Hess m({x})u\left[\begin{array}{c}
1\\ u
\end{array}\right]\ud u\bigg)^2} \nonumber\\
&&\hspace{-20pt}+\frac{\vv^T_{i+1}\nu_{1,x}^{-1}\nu_{2,x}\nu^{-1}_{1,x}\vv_{i+1}}{nh^{\frac{d}{2}+1}}\frac{\sigma^2({x})}{f({x})}+O_p\Big(h^{\frac{1}{2}}\hpca^{\frac{3}{4}}+h\hpca^{\frac{1}{2}}\Big)+O_p\Big(\frac{1}{n^{\frac{1}{2}}h^{\frac{d}{4}-\frac{3}{2}}}+\frac{1}{nh^{\frac{d}{2}+\frac{1}{2}}}+\frac{1}{n^{\frac{3}{2}}h^{\frac{3d}{4}}}\Big),\nonumber
\end{eqnarray}
where $\{\partial_i({x})\}_{i=1}^d$ is an orthonormal basis of $T_{{x}}\MM$ described in Lemma \ref{lemma6} of the Supplementary.
\end{thm}
 
{
Based on Theorem \ref{thm:interior:cov}, \ref{thm:boundary:cov} and Section \ref{appendix:background} of the Supplementary, we know that the estimator (\ref{algo:estimator:CovDeri}) indeed can be used to estimate the embedded gradient of $m$. Since the application of the estimate of the gradient is not the focus of this paper, we refer the readers to \cite{coifman_lafon:2006,Mukherjee:2010}.}

\section{Numerical Examples}\label{numerics}
To demonstrate the applicability of the proposed algorithm MALLER, we test it on a series of simulations and a real dataset and compared it with the nonparametric exterior derivative estimator (NEDE), nonparametric adaptive lasso exterior derivative estimator (NALEDE), nonparametric exterior derivative estimator for the ``large $p$, small $n$'' (NEDEP) and nonparametric adaptive lasso exterior derivative estimator for the ``large $p$, small $n$'' (NALEDEP) proposed in \cite{aswani_bickel:2011}, for which the codes are provided by the authors of \cite{aswani_bickel:2011}\footnote{\url{http://www.eecs.berkeley.edu/~aaswani/EDE_Code.zip}}. The code for implementation of MALLER is in the authors' homepage\footnote{\url{http://www.math.princeton.edu/~hauwu/regression.zip}}. 


All the observed values of the predictors in both the training dataset and the testing dataset are normalized by $x^0_l :=(x_l-\hat{\mu})/s$, where $\hat{\mu}$ is the sample mean of $\{x_l\}_{l=1}^{n}$, $l=1,\ldots,n+10$ and $s=\max_{i,j=1,\ldots,n}\|x_i-x_j\|_{\RR^p}$. In order to facilitate the notation we write $x_l$ instead of $x^0_l$ in the sequel. In step 1 of our algorithm, we used the MLE dimension estimation code provided by the authors of \cite{levina_bickel:2005}\footnote{\url{ http://www.stat.lsa.umich.edu/~elevina/mledim.m}} to evaluate the intrinsic dimension of the manifold. In step 2, we used the code provided by the authors of \cite{Zelnik-Manor_Perona:2004}\footnote{\url{http://www.vision.caltech.edu/lihi/Demos/SelfTuningClustering.html}}. In step 3, we chose  $\hpca=0.015$. In the bandwidth selection step, for each regressant, we worked out the bandwidth selection procedure given in Section \ref{bandwidth} on $21$ logarithmically equi-spaced candidate bandwidths in the interval $[0.01, 0.1]$ when $d=1$ and $[0.01, h_d]$ when $d>1$, where 
\begin{equation}\label{choice:hd}
h_d=\frac{1}{4}\bigg( \frac{d\Gamma(d/2)}{\sqrt{\pi}\Gamma\left((d+1)/2\right)} \bigg)^{2/d}(0.1)^{1/d}.
\end{equation}
This choice of $h_d$ is motivated by the following facts. Fix $d>1$. The volume of $S^d$ is $|S^d|=\frac{2\pi^{\frac{d+1}{2}}}{\Gamma(\frac{d+1}{2})}$, where $\Gamma$ is the Gamma function, and the volume of a geodesic ball of radius $0<\delta(d)\ll1$ centered at ${x}\in S^d$, denoted as $B^{S^d}_{\delta(d)}({x})$, is approximately $\frac{\delta(d)^{d}|S^{d-1}|}{d}=\frac{2\pi^{d/2}\delta(d)^{d}}{d\Gamma(d/2)}$. Thus, the ratio of the volume of $B^{S^d}_{\delta(d)}({x})$ to $|S^{d}|$ is $r(d,\delta(d))=\frac{\delta(d)^{d}\Gamma((d+1)/2)}{\sqrt{\pi}d\Gamma(d/2)}$. Suppose $\delta(d)=\delta\ll 1$ for all $d$, then $r(d,\delta)$ gets smaller as $d$ increases.  That is, if the number of data points sampled from $S^d$ is the same and $\delta(d)$ is fixed for all $d$, the number of data points located in $B^{S^d}_{\delta(d)}({x})$ decreases to zero exponentially. This fact plays a role in the numerics, especially in the bandwidth selection problem, since in practice the number of neighboring points is not controllable. We thus choose the largest bandwidth $h_d$ by solving 
$\frac{(2\sqrt{h_d})^{d}\Gamma((d+1)/2)}{\sqrt{\pi}d\Gamma(d/2)} = r(1,0.1)=\frac{\sqrt{0.1}}{\pi}$,
which leads to (\ref{choice:hd}).
We emphasize the non-optimality of this scheme to set the candidate bandwidths for general manifolds of dimension $d$, which is out of the scope of this paper. The kernel function $K$ used in step 4 of our MALLER algorithm was taken as $K(u)=\exp(-7u^2)\mbox{I}_{[0,1]}(u)$.

In Sections \ref{simulation2} -- \ref{isomapface} we report the root average square estimation error (RASE) to measure the accuracy of different estimators:
$$
\text{ RASE}=\sqrt{\frac{1}{10} \sum_{i=n+1}^{n+10} \big|\hat{m}(x_i)-m(x_i)\big|^2 },
$$
where $\hat{m}(x_i)$ is the result of each estimator.

We ran our simulations and data analysis on a computer having $96$GB of ram, two Intel Xeon X5570 CPUs, each with four cores running at $2.93$GHz. No parallel computation was implemented.

\subsection{Simulated data: regression on the Klein bottle}\label{simulation2}
Consider the 2-dimensional closed and smooth manifold, the Klein bottle, embedded in $\RR^4$, which is parametrized by $\phi_{\text{Klein}}:[0,2\pi)\times[0,2\pi)\to\RR^4$ so that
$$
(u,v)\stackrel{\phi_{\text{Klein}}}{\mapsto}\big((2\cos v+1)\cos u,\,(2\cos v+1)\sin u,\,2\sin v\cos(u/2),\,2\sin v\sin(u/2)\big).
$$
We sampled $n=1500$ or $1000$ points uniformly from $[0,2\pi)\times[0,2\pi)$, denoted as $\{(U_l,V_l)\}_{l=1}^n$, and then obtained the corresponding  $n$ observations $\{X_l\}_{l=1}^n$ on the predictors $X$  by the parametrization $\phi_{\text{Klein}}$. Notice that the uniform sampling design on $[0,2\pi)\times[0,2\pi)$ corresponds to a non-uniform sampling design on the Klein bottle. 
To generate the responses $\{Y_l\}_{l=1}^n$ corresponding to $\{X_l\}_{l=1}^n$, note that the mapping $\phi_{\text{Klein}}$ is 1-1 and onto, so any $(u,v)$ in $[0,2\pi)\times[0,2\pi)$ can be written as
$(u,v)=\phi_{\text{Klein}}^{-1}(x)$ for some $x$ in the embedded Klein bottle.  So, consider the following regression model on the Klein bottle:
$$
Y:= m(X)+\sigma(X)\,\epsilon,
$$
where  
\begin{eqnarray}
&&m(X) := 7\sin(4 U) + 5\cos(2 V)^2 + 6\exp\{-32( (U-\pi)^2+(V-\pi)^2 )\},\nonumber\\
&&\sigma(X) := \sigma_0(1+0.1\cos(U)+0.1\sin(V)),\nonumber
\end{eqnarray} 
$\epsilon\sim\mathcal{N}(0,1)$ is independent of $X$, and
$\sigma_0$ is the noise level (in $Y$) which  determines the signal-to-noise ratio
\begin{equation*}
\text{snrdb}:=10\log_{10}\Big(\frac{\operatorname{Var}{Y}}{\sigma_0^2}\Big).
\end{equation*}
Furthermore, let
\[
W=X+\sigma_X\eta,
\] 
where $\sigma_X\geq 0$, and $\eta$ is a bivariate normal random vector with zero mean and identity covariance matrix,  independent of $X$ and $\epsilon$. 
Consider estimating $m(X)$ based on observations on $(W,Y)$. In this case, $W=X$ and $X$ is observed without error when $\sigma_X=0$, and $W$ is $X$ contaminated with error when $\sigma_X>0$. 
In the simulations, we took $\sigma_X=0$ or $0.2$ and $\text{snrdb}=5$ or $2$. 
For each simulated sample, we drew $n$ observations $\{(W_i,Y_i)\}_{i=1}^n$ to form the training dataset. Then, independent of the training sample, we sampled randomly $10$ points  $\{W_i\}_{i=n+1}^{n+10}$ as the regressants and tried to estimate the values of $m$ at $\{X_{n+j}\}_{j=1}^{10}$ based on $\{(W_i,Y_i)\}_{i=1}^{n}$.

We evaluated the performance of each estimator by computing the average and standard deviation of its  RASE's over $200$ realizations. The estimated dimension by the MLE intrinsic dimension estimator was $2$ for all of the 200 realizatioins, as is expected. The results of all the estimators and their computation time 
are listed in Table \ref{table:Klein} and Table \ref{table:Klein:err}, from which we can draw the following conclusions. When there is no error-in-variable, i.e. $\sigma_X=0$, MALLER outperforms the four competitors in all of the cases, with significantly smaller  RASE average and similar  RASE standard deviation. 
Also, the MALLER performs well when there exists error in the predictors.
The fact that the computation time for MALLER is longer than that for the other four estimators can be explained as follows. Besides the sample size $n$, the computation 
time for the estimators in \cite{aswani_bickel:2011} also depend on the ambient space dimension $p$ which is $4$ in this example. On the other hand, in addition to $n$, the computation time for MALLER also depends on the estimated intrinsic dimension $d$ which is 2 in this example. This fundamental difference between MALLER and those in
\cite{aswani_bickel:2011} will become apparent when $p$ increases and 
$p\gg d$, as in the Isomap face example discussed in Section \ref{isomapface}. 
\scriptsize
\begin{table}[h] 
\begin{center}
{\small
\begin{tabular}{| c | c | c | c | c |  }
\hline
\multirow{3}{*}{} & \multicolumn{4}{|c|}{Klein bottle, $\sigma_X=0$,  RASE. } \\ 
\cline{2-5}
 & \multicolumn{2}{|c|}{$n=1500$} & \multicolumn{2}{|c|}{$n=1000$}\\
\cline{2-5}
  & $\text{snrdb}=5$ & $\text{snrdb}=2$ & $\text{snrdb}=5$ & $\text{snrdb}=2$  \\
\hline
MALLER & $1.8675 \pm 0.5222$ & $2.3818 \pm 0.666$ & $2.3255 \pm 0.5999$ & $2.7454 \pm 0.9151$\\
NEDE & $2.552 \pm 0.5581$ & $2.9382 \pm 0.631$ & $3.4209 \pm 0.6535$ & $3.6469 \pm 0.6793$ \\
NALEDE & $2.5519 \pm 0.5581$ & $2.9417 \pm 0.6331$ & $3.4288 \pm 0.6522$ & $3.6523 \pm 0.6798$ \\
NEDEP & $2.5514 \pm 0.558$ & $2.9371 \pm 0.6313$ & $3.4212 \pm 0.6534$ & $3.6469 \pm 0.6787$ \\
NALEDEP & $2.5511 \pm 0.5583$ & $2.9406 \pm 0.6335$ & $3.429 \pm 0.6524$ & $3.6528 \pm 0.6791$ \\
\hline
\multirow{3}{*}{} & \multicolumn{4}{|c|}{Klein bottle, the computation time. } \\ 
\hline
MALLER & $76.9222 \pm 29.0305$ & $68.114 \pm 22.3079$ & $32.9121 \pm 10.191$ & $32.7163 \pm 11.3034$ \\
NEDE & $6.0438 \pm 0.1573$ & $6.0416 \pm 0.1709$ & $5.569 \pm 0.1514$ & $5.5878 \pm 0.152$ \\
NALEDE & $11.6054 \pm 0.289$ & $11.5148 \pm 0.2853$ & $10.5719 \pm 0.266$ & $10.5617 \pm 0.265$ \\
NEDEP & $11.4768 \pm 0.2978$ & $11.4656 \pm 0.3199$ & $10.5246 \pm 0.2875$ & $10.5576 \pm 0.2896$ \\
NALEDEP & $17.1086 \pm 0.4276$ & $17.0057 \pm 0.4317$ & $15.5967 \pm 0.4015$ & $15.601 \pm 0.4025$ \\
\hline
\end{tabular}
}
\end{center}
\vspace{0pt}
\caption{\small Regression on the Klein bottle without error in the predictors. The averages and standard deviations, over $200$ realizations, of RASE and the computation time (in seconds) for different estimators tested on different configurations.}\label{table:Klein}
\end{table}
\normalsize

\begin{table}[h] 
\begin{center}
\begin{tabular}{| c | c | c | c | c |  }
\hline
\multirow{3}{*}{} & \multicolumn{4}{|c|}{Klein bottle, $\sigma_X=0.2$, RASE. } \\
\cline{2-5}
 & \multicolumn{2}{|c|}{$n=1500$} & \multicolumn{2}{|c|}{$n=1000$}\\
\cline{2-5}
  & $\text{snrdb}=5$ & $\text{snrdb}=2$ & $\text{snrdb}=5$ & $\text{snrdb}=2$  \\
\hline
MALLER & $3.9227 \pm 0.6898$ & $4.02 \pm 0.7214$ & $3.9514 \pm 0.6785$ & $4.0512 \pm 0.6932$\\
NEDE & $3.9754 \pm 0.6508$ & $4.1225 \pm 0.6255$ & $4.1697 \pm 0.6599$ & $4.2845 \pm 0.6483$ \\
NALEDE & $3.9759 \pm 0.6509$ & $4.131 \pm 0.6252$ & $4.1702 \pm 0.6612$ & $4.2848 \pm 0.6494$ \\
NEDEP & $3.9759 \pm 0.652$ & $4.122 \pm 0.6264$ & $4.1708 \pm 0.6601$ & $4.2848 \pm 0.6479$ \\
NALEDEP & $3.9767 \pm 0.6518$ & $4.1227 \pm 0.626$ & $4.171 \pm 0.6619$ & $4.2851 \pm 0.6492$ \\
\hline
\end{tabular}
\end{center}
\caption{\small Regression on the Klein bottle with error in the predictors. The averages and standard deviations over $200$ realizations of RASE for different estimators tested on different configurations.}\label{table:Klein:err}
\end{table}

\subsection{Real data: Isomap face data}\label{isomapface}
We further tested our algorithm on the Isomap face dataset \cite{Tenenbaum_deSilva_Langford:2000}\footnote{\url{http://isomap.stanford.edu/datasets.html}}. The dataset consists of $698$ $64\times 64$ images, denoted as $\{I^{64}_l\}_{l=1}^{698}$, parametrized by three variables: the horizontal orientation, the vertical orientation, and the illumination direction. Thus, the data were sampled from a 3-dimensional manifold embedded in $\RR^{64\times 64}$. 
When we view each image as a point in $\RR^{64\times 64}$, the ambient space dimension $p=64\times 64$ is large, so in \cite{aswani_bickel:2011} the authors suggested to rescale the images from $64\times 64$ to $7\times7$ pixels in size. Denote the resized images of size $k\times k$ as $\{I^{k}_l\}_{l=1}^{698}$, where  $k=1,\ldots,64$. We performed $200$ replications of the following experiment, which is suggested in \cite{aswani_bickel:2011}. Fix $k=7$. We randomly split $\{I^{7}_l\}_{l=1}^{698}$ into a training set consisting of $688$ images and a testing set consisting of $10$ images. The horizontal orientation of the images in the testing set were then estimated based on the training set. Table \ref{table:isomap}, which summaries the results, shows that MALLER improves on the existing methods substantially in the sense of reduced RASE average and standard deviation.  We mention that NEDEP and NALEDEP behave worse than NEDE and NALEDE due to the frequent occurrence of blowup in the iteration, and the reported results are the best ones among several trials we carried out. 

\begin{table}[h] 
\begin{center}
{\small
\begin{tabular}{| c | c | c | c | c |}
\hline
 & \multicolumn{2}{|c|}{Isomap face database, $k=7$}\\
\cline{2-3}
  &  RASE & computation time  \\
\hline
MALLER & $1.2168 \pm 0.8131$ & $131.5847 \pm 17.5136$ \\
NEDE & $1.7852 \pm 1.2122$ & $34.4606 \pm 4.5847$  \\
NALEDE & $1.7759 \pm 1.1995$ & $170.7088 \pm 28.8193$ \\
NEDEP & $1.8685 \pm 1.2413$ & $53.7212 \pm 8.3594$  \\
NALEDEP & $2.8095 \pm 3.6525$ & $187.3745 \pm 31.2623$  \\
\hline
\end{tabular}
}
\end{center}
\caption{\small The averages and standard deviations, over $200$ replications, of  RASE and computation time in seconds for different estimators tested on the resized Isomap face data $\{I^{7}_l\}_{l=1}^{698}$.}\label{table:isomap}
\end{table}

Next, we carried out another $200$ replications of the same experiment but with $k=14, 21$, or $28$. The MLE intrinsic dimension estimate was $3$ in all the replications when $k=7, 14$ or $21$, and was $4$ all the time when $k=28$. The results are given in Table \ref{table:isomap:2}. We mention that when $k=14, 21$ or $28$, it took long time to compute the methods in \cite{aswani_bickel:2011} and the experiment cannot be finished within a reasonable time frame, so we decided not to include them in the comparison. When $k=7,8,\ldots,16$, the estimated time (average over $3$ realizations) to finish one replication for the methods in \cite{aswani_bickel:2011} are plotted in Figure \ref{fig:isomap:time}, which shows clearly the dependence of these methods on the ambient space dimension $k\times k$.

\begin{table}[h] 
\begin{center}
{\small
\begin{tabular}{|c | c | c | c |  }
\hline
  & $k=14$ & $k=21$ & $k=28$ \\
\cline{2-4}
 & \multicolumn{3}{|c|}{Isomap face database,  RASE} \\
\hline
MALLER & $0.9865\pm 0.5473$ & $1.0259\pm 0.5098$ &  $0.9369\pm 0.7403$\\
\hline
 & \multicolumn{3}{|c|}{Isomap face database, computation time}\\
\hline
MALLER & $108.3796\pm 12.0145$ & $148.9841\pm 20.0436$ & $164.3576\pm 28.8329$ \\
\hline
\end{tabular}
}
\end{center}
\caption{\small The averages and standard deviations over $200$ replications of  RASE and computation time in seconds for MALLER tested on the resized Isomap face data $\{I^{k}_l\}_{l=1}^{698}$, $k=14,21,28$. }\label{table:isomap:2}
\end{table}

\begin{figure}[h]	
\begin{center}
\subfigure{
\includegraphics[width=0.5\textwidth]{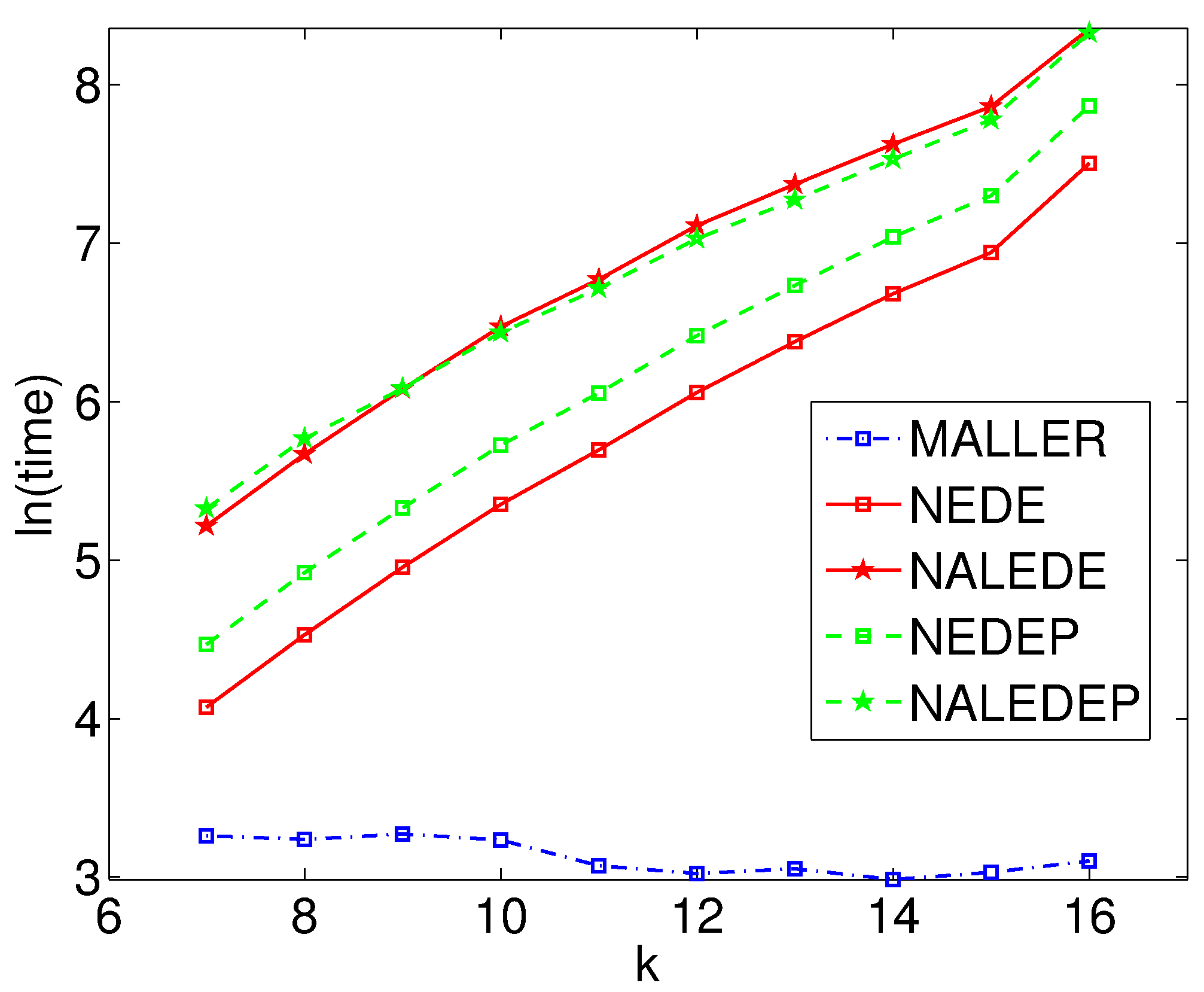}
}
\end{center}
\vspace{-20pt}
\caption{\small The running time for MALLER, NEDE, NALEDE, NEDEP and NALEDEP when $k=7,8,\ldots, 16$. The $y$-axis is in the natural log scale.}\label{fig:isomap:time}
\end{figure}

Note, from Table \ref{table:isomap} and Table \ref{table:isomap:2}, that when $k$ changes from $14$ to $7$ the RASE average of MALLER increases noticeably, and it decreases when $k$ changes from $21$ to $28$. In the following are some partial explanations for these. 
It is clear that resizing the images from $64\times 64$ pixels to $k\times k$ pixels  for a smaller value of $k$ causes a reduction of the resolution of the images. Taking $k=1$, the extremal case, as an example, the images $\{I^1_l\}_{l=1}^{698}$ are scalar values distributed in $\RR$, and obviously the topological structures of $\{I^1_l\}_{l=1}^{698}$ are totally different from that of the original images. 
This fact indicates that over-resizing the images leads to the distortion of the topology, which 
partially explains the increase of the RASE of MALLER when $k$ changes from $14$ to $7$.
Further, the fact that the RASE average dropped again when $k$ changes from $21$ to $28$ may be explained by the reason that, as the estimated intrinsic dimension increased from $3$ to $4$, the extra dimension helps to reduce the estimation error introduced by the complex geometric structure when the resolution is high. 
We emphasize that the above explanations for the RASE average fluctuation need to be quantified with further analysis, which is out of the scope of this paper and will be reported in a future work.

In conclusion, the Isomap face database example shows the strength of MALLER: once the number of observations $n$ is large enough compared with the intrinsic dimension $d$ of the manifold, which may be small compared with the dimension $p$ of the ambient space, our method provides improvement over existing estimators from both the viewpoints of the prediction error and computation time. 

\subsection{Gradient and Covariant Derivative Estimation}
{
We tested our estimator $\widehat{\iota_*\mbox{\tt{grad}}m}(x)$, given in (\ref{algo:estimator:CovDeri}), on the $2$-dimensional torus $\mathbb{T}$ embedded in $\RR^3$ via $\iota$, which is parametrized by, except for a set of measure zero,
\begin{equation}
\phi:(u,v)\mapsto \left( (2+\cos(v))\cos(u), (2+\cos(v))\sin(u), \sin(v) \right),
\end{equation}
where $(u,v)\in I:=(0,2\pi)\times (0,2\pi)$. Considered model (\ref{model1}), where $X=\phi(U,V)$, the regression function $m:\mathbb{T}\to \RR$ is given by $$m(\phi(u,v))=\cos(u)\sin(4v+1),$$ 
$\epsilon\sim \mathcal{N}(0,1)$ and $\sigma(\iota^{-1}(X)) = \sigma_0(1+0.1\cos(U)+0.1\sin(V))$ with $\sigma_0$ chosen so that snrdb$=5$ or $40$.   A direct calculation leads to
\begin{equation}\label{simulation:torus:gradient}
\iota_*\mbox{\tt{grad}}m(\phi(u,v))=\left(
\begin{array}{c}
\sin^2(u)\sin(4v+1)-4\cos(u)^2\sin(v)\cos(4v+1)\\ 
-\sin(u)\cos(u)\sin(4v+1)-4\sin(u)\cos(u)\sin(v)\cos(4v+1)\\
4\cos(u)\cos(v)\sin(4v+1)
\end{array}
\right).
\end{equation}
The detailed calculation of (\ref{simulation:torus:gradient}) can be found in the Supplementary. 

We sampled $6000$ points $\{(U_i,V_i)\}_{i=1}^{6000}$ uniformly from $I$ and then generate $\{(X_i,Y_i)\}_{i=1}^{6000}$ according to the above model. Notice that this sampling scheme is non-uniform on the torus. Then we randomly picked $3000$ points $\{X_i=\phi(U_i,V_i)\}_{i=6001}^{9000}$ as the testing sample, and compute the gradient estimates $\{\widehat{\iota_*\mbox{\tt{grad}} m}(X_i)\}_{i=6001}^{9000}$ based on the training sample $\{(X_i,Y_i)\}_{i=1}^{6000}$. The estimates are visually demonstrated in Figure \ref{fig:gradient}, together with the ground truth (\ref{simulation:torus:gradient}) for comparison.  

\begin{figure}[h]	
\begin{center}
\subfigure{
\includegraphics[width=0.46\textwidth]{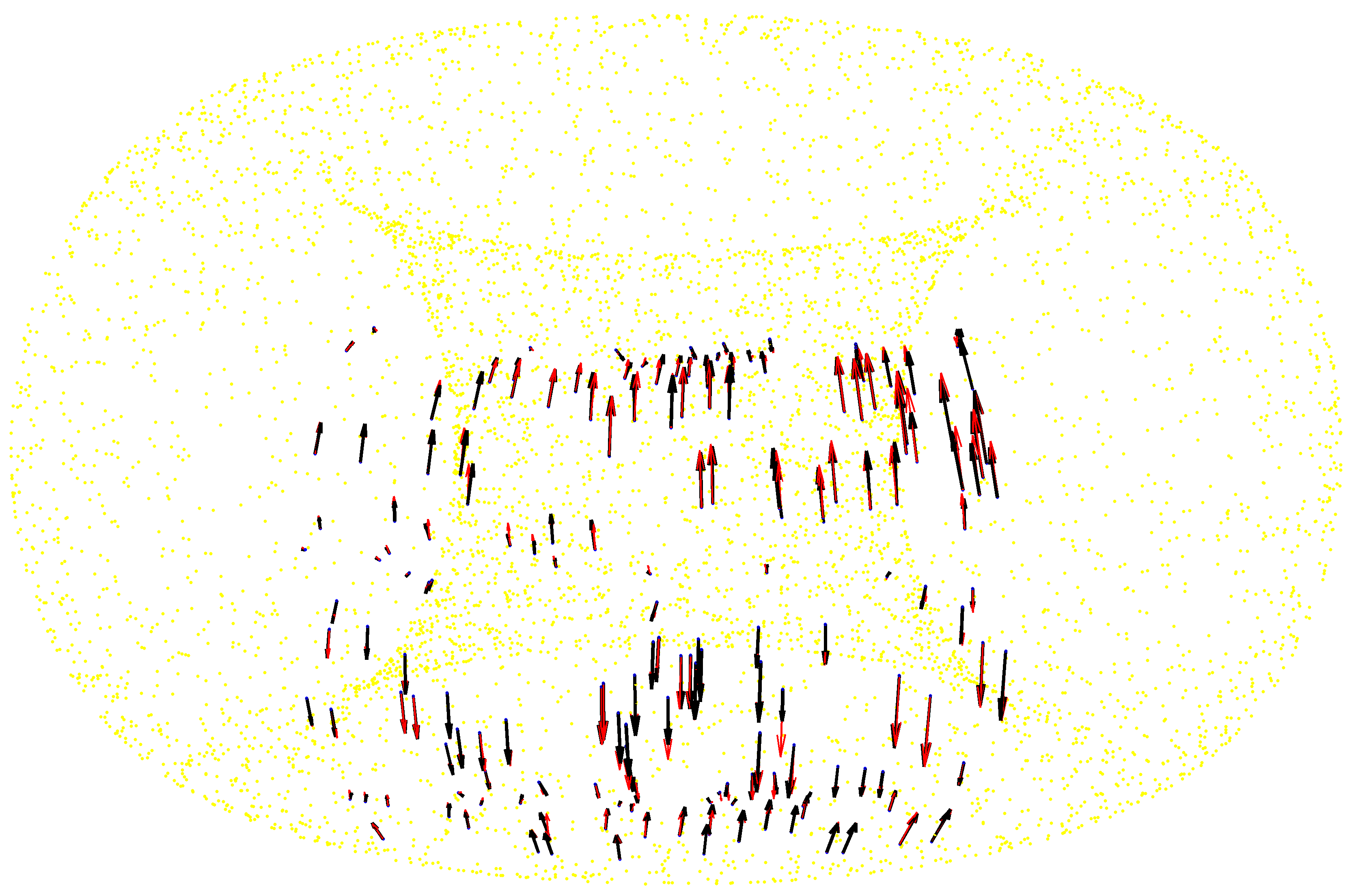}
}
\subfigure{
\includegraphics[width=0.46\textwidth]{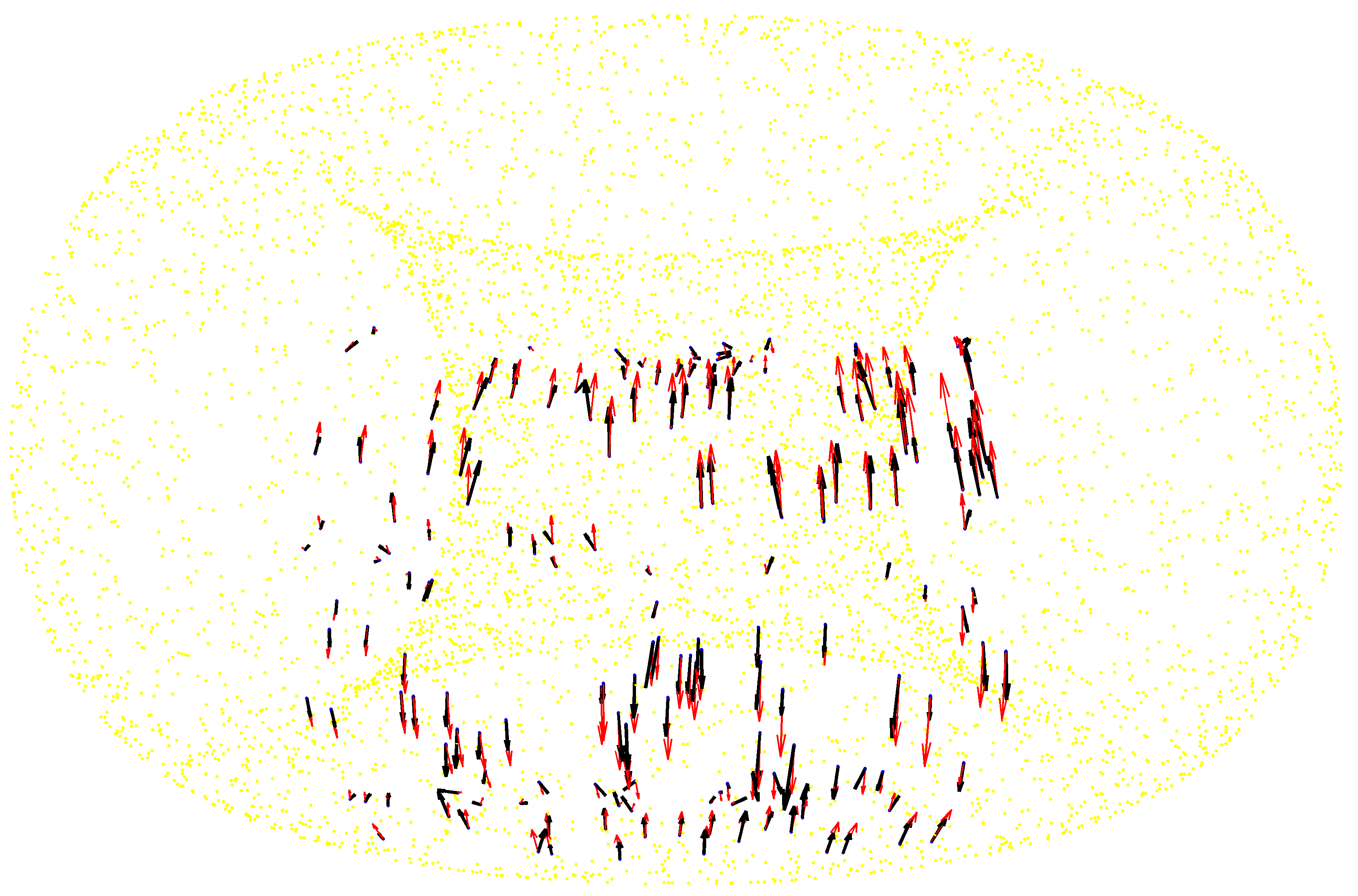}
}
\end{center}
\caption{\small Gradient estimates. Left: snrdb=$40$dB; Right: snrdb=$5$dB. The blue circles are the portion of the testingsample $\{(u_i,v_i)\}_{i=6001}^{9000}$ such that $|v_i|<1$ and $u_i>2$, the red arrows are $\iota_*\mbox{\tt{grad}}m(\phi(u_i,v_i))$ and the black arrows are $\widehat{\iota_*\mbox{\tt{grad}} m}(\phi(u_i,v_i))$.}\label{fig:gradient}
\end{figure}

\section{Implications to Manifold Learning}\label{diffusionmap}
Another branch of approaches to high-dimensional, massive data analysis are the graph based algorithms such as locally linear embedding (LLE) \cite{Roweis_Saul:2000}, ISOMAP \cite{Tenenbaum_deSilva_Langford:2000}, Hessian LLE \cite{donoho_grimes:2003}, the Laplacian eigenmap \cite{belkin_niyogi:2003}, local tangent space alignment \cite{Zhang_Zha:2004}, diffusion maps \cite{coifman_lafon:2006}, and vector diffusion maps \cite{vdm}. In addition to preserving the nonlinearity of the data structure, one advantage of these approaches is their adaptivity to the data, that is, the model imposed on the data is relatively weakened so that the information revealed from the analysis is less distorted by model mis-specification. These advantages render the graph based algorithms attractive and popular in data analysis. 
When the data are assumed to be sampled from a compact and smooth $d$-dimensional manifold $\MM$, the key step of these methods is the learning of the intrinsic geometric quantities, for example, the Hessian operator \cite{donoho_grimes:2003}, the Laplace-Beltrami operator \cite{belkin_niyogi:2003,coifman_lafon:2006} or the connection Laplacian \cite{vdm}. What we are concerned with in this section is the estimation of the Laplace-Beltrami operator $\Delta$ of $\MM$, considered in the diffusion map framework \cite{coifman_lafon:2006}, via MALLER. 
We refer the readers to these literature for further discussions and references. 
Throughout this section, we make use of the same assumptions and notation as in Sections 
\ref{algorithm} and \ref{theory}. 

We start with discussing the relationship between the diffusion map framework and generalizing the Nadaraya-Watson kernel regression method to the manifold setup. 
Suppose $\MM$ is compact, smooth and without boundary. Fix a bandwidth $h>0$. First we define a $n\times n$ weight matrix $W$ and a $n\times n$ diagonal matrix $D$ by
\begin{equation}\label{W0}
W(i,j)=K\left(\frac{\|X_i-X_j\|_{\RR^p}}{\sqrt{h}}\right)\quad\mbox{and}\quad
D(i,i) = \sum_{j=1}^{n} W(i,j).
\end{equation}
Then $A:=D^{-1}W$ can be interpreted as a Markov transition matrix of a discrete random walk over the sample points $\{X_i\}_{i=1}^n$, where the transition probability in a single step from the sample point $X_i$ to the sample point $X_j$ is given by $A(i,j)$. 

Note that $A$ can be used to generalize the Nadaraya-Watson kernel method originally defined for nonparametric regression on $\RR^p$ to the manifold $\MM$ setup. Indeed, given the regression model (\ref{model1}), define this generalized Nadaraya-Watson estimator $\hat{m}_{NW}$ of $m$ at $X_i$ as
\[
\hat{m}_{NW}(X_i,h) := (A\vY)(i)=\frac{\sum_{j=1}^nK\left(\frac{\|X_i-X_j\|_{\RR^p}}{\sqrt{h}}\right)Y_j}{\sum_{j=1}^nK\left(\frac{\|X_i-X_j\|_{\RR^p}}{\sqrt{h}}\right)}, \, i=1,\ldots,n,
\]
i.e. take $A$ as the smoothing matrix of $\hat{m}_{NW}(\cdot, h)$. 
Clearly the conditional expectation of the estimator $\hat{m}_{NW}(X_i,h)$ becomes
\begin{eqnarray}
&&\EE\big\{\hat{m}_{NW}(X_i,h)\big|\mathcal{X}\big\}
= (A\vm)(i)=\frac{\sum_{j=1}^nK\left(\frac{\|X_i-X_j\|_{\RR^p}}{\sqrt{h}}\right)m({X}_j)}{\sum_{j=1}^nK\left(\frac{\|X_i-X_j\|_{\RR^p}}{\sqrt{h}}\right)},\label{NWconvergence}
\end{eqnarray}
where $\vm$ is defined in (\ref{def:Yandm}). When $m\in C^3(\MM)$ and ${X}_i\notin \MM_{\sqrt{h}}$, the asymptotic expansion of (\ref{NWconvergence}) has been shown in \cite{coifman_lafon:2006,hein_audibert_luxburg:2005, singer:2006}. Indeed, we have, as $n\to \infty$,
\begin{eqnarray}
(A\vm)(i)= m({X}_i)+h\frac{\mu_{1,2}}{2d}\bigg(\Delta m({X}_i)+2 \frac{m({X}_i)\Delta f({X}_i)}{f({X}_i)}\bigg)+ O(h^2)+O_p\Big(\frac{1}{n^{\frac{1}{2}}h^{\frac{d}{4}-\frac{1}{2}}}\Big).\nonumber
\end{eqnarray}
Note that in \cite{coifman_lafon:2006} the kernel is normalized so that $\mu_{1,0}=1$ and $\mu_{1,2}/d=2$.
When $f$ is constant, the second order conditional bias term contains information about the Laplace-Beltrami operator of $(\MM,g)$. This fact, however, is in general ignored when the focus is the nonparametric regression problem. On the contrary, since knowledge of the Laplace-Beltrami operator leads to abundant information about the manifold, in \cite{coifman_lafon:2006} the matrix $L_0:=h^{-1}(D^{-1}W-I_n)$ and its relationship with the Laplace-Beltrami operator are extensively studied, and the eigenvectors of $A$ are used to define the diffusion map. 
When $f$ is not constant, 
the $f$-dependence is removed by the following normalization \cite{coifman_lafon:2006}. Define a $n\times n$ weight matrix $W_1$ and a $n\times n$ diagonal matrix $D_1$ by
\begin{equation}
\label{Walpha}
W_1 = D^{-1}WD^{-1},\quad\mbox{and}\quad D_1(i,i) = \sum_{j=1}^n W_1(i,j)
\end{equation}
where $W$ and $D$ are defined in (\ref{W0}),
and
\[
L_1=h^{-1}\big(D_1^{-1}W_1-I_n\big).
\]
When $n\to\infty$, it is shown in \cite{coifman_lafon:2006} that for any $m\in C^3(\MM)$ the matrix $L_1$ satisfies the following convergence:
\begin{equation}\label{remark:dm:L1m}
(L_{1}\vm)(i) = \frac{\mu_{1,2}}{2d}\Delta m(X_i) + O(h) + O_p\Big(\frac{1}{n^{1/2}h^{d/4+1/2}}\Big).
\end{equation}
Notice that the effect of the normalization (\ref{Walpha}) is actually to cancel out the effect of the non-uniformality in $f$ on the matrix $L_0$. We remark that the matrix $D_1^{-1}W_1$ can thus be used as the smoothing matrix of a new estimator of $m$  which is design adaptive.  

If we view the Nadaraya-Watson kernel method on $\RR^p$ as the local zero-order polynomial regression, the LLR on $\RR^p$ can be viewed as the first-order companion of the Nadaraya-Watson kernel method which takes the local slope into account \cite{ruppert_wand:1994}. We discuss extensively  its generalization to the regression on manifold setup in Section \ref{algorithm}, its large sample behaviors in Section \ref{theory}, and its numerical results are demonstrated in Section \ref{numerics}. Recall that the conditional bias of MALLER, given in (\ref{rslt:interior:mfd}) of the Supplementary, depends on the Laplace-Beltrami operator:
\[
\EE\{\hat{m}(X,h)-m(X)| \mathcal{X}\}=h\frac{\mu_{1,2}}{2d}\Delta m(X)+O(h^2+h\hpca^{3/4})+O_p\Big(\frac{1}{n^{1/2}h^{d/4-1}}\Big).
\]
This fact leads us to build up an alternative matrix to approximate the Laplace-Beltrami operator.
Fix $h>0$ and consider the following $n\times n$ matrix
\begin{equation}\label{section5:Ap}
A_p=\left[\begin{array}{c}
\vv_1^T(\XX^T_{X_1}\WW_{X_1}\XX_{X_1})^{-1}\XX_{X_1}^T\WW_{X_1}\\
\vdots\\
\vv_1^T(\XX^T_{X_n}\WW_{X_n}\XX_{X_n})^{-1}\XX_{X_n}^T\WW_{X_n}
\end{array}
\right],
\end{equation}
where the $i$-th entry is defined by (\ref{design}), (\ref{weighted}), and (\ref{algo:estimator:mhat}). Note that $A_p$ is the smoothing matrix of MALLER, that is, 
$A_p\vY=\big(\hat{m}(X_1,h),\ldots,\hat{m}(X_n,h)\big)^T$ from  (\ref{algo:estimator:mhat}).
Using this smoothing matrix and defining 
\[
L_p = h^{-1}\big(A_p-I_{n}\big),
\] 
for any $m\in C^3(\MM)$, we directly have
\begin{equation}\label{section5:newDelta}
(L_p\vm)(i)=\frac{\mu_{1,2}}{2d}\Delta m(X_i)+O(h+\hpca^{3/4})+O_p\Big(\frac{1}{n^{1/2}h^{d/4}}\Big).
\end{equation}
Thus the matrix $L_p$ can be used to construct an estimator of the Laplace-Beltrami operator $\Delta$. Notice that we do not need an extra step to handle the non-constant p.d.f. issue here because the design adaptive property of $\hat{m}(X,h)$ ensures that the leading term in the right-hand side of (\ref{section5:newDelta}) is independent of $f$. With the estimator $L_p$ of $\Delta$, massive data analysis can be carried out in the same way as those in the diffusion map framework if the manifold assumption is reasonable. We remark that the knowledge of the non-constant p.d.f. is useful in some problems. For example, in \cite{coifman_lafon:2006,nadler_lafon_coifman:2006} the authors showed a strong connection between the non-constant p.d.f. with the Fokker-Plank operator, which is useful in the low-dimensional representation of stochastic systems.   

In Figure \ref{fig:Sk:ev}, some numerical results of estimating the $\Delta$ of $\MM$ by this new method are demonstrated. We sampled $1000$, $2000$ and $4000$ points uniformly from the $S^2$, $S^3$ and $S^4$ embedded in $\RR^{3}$, $\RR^{4}$ and $\RR^{5}$ respectively, and built the matrix $L_p$ from the sample points with $h=0.1$. It is a well known fact that the $l$-th eigenvalue of the Laplace-Beltrami operator of $S^k$ is $-l(l+k-1)$ with multiplicity ${k+l \choose k}-{k+l-2 \choose k}$, where ${ \cdot \choose \cdot}$ is the binomial coefficient. The results in Figure \ref{fig:Sk:ev} show that the new estimator for the Laplace-Beltrami operator agrees with this well known fact numerically.
\begin{figure}[h]	
\begin{center}
\subfigure{
\includegraphics[width=0.31\textwidth]{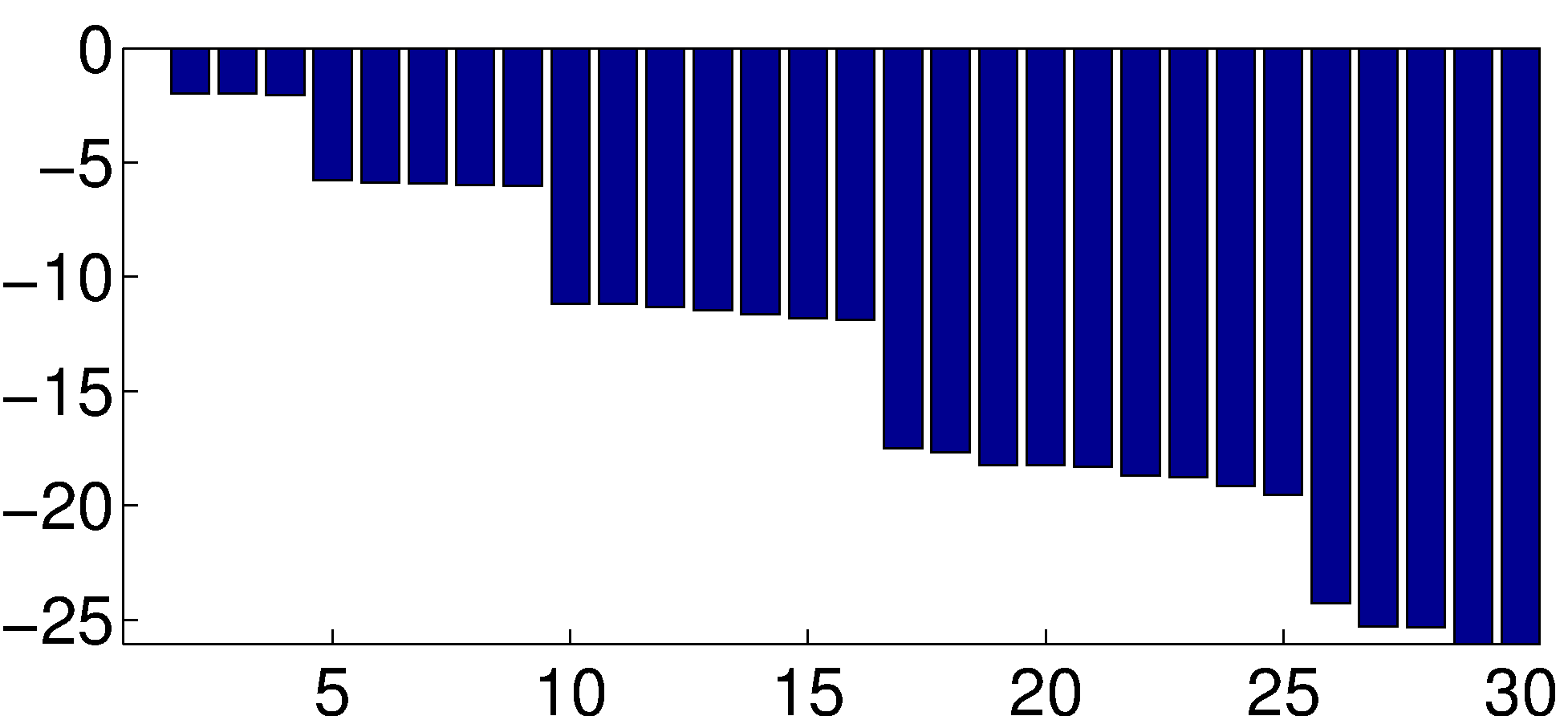}}
\subfigure{
\includegraphics[width=0.31\textwidth]{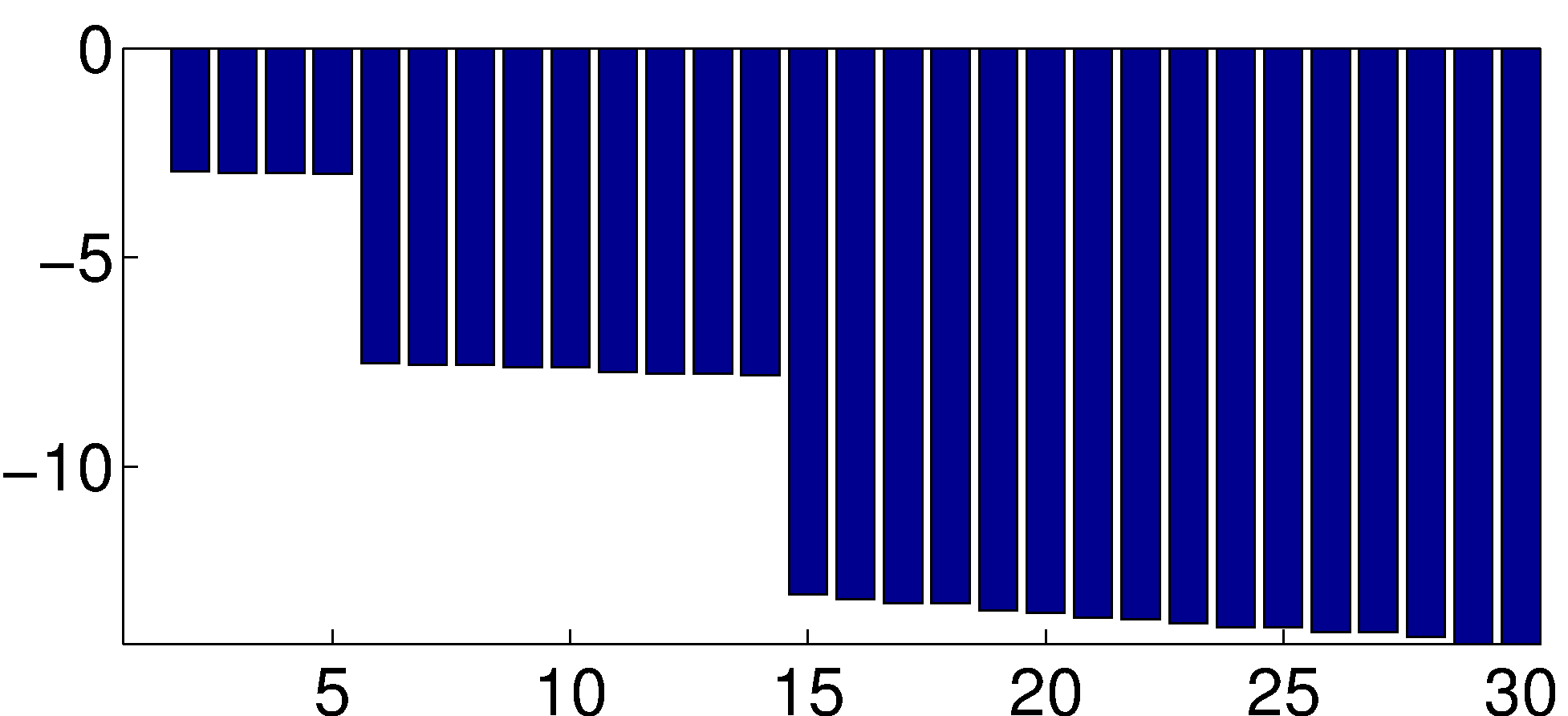}}
\subfigure{
\includegraphics[width=0.31\textwidth]{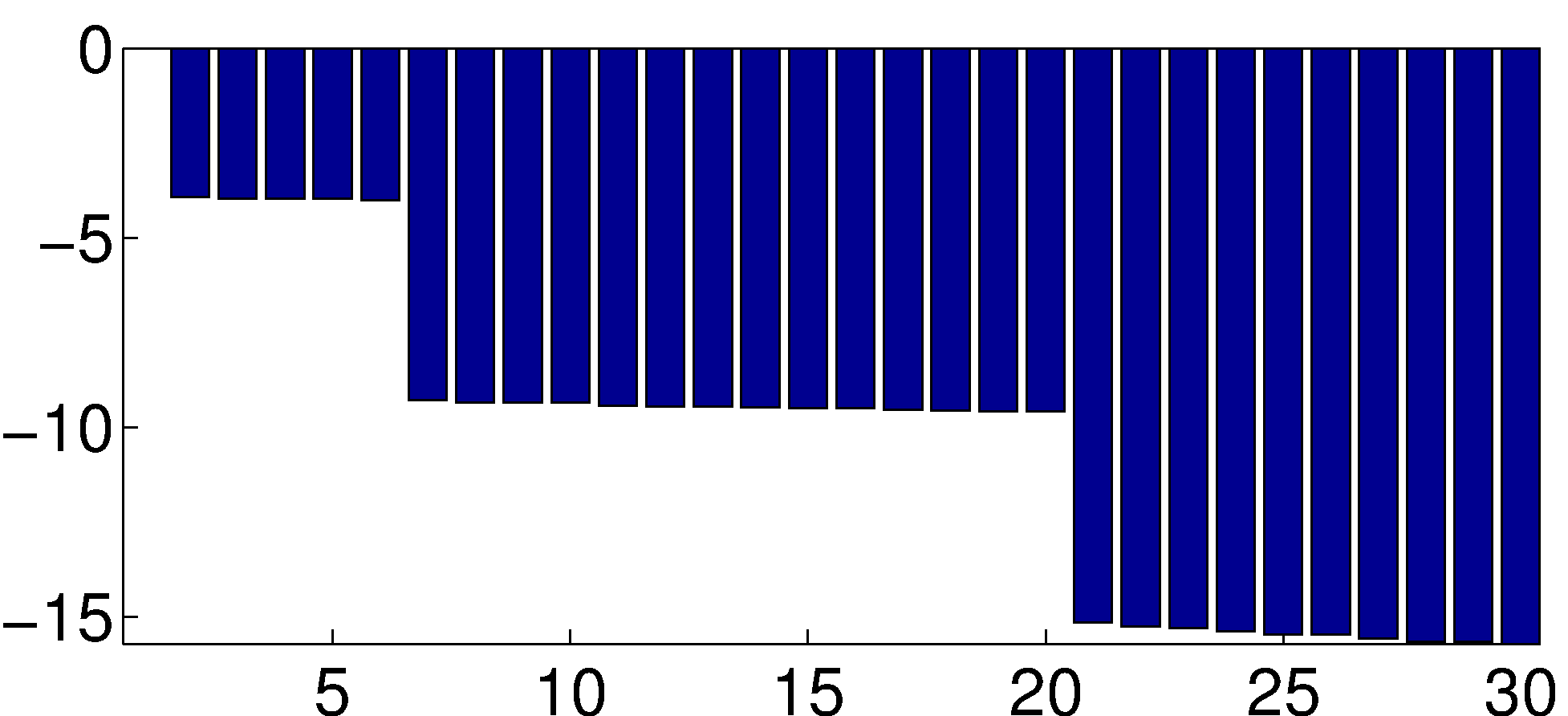}}
\end{center}
\vspace{-0.6cm}
\caption{\small From left to right: bar plots of the first $30$ eigenvalues of $L_p$ when the data points were sampled uniformly from $S^2$, $S^3$ and $S^4$. Note that the first few eigenvalues of $\Delta$ are $0,-2,-6,-12$ for $S^2$, $0,-3,-8,-14$ for $S^3$ and $0,-4,-10,-18$ for $S^4$, and the multiplicities of the first few eigenvalues of $\Delta$ are $1,3,5,7$ for $S^2$, $1,4,9,16$ for $S^3$ and $1,5,14,30$ for $S^4$. This fact is well resembled by the corresponding spectrum of $L_p$.}\label{fig:Sk:ev}
\vspace{-0.1cm}
\end{figure}

Up to now there are two ways to estimate the Laplace-Beltrami operator: one is based on generalizing the Nadaraya-Watson kernel method to the manifold setup as suggested by (\ref{remark:dm:L1m}) and studied in \cite{coifman_lafon:2006}, and the other is based on MALLER, which generalizes the LLR to the manifold setup, as suggested by (\ref{section5:newDelta}). The difference between these two approaches is most obvious when the manifold has smooth boundary. 

Suppose $\MM$ is compact, smooth and its boundary $\partial\MM$ is non-empty and smooth. When $X_i\in \MM_{\sqrt{h}}$, the asymptotic behavior of $D_1^{-1}W_1$ has been shown in the proof of Proposition 10 of \cite{coifman_lafon:2006}:
\begin{equation}\label{dm:bdry:blowup}
(D^{-1}_1W_1\vm)(i)=m(X_0)+\sqrt{h}C_1\partial_\nu m(X_0)+O(h)+O_p\Big(\frac{1}{n^{1/2}h^{d/4-1/2}}\Big),
\end{equation}
where $C_1=O(1)$, $X_0\in\partial\MM$ is the point on the boundary $\partial\MM$ closest to $X_i$, and $\nu$ is the normal direction at $X_0$. If the $\sqrt{h}$-order term is non-zero, the estimator $(L_{1}\vm)(i)$ in (\ref{remark:dm:L1m}) blows up when $h\to 0$. To avoid this blowup and to get an estimate of the Laplace-Beltrami operator on $\MM$, the Neuman's boundary condition $\frac{\partial m}{\partial\nu}=0$ is necessary. 
Thus, solving the eigenvalue problem of $L_1$ is a discrete approximation to solving the eigenvalue problem of the Laplace-Beltrami operator with the Neuman's boundary condition.

The situation is totally different for the proposed estimator $L_p$. The asymptotic behavior of the conditional bias of MALLER at $X_i\in \MM_{\sqrt{h}}$ provided in Corollary \ref{col:smoothbdry}
leads to
\begin{equation}\label{manifold_learning:new_estimator}
(L_p\vm)(i)=\frac{1}{2}\sum_{k=1}^dc_k(X_i)\nabla^2_{\partial_k,\partial_k}m(X_i)+O_p(h^{-1/2}\hpca^{3/4}+\hpca^{1/2})+O_p\Big(\frac{1}{n^{1/2}h^{d/4}}\Big).
\end{equation}
Thus, we know that when ${X}_i$ is near the boundary, the estimator $L_p$ does not blow up when $h\to 0$, and a different boundary condition can be imposed. 

{
Notice that the importance of using different bandwidths in the tangent plane estimation and in the LLR on the tangent plane becomes clear from (\ref{section5:newDelta}) and (\ref{manifold_learning:new_estimator}). Indeed, if we take $\hpca<h$ then it follows from (\ref{section5:newDelta}) (resp. (\ref{manifold_learning:new_estimator})) that the first order error of the estimator for the Laplace-Beltrami operator inside the manifold is smaller than the order $h^{3/4}$ (resp. $h^{1/4}$).}

In Figure \ref{fig:section5:bdry}, we demonstrate the eigenvectors of the estimator $L_p$ for the Laplace-Beltrami operator of a manifold with boundary. Specifically, we sampled $2000$ points $\{X_l\}_{l=1}^{2000}$ uniformly from the interval $[0,1]$ embedded in $\RR$, and evaluated the eigenvectors of $L_p$ built on $\{X_l\}_{l=1}^{2000}$. Notice that the eigenvectors shown in Figure \ref{fig:section5:bdry} can not happen, except for the first one, if the Laplace-Beltrami operator satisfies the Neuman's condition. The survey of the boundary condition suitable for the estimator $L_p$ is out of the scope of this paper, and we leave it as a future work.
\begin{figure}[h]	
\begin{center}
\subfigure{
\includegraphics[width=0.95\textwidth]{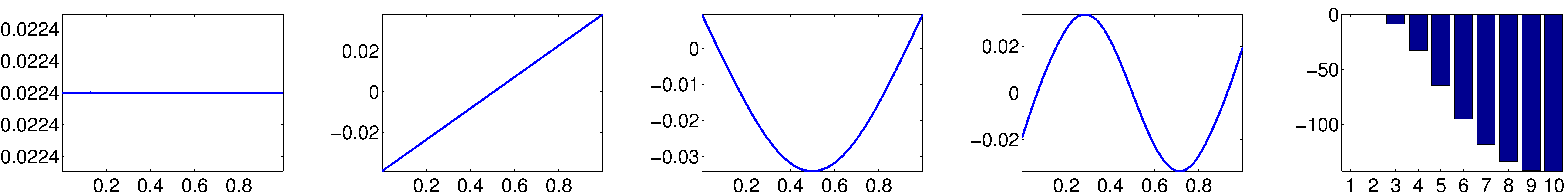}
}
\end{center}
\caption{\small From left to right: the first four eigenvectors of $L_p$ and the first $10$ eigenvalues of $L_p$ when sampling from $[0,1]$. The first two eigenvalues are zero. Notice that the second, third and fourth eigenvectors can not happen if the Laplace-Beltrami operator satisfies the Neuman's condition.}\label{fig:section5:bdry}
\end{figure}
\bigskip

\section{Discussions} \label{discussion}
When the $p$-dimensional predictor vector $X$ has some $d$-dimensional manifold structure, we obtain MALLER by constructing the traditional LLR on the estimated embedded tangent plane, which is of dimension $d$ instead of $p$. Consequently, both the estimation accuracy and computational speed depend only on $d$ but not on $p$. Keeping $p, d, n$ as fixed numbers, this feature is particularly advantageous when $d\ll n< p$, as is shown in the Isomap face database example in the numerical section.  
{
We mention that MALLER works in this case hinges on the capability of estimating the tangent plane. Since our model is noise free in the predictors, this capability can be explained by the theoretical findings in \cite{kaslovsky_meyer:2011} and \cite{nadler:2008}.} 
In \cite{nadler:2008}, the spike model is studied and the recovery of the subspace spanned by the response vectors is guaranteed even if $p\geq n$, when there is no noise \cite[(2.13)]{nadler:2008}. Under the manifold setup, locally the manifold model behaves like the Euclidean space, so it is expected to have similar results as those in \cite{nadler:2008}, which is shown in \cite{kaslovsky_meyer:2011}.
Furthermore, we emphasize that, while in \cite{aswani_bickel:2011} this case is modeled as the large $p$ small $n$ problem, where $p$ grows with $n$, and sparsity conditions and thresholding are employed, here we treat $p$ as a fixed number and take the fact that $n$ is larger than $d$.

\subsection{The Relationship with NEDE}
MALLER is not the first LLR regression scheme proposed to adapt to the manifold structure. NEDE, given in \cite{aswani_bickel:2011}, is a manifold-adaptive LLR constructed in the $p$-dimensional ambient space with regularization imposed on the directions perpendicular to the estimated embedded tangent plane.
At the first glance MALLER seems to be a special case of NEDE \cite[(4.6)]{aswani_bickel:2011} by taking $\lambda_n=\infty$ in \cite[(4.6)]{aswani_bickel:2011}. However, there are several distinct differences between the two methods. In this section we follow the notation used in \cite{aswani_bickel:2011}.

First, when $\lambda_n=\infty$ for all $n$
, although $\tilde{\beta}$ in \cite[(4.6)]{aswani_bickel:2011} is forced to be located on the estimated embedded tangent plane, the NEDE algorithm still runs in the ambient space and the minimization problem in \cite[(4.6)]{aswani_bickel:2011} becomes ill-posed. Indeed,  the solution in \cite[(4.6)]{aswani_bickel:2011} depends on the inverse of the matrix $\hat{C}_n+\lambda_n\hat{P}_n/nh^{d+2}$, which is unstable to solve when $\lambda_n=\infty$. This numerical instability of NEDE when $\lambda_n=\infty$ can also be shown numerically. 
As an illustration, we ran NEDE with $\lambda_n=e^{100}$  (within the machine precision)  on the Isomap face database with the images downsized to $7\times 7$ pixels. Then, it happened that the optimal value of $d$ chosen by the NEDE algorithm was close to $49=7\times 7 =p$ ($48.325\pm 1.3019$ over $100$ replications) due to the degeneracy of $\hat{C}_n+\lambda_n\hat{P}_n/nh^{d+2}$, and the final RASE was $12.3684 \pm 6.1161$ (over $100$ replications), which is roughly ten times of the RASE of MALLER. Even when we set $d=3$ and $\lambda_n=e^{100}$ in the NEDE algorithm and tested it on the same $7\times 7$-pixel images, the final RASE was still $10.5829 \pm 6.0986$ after $100$ replications.

Second, even if NEDE \cite[(4.6)]{aswani_bickel:2011} is stable to solve when $\lambda_n=\infty$, the bandwidth selection problem in NEDE still depends on $p$, which leads to different results compared with MALLER. Specifically, the selected bandwidth would be larger and hence the bias is increased. 

Third, in NEDE the bandwidth used in the tangent plane estimation is taken to be the same as the one used in the LLR estimation, while in MALLER we estimate the tangent plane using a different bandwidth $h_{\text{pca}}$ which by the asymptotic analysis should be taken to be smaller than the bandwidth $h$ in the LLR step. Thus, the tangent plane estimate obtained by NEDE is different from that obtained by MALLER. Since this estimation error does not contribute to the leading bias term, the difference is not significant in the regression problem. However, if we would like to have a better estimator of the Laplace-Beltrami operator, this error becomes significant, as is shown in Section \ref{diffusionmap}.  
 
In conclusion, MALLER is different from NEDE even if the parameter $\lambda_n$ in NEDE is set to $\infty$, both theoretically and numerically. And, the key features that render the two algortihms different are those mentioned above, not the more sophisticated method MALLER uses to  select the bandwidth in the LLR.
%
%
%
%
%
%
%
%
%
%
%
%
%
%
%
%
%
%
%
%
%
%
%
%
%
%
%
%
%
%

\subsection{Future Directions}

To sum up this paper, here are several issues left open and are of interest for future research:
\begin{enumerate}
\item Like in any smoothing methods, bandwidth selection is crucial for the proposed MALLER. 
Our bandwidth selection procedure is built on balancing between estimates of the conditional bias and variance. 
Although this approach worked well in our numerical studies, there is still room for improvement. 
\item 
We include in our algorithm a clustering tool to alleviate numerical problems caused by the condition number, without having to estimate the condition number. This is not the ultimate solution; instead, the ideal solution is to estimate the condition number, and then use that information in the subsequent steps. 
\item In this paper we consider the case where the predictor vector 
is directly observable. In some situations, the predictor vector itself is subject to noise, and 
the tangent plane and regression estimation steps has to be adjusted accordingly. 
This is closely related to the deconvolution and measurement error problems in the literature, in the Euclidean setup. 
\item In MALLER, the dimensionality is reduced to the intrinsic structure of the predictors. The dimensionality may be further reduced by taking into account the relationship between the response and the predictors \cite{xia:2007, xia:2008}.
\item 
The smoothing matrix of MALLER is shown to be useful for estimating the Laplace-Beltrami operator with the boundary condition different from Neuman's condition, it is worthwhile to investigate further such a new set of tools for manifold learning.
\item In applications, the response itself may be multivariate 
as well. The case when the responses are positive-definite matrices and the predictor vector is non-degenrated in $\RR^p$ was considered by \cite{zhu:2009}. It is interesting to investigate the case when both the response and the predictor vector have manifold structures.
\end{enumerate}


\bibliographystyle{plain}
\bibliography{lpoly_mfds_submit}

\newpage

\centerline{\bf {\Large Supplementary Materials for ``Local Linear Regression}}
\centerline{\bf {\Large  on Manifolds  and its Geometric Interpretation''}}

\centerline{\large by Ming-Yen Cheng, and  Hau-Tieng Wu}

\setcounter{equation}{0}
\setcounter{section}{0}
\setcounter{page}{1}
\renewcommand{\theequation}{A.\arabic{equation}}
\renewcommand{\thesection}{A.\arabic{section}}

\section{Exterior derivative, covariant derivative and gradient}\label{appendix:background}

In this appendix we provide the required differential geometry background about the covariant derivative, gradient, exterior derivative and their relationships. We refer the readers to \cite{doCarmo:1992} for more details. 

We start from recalling the definition of the gradient vector field of a given function defined on the Euclidean space. Given $m:\RR^d\rightarrow \RR$, the gradient vector field or the total differentiation, denoted as $\nabla m$ is defined as
\[
\nabla m:=\left(\frac{\partial m}{\partial x_1},\ldots,\frac{\partial m}{\partial x_d}\right)
\]
so that for $v\in\RR^d$ we have the directional derivative
\begin{equation}\label{app:background:Euclidean:1}
\nabla_vm(x):=(\nabla m)(v):=\lim_{t\rightarrow 0}\frac{m(x+tv)-m(x)}{t}.
\end{equation}
Often we use another notation to represent the directional derivative:
\begin{equation}\label{app:background:Euclidean:2}
\langle \nabla m(x), v\rangle:=\nabla_vm(x)
\end{equation}
This definition, however, can not be generalized to the manifold setup directly. Indeed, the quantity $x+tv$ in (\ref{app:background:Euclidean:1}) does not make sense in general. To obtain a suitable notion of differentiation, we consider the following definitions. Fix a differentiable $d$-dim manifold $\MM$ and a $C^1$ function $m:M\to\RR$. For a given differentiable vector field $V$, locally around $x\in\MM$ we can find a curve $c(t)$ so that $c(0)=x\in\MM$ and $c'(0)=V_x$, the value of $V$ at $x$ so that $V$ acts on $m$ at $x$ by
\begin{equation}\label{app:background:M:Xf}
Vm(x):=\frac{\ud m(c(t))}{\ud t}\Big|_{t=0}.
\end{equation}
The exterior derivative of $m$, denoted as $\ud m$ at $x$ is defined as:
\begin{equation}\label{app:background:M:exterior}
((\ud m)V)(x):=\langle (\ud m)_x,V_x\rangle:=Vm(x),
\end{equation}
where $\langle\cdot,\cdot\rangle$ means that the first entry is the dual of the second entry.
We can thus view the exterior derivative of $m$ as a 1-form, which maps a given vector field into a scalar valued function.
Next we define the covariant derivative of $m$, denoted as $\nabla m$. Fixed a $C^1$ curve $c(t)$ on $\MM$ so that $c(0)=x$. The covariant derivative of $m$ in the direction of $c'(0)$ is defined as
\[
\nabla_{c'(0)}m:=\lim_{t\to 0}\frac{P_{c(0),c(t)}m(c(t))-m(c(0))}{t},
\]
where $P_{c(0),c(t)}$ is the parallel transport of the trivial scalar bundle. Since $P_{c(0),c(t)}$ is trivial, the covariant derivative of $m$ in the direction of $c'(0)$ is reduced to
\begin{equation}\label{app:background:M:cov}
\nabla_{c'(0)}m=\lim_{t\to 0}\frac{m(c(t))-m(c(0))}{t}=\frac{\ud m(c(t))}{\ud t}=Vm(x).
\end{equation}
Thus $\nabla m$ 1s a 1-form, which maps a given vector field to a scalar value. 
If $\MM$ is Riemannian, that is, $\MM$ is endowed with a Riemannian metric $g$, we can further define the gradient of $m$, which is a vector field denoted as $\mbox{\tt{grad}}m$, as:
\begin{equation}\label{app:background:M:gradient}
g(\mbox{\tt{grad}}m(x), V_x)=\langle (\ud m)_x,V_x\rangle.
\end{equation} 

It is clear from (\ref{app:background:M:exterior}) and (\ref{app:background:M:cov}) that for a given differentiable function $m$, its exterior derivative and covariant derivative are the same. Notice that from (\ref{app:background:Euclidean:1}) and (\ref{app:background:M:cov}), the covariant derivative of $m$ defined on $\MM$ is a natural generalization of the total derivative of $m$ defined on the Euclidean space. In other words, the total derivative of $m$ defined on the Euclidean space should be viewed as a 1-form. The gradient defined in (\ref{app:background:M:gradient}) is directly related to the covariant derivative via the metric $g$. This definition is exactly the same as that in (\ref{app:background:Euclidean:2}) since in the Euclidean space, the metric $g$ in the local coordinate $\{\partial_i\}_{i=1}^d$ around $x$ is nothing but $\big(g_{ij}\big)_{1\leq i,j\leq d}=I_d$, where $g_{ij}:=g(\partial_i,\partial_j)$. In other words, if we view the Euclidean space as a manifold with the canonical metric, we can either view the total differentiation as a 1-form, the covariant derivative (\ref{app:background:Euclidean:1}), or as a vector field, the gradient (\ref{app:background:Euclidean:2}); but in the manifold setup, these two notions are not exactly the same but related by the chosen metric $g$ as in (\ref{app:background:M:gradient}).

With the above definitions and clarifications, for a fixed local coordinate around $x$, we have 
\begin{equation}\label{app:background:M:gradient:coor}
\mbox{\tt{grad}}m=\sum_{i,j=1}^d g^{ij}\partial_im\partial_j,
\end{equation}
where $\{\partial_l\}_{l=1}^d$ is the coordinate around $x$, $\partial_im$ is defined by (\ref{app:background:M:Xf}) and $\big(g^{ij}\big)_{1\leq i,j\leq d}$ is the inverse of $\big(g_{jk}\big)_{1\leq i,j\leq d}$,
while the covariant derivative of $m$ is 
\[
\ud m=\nabla m=\sum_{l=1}^d\partial_lm\ud x^l,
\] 
where $\{\ud x^l\}_{l=1}^d$ is the dual of $\{\partial_l\}_{l=1}^d$. Thus, if we choose a normal coordinate around $x$ so that $g_{ij}=\delta_{ij}$ at $x$, where $\delta_{ij}$ denotes the kronecker delta, the coefficients of the covariant derivative of $m$ at $x$ is the same as the coefficients of the gradient of $m$ at $x$. Note that $\mbox{\tt{grad}}m(x)$ (or $\ud m(x)$) is the same regardless the choice of the local basis.

Notice that as is stated in Theorem \ref{thm:interior:cov} and \ref{thm:boundary:cov}, the estimated first order covariant derivative of $m$, $\widehat{\nabla_{\partial_i}m}({x},h)$, depends on the estimated basis of $\iota_*T_{{x}}\MM$. Thus, we have to take this basis into account to estimate the embedded gradient of $m$, $\iota_*\nabla m({x})$, as is considered in (\ref{algo:estimator:CovDeri}). Notice that since MALLER provides the estimate of $\nabla_{\partial_l}m$ at $x$ for $l=1,\ldots,d$, we can get the estimate of the covariant derivative or the exterior derivative of $m$ by taking the dual basis of $\{\partial_l\}_{l=1}^d$ into consideration.
 
We demonstrate the detailed calculation of the gradient given in (\ref{simulation:torus:gradient}). Since $\phi(u,v)=((2+\cos(v))\cos(u),(2+\cos(v))\sin(u),\sin(v))$, It is clear that
\[
\ud \phi=\left[
\begin{array}{cc}
-(2+\cos(v))\sin(u) & -\sin(v)\cos(u)\\
(2+\cos(v))\cos(u) & -\sin(v)\sin(u)\\
0 & \cos(v)
\end{array}
\right].
\]
By denoting $e_1=(1,0)\in\RR^2$ and $e_2=(0,1)\in\RR^2$, we get a set of embedded vector fields defined on $\phi([0,2\pi)\times[0,2\pi))$:
\[
E_1=\frac{\ud \phi(e_1)}{\|\ud \phi(e_1)\|}=(-\sin(u),\cos(u),0)
\]
and
\[
E_2=\frac{\ud \phi(e_2)}{\|\ud \phi(e_2)\|}=(-\sin(v)\cos(u),-\sin(v)\sin(u),\cos(v)),
\] 
which are orthonormal with related to the canonical metric of $\RR^3$. Since $\iota$ is an isometric embedding of the torus into $\RR^3$, $E_i=\iota_*\partial_i$, $i=1,2$, where $\partial_i$ is an orthonormal frame defined on the torus. Thus, by (\ref{app:background:M:gradient:coor}) the embedded gradient of $m$ at $\iota(x)$ can be evaluated by
\begin{equation}\label{app:background:M:gradient:example}
\iota_*(\mbox{\tt{grad}}m(x))=\partial_1m(x)\iota_*\partial_1(x)+\partial_2m(x)\iota_*\partial_2(x)=\partial_1m(x)E_1(x)+\partial_2m(x)E_2(x),
\end{equation}
where $\partial_i(x)$ is the value of $\partial_i$ at $x$. By definition, we have
\begin{eqnarray}
&&\partial_1m(x)=\frac{\ud m(c_1(t))}{\ud t}|_{t=0}=\frac{\ud m(\phi(u+\frac{t}{2+\cos(v)},v))}{\ud t}=\frac{-\sin(u)\sin(4v+1)}{2+\cos(v)}\nonumber\\
&&\partial_2m(x)=\frac{\ud m(c_2(t))}{\ud t}|_{t=0}=\frac{\ud m(\phi(u,v+t))}{\ud t}=4\cos(u)\cos(4v+1)\nonumber
\end{eqnarray}
where $\iota(x)=\phi(u,v)$, $c_i(0)=x$ and $c_i'(0)=\partial_i(x)$ for $i=1,2$. Note that $\ud\phi(e_1)$ is not of unit norm, so we have to normalize $e_1$ by $2+\cos(v)$ when we evaluate $\partial_1m(x)$. Plugging the above into (\ref{app:background:M:gradient:example}), we get (\ref{simulation:torus:gradient}).

\section{Proofs}

The following lemmas are needed to finish the proofs of the theoretical results. The proofs of the first three lemmas can be found in \cite{vdm}. The first lemma describes how the volume form depends on the curvature.
The second lemma describes how to express the relationship between two points on the manifold $\MM$ after being embedded in $\RR^p$. Recall that the notion of ``subtraction'' between two points on $\MM$ is not well defined. However, once these two points are embedded to $\RR^p$, the notion of ``subtraction'' makes sense, and the result of subtraction can be expressed by some geometric quantities of $\MM$ and the embedding itself. The third lemma describes the error when we try to estimate the geodesic distance between two close points on $\MM$  by the Euclidean distance between their embedded points. Notice that in practice the geodesic distance between two close points on $\MM$ is unknown a priori, and we can only estimate it by the Euclidean distance between their embedded points.
\begin{lem}\label{lemma1}
In polar coordinates around ${x}\in\MM$, the volume form $\ud V$ is
\[
\ud V(\exp_{{x}}t\theta)=\big(t^{d-1}+t^{d+1}\Ric(\theta,\theta)+O(t^{d+2})\big)\ud t\ud\theta,
\]
where $\theta\in T_{{x}}\MM$, $\|\theta\|=1$ and $t>0$.
\end{lem}

\begin{lem}\label{lemma2}
Fix ${x}\in\MM$ and denote by $\exp_{{x}}$ the exponential map at ${x}$. With the identification of $T_{\iota({x})}\RR^p$ with $\RR^p$, for $\theta\in T_{{x}}\MM$ with $\|\theta\|=1$ and $t\ll 1$, we have
\begin{equation}\label{relateexp1}
\iota\big(\exp_{{x}}t\theta\big)=\iota({x})+t\iota_*\theta+t^2\frac{\II_{{x}}(\theta,\theta)}{2}+O(t^3).
\end{equation}
\end{lem}

\begin{lem}\label{lemma3}
Suppose ${x},{y}\in \MM$ such that ${y}=\exp_{{x}}(t\theta)$, where $\theta\in T_{{x}}\MM$ and $\|\theta\|=1$. If $t\ll 1$, then $\tilde{t}=\|\iota({x})-\iota({y})\|_{\RR^p}\ll 1$ satisfies
\begin{equation}
t=\tilde{t}+\frac{1}{24}\|\II_{{x}}(\theta,\theta)\|\tilde{t}^3+O(\tilde{t}^4).
\end{equation}
\end{lem}

By combining the above lemmas, we get the following two lemmas. In Lemma \ref{lemma4}, we quantify the volume error introduced by estimating the geodesic distance between two points $x,y\in\MM$ by the Euclidean distance between $\iota(x)\in\RR^p$ and $\iota(y)\in\RR^p$. In Lemma \ref{lemma5}, we collect some routine calculus. 
\begin{lem}\label{lemma4}
Fix ${x}\in\MM$ and $0<\delta\ll 1$. For $v_i\in S^{p-1}$, $i=1,\ldots,\ell$, we have
\begin{eqnarray}
\int_{\tilde{B}^{\MM}_\delta({x})}\Pi_{i=1}^\ell\langle y-x, v_i\rangle\ud V({y})
=\int_{B^\MM_{\delta}({x})}\Pi_{i=1}^\ell\langle y-x, v_i\rangle\ud V({y})+O(\delta^{d+\ell+2})\nonumber.
\end{eqnarray}
where
$$
\tilde{B}^{\MM}_\delta({x}):=\iota^{-1}\left(B^{\RR^p}_\delta(x)\cap \iota(\MM)\right)\subset \MM.
$$ 
In particular, the volume of $\tilde{B}^{\MM}_\delta({x})$ differs from that of $B^\MM_{\delta}({x})$ by $O(\delta^{d+2})$. 
\end{lem}
\begin{proof}
By direct calculation:
\begin{eqnarray}
&&\int_{\tilde{B}^{\MM}_\delta({x})}\Pi_{i=1}^\ell\langle y-x, v_i\rangle\ud V({y}) \nonumber\\
&=& \int_0^{\delta+O(\delta^{3})}\int_{S^{d-1}}\Pi_{i=1}^\ell\langle t\iota_*\theta+O(t^2), v_i\rangle \left[t^{d-1}+O(t^{d+1})\right]\ud\theta\ud t\nonumber\\
&=&\int_0^{\delta}\int_{S^{d-1}} \Pi_{i=1}^\ell\langle t\iota_*\theta+O(t^2)\left[t^{d-1}+O(t^{d+1})\right]\ud\theta\ud t+O(\delta^{d+l+2})\nonumber\\
&=&\int_{B^\MM_{\delta}({x})}\Pi_{i=1}^\ell\langle y-x, v_i\rangle\ud V({y})+O(\delta^{d+l+2})\nonumber,
\end{eqnarray}
where the first equality comes from Lemma \ref{lemma1}, Lemma \ref{lemma2} and Lemma \ref{lemma3} and the others comes from direction calculations.
\end{proof}

\begin{lem}\label{lemma5}
Fix ${x}\in \MM\backslash \MM_{\sqrt{h}}$, where $h\ll 1$, $v\in\RR^p$, a function $\phi\in C^2(\MM)$ and the kernel function $K$ compactly supported in $[0,1]$ so that $K|_{[0,1]}\in C^1([0,1])$. Then for each $\ell\in\NN$ we have:
\begin{align}
 (a)&\,\,\EE K^\ell_h(X,x)\phi(X)=\mu_{\ell,0}f({x})\phi({x})+O(h);\nonumber\\
 (b)&\,\,\EE K^\ell_h(X,x)(X-x)\phi(X)\nonumber\\
 &\quad=h \mu_{\ell,2}\Big\{\sum_{l=1}^d \Big[\phi({x})\iota_*\partial_l\nabla_{\partial_l} f({x})+f({x})\iota_*\partial_l\nabla_{\partial_l} \phi({x})\Big]\nonumber\\
 &\qquad\qquad\qquad\qquad\qquad+\frac{f({x})\phi({x})}{2|S^{d-1}|}\int_{S^{d-1}}\II_{{x}}(\theta,\theta) \ud \theta\Big\}+O(h^{\frac{3}{2}});\nonumber\\
 (c)&\,\,\EE \left(K^\ell_h(X,x)(X-x)(X-x)^T\phi(X)\right)_{i,j}\nonumber\\
 &\quad=\left\{\begin{array}{ll}
h\frac{\mu_{\ell,2}}{d}f({x})\phi({x})+O(h^2) & \mbox{when }1\leq i=j\leq d\\
O(h^2)&\mbox{ otherwise}
\end{array}\right.;\nonumber\\
 (d)&\,\,\EE K^\ell_h(X,x)(X-x)(X-x)^T\langle X-x,v\rangle \phi(X)\nonumber\\
 &\quad=h^2\frac{\mu_{\ell,4}}{|S^{d-1}|}\int_{S^{d-1}}\bigg\{\iota_*\theta\iota_*\theta^T\langle \iota_*\theta,v\rangle \big(\phi({x})\nabla_\theta f({x})+f({x})\nabla_\theta \phi({x})\big)\nonumber\\
&\qquad +\frac{f({x})\phi({x})}{2}\Big(\iota_*\theta\iota_*\theta^T\langle \II(\theta,\theta),v\rangle+\II_{{x}}(\theta,\theta)\iota_*\theta^T+\iota_*\theta\II_{{x}}(\theta,\theta)^T\Big)\langle \iota_*\theta,v\rangle\bigg\}\ud \theta\nonumber\\
&\qquad +O(h^{5/2}).\nonumber
\end{align} 
\end{lem}

\begin{proof}
{\allowdisplaybreaks
These expectations are evaluated by Taylor's expansion and by Lemma \ref{lemma1} to Lemma \ref{lemma4}. We start with evaluating (a). 
\begin{equation*}
\begin{split}
&\EE K^\ell_h(X,x)\phi(X)= \int_{\tilde{B}^{\MM}_{\sqrt{h}}({x})} K^\ell_{h}(y,x) \phi({y})f({y})\ud V({y})\\
=\,& \int_{B^\MM_{\sqrt{h}}({x})} K^\ell_{h}(y,x) \phi({y})f({y})\ud V({y})+O(h)\\
=\,& \int_{S^{d-1}}\int_0^{\sqrt{h}} h^{-d/2}\Big(K^\ell\Big(\frac{t}{\sqrt{h}}\Big)+O\Big(\frac{t^3}{\sqrt{h}}\Big)\Big)\Big(\phi({x})+t\nabla_\theta \phi({x})+O(t^2)\Big)\\
&\times\Big(f({x})+t\nabla_\theta f({x})+O(t^2)\Big)\big(t^{d-1}+O(t^{d+1})\big)\ud t\ud \theta+O(h)\\
=\,&\mu_{\ell,0}f({x})\phi({x})+O(h),
\end{split}
\end{equation*}
where the first equality comes from (\ref{definition:expectation:manifold}), the second equality comes from Lemma \ref{lemma3} and Lemma \ref{lemma4}, the third equality comes from the Taylor's expansion and Lemma \ref{lemma1} and the last equality comes from the symmetry of $S^{d-1}$. Indeed, the odd moments in the integral vanish because $S^{d-1}$ is symmetric.

Next, by the same arguments as those leading to (a) and Lemma \ref{lemma2}, the left hand side of (b) becomes:
\begin{eqnarray}
&&\EE K^\ell_h(X,x)(X-x)\phi(X)= \int_{\tilde{B}^{\MM}_{\sqrt{h}}({x})} K^{\ell}_{h}(y,x)(y-x) \phi({y})f({y})\ud V({y}) \nonumber \\
&=&\int_{B^{\MM}_{\sqrt{h}}({x})} K^{\ell}_{h}(y,x)(y-x) \phi({y})f({y})\ud V({y}) +O(h^{3/2}) \nonumber \\
&=&\int_{S^{d-1}}\int_0^{\sqrt{h}} h^{-d/2}\Big(K^\ell\Big(\frac{t}{\sqrt{h}}\Big)+O\Big(\frac{t^3}{\sqrt{h}}\Big)\Big)\Big(t\iota_*\theta+\frac{t^2}{2}\II_{{x}}(\theta,\theta)+O(t^3)\Big)\nonumber\\
&&\times(\phi({x})+t\nabla_\theta \phi({x})+O(t^2)\big)\big(f({x})+t\nabla_\theta f({x})+O(t^2)\big)\nonumber\\
&&\times\big(t^{d-1}+\Ric(\theta,\theta)t^{d+1}+O(t^{d+2})\big)\ud t\ud \theta+O(h^{3/2}) \nonumber \\
&=&h\int_{S^{d-1}}\int_0^{1} K^\ell\left(t\right)\Big(\phi({x})\iota_*\theta\nabla_\theta f({x})+f({x})\iota_*\theta\nabla_\theta \phi({x})\nonumber\\
&&\qquad\qquad\qquad\qquad\qquad+\frac{\II_{{x}}(\theta,\theta) f({x})\phi({x})}{2}\Big)t^{d+1}\ud t\ud \theta+O(h^{\frac{3}{2}}).\label{proof:lemma5:b}
\end{eqnarray}
A direct calculation shows that
\begin{align}
\int_{S^{d-1}}\theta \nabla_\theta f({x})\ud \theta
=\,&\sum_{l,k=1}^d\partial_i\nabla_{\partial_k}f({x})\int_{S^{d-1}} \theta^l\theta^k\ud \theta
=\frac{|S^{d-1}|}{d}\sum_{l=1}^d\partial_l\nabla_{\partial_l}f({x})\label{proof:lemma5:b:app}.
\end{align}
By plugging (\ref{proof:lemma5:b:app}) into (\ref{proof:lemma5:b}) we conclude (b). 

By the same arguments as those leading to (b), we get (c):
\begin{eqnarray}
&&\EE \left(K^\ell_h(X,x)(X-x)(X-x)^T\phi(X)\right)_{i,j}\nonumber\\
&=&\int_{\tilde{B}^{\MM}_{\sqrt{h}}({x})}K^\ell_{h}(y,x)(y-x)(y-x)^T \phi({y})f({y})\ud V({y}) \nonumber \\
&=&\int_{S^{d-1}}\int_0^{\sqrt{h}} h^{-d/2}\left(K\Big(\frac{t}{\sqrt{h}}\Big)+O\Big(\frac{t^3}{\sqrt{h}}\Big)\right)\Big(t^2\iota_*\theta\iota_*\theta^T+O(t^{3})\Big)\nonumber\\
&&\times\Big(\phi({x})+t\nabla_\theta \phi({x})+O(t^2)\Big)\Big(f({x})+t\nabla_\theta f({x})+O(t^2)\Big)\nonumber\\
&&\times\Big(t^{d-1}+\Ric(\theta,\theta)t^{d+1}+O(t^{d+2})\Big)\ud t\ud \theta+O(h^2)\nonumber \\
&=& hf({x})\phi({x})\int_{S^{d-1}}\int_0^{1} K\left(t\right)\iota_*\theta(\iota_*\theta)^Tt^{d+1}\ud t\ud \theta+O(h^{2})\nonumber\\
&=&\left\{\begin{array}{ll}
h\frac{\mu_{\ell,2}}{d}f({x})\phi({x})+O(h^2) & \mbox{when }1\leq i=j\leq d\\
O(h^2)&\mbox{ otherwise}
\end{array}\right.,
\end{eqnarray}
where the last equality comes from the fact that $\iota_*$ is linear.

Equation (d) follows from the same arguments as in the above:
\begin{eqnarray}
&&\EE K^\ell_h(X,x)(X-x)(X-x)^T\langle X-x,v\rangle \phi(X)\nonumber\\
&=&\int_{\tilde{B}^{\MM}_{\sqrt{h}}({x})} K^\ell_{h}(y,x)(y-x)(y-x)^T\langle y-x,v\rangle \phi({y})f({y})\ud V({y}) \nonumber \\
&=&\int_{S^{d-1}}\int_0^{\sqrt{h}} \frac{1}{h^{d/2}}\Bigg\{ K\Big(\frac{t}{\sqrt{h}}\Big)\Big(t^2\iota_*\theta(\iota_*\theta)^T+\frac{t^3}{2}\big(\II_{{x}}(\theta,\theta)\iota_*\theta^T+\iota_*\theta\II_{{x}}(\theta,\theta)^T\big)\Big)\nonumber\\
&&\times\Big(t\langle \iota_*\theta,v\rangle+\frac{t^2}{2}\langle \II(\theta,\theta),v\rangle\Big)\Big(\phi({x})+t\nabla_\theta \phi({x})\Big)\Big(f({x})+t\nabla_\theta f({x})\Big)t^{d-1}\nonumber\\
&&+O(t^{d+5})\Bigg\}\ud t\ud \theta+O(h^{5/2}) \nonumber \\
&=& h^2\frac{\mu_{\ell,4}}{|S^{d-1}|}\int_{S^{d-1}}\bigg\{\iota_*\theta\iota_*\theta^T\langle \iota_*\theta,v\rangle \big(\phi({x})\nabla_\theta f({x})+f({x})\nabla_\theta \phi({x})\big)\nonumber\\
&& +\frac{f({x})\phi({x})}{2}\Big(\iota_*\theta\iota_*\theta^T\langle \II(\theta,\theta),v\rangle+\II_{{x}}(\theta,\theta)\iota_*\theta^T+\iota_*\theta\II_{{x}}(\theta,\theta)^T\Big)\langle \iota_*\theta,v\rangle\bigg\}\ud \theta\nonumber\\
&& +O(h^{5/2}).\nonumber
\end{eqnarray}
}
\end{proof}

Next we describe how the local PCA provides the estimate of the tangent plane. 
Although locally a manifold $\MM$ is close to some Euclidean space, there is always a gap caused by the curvature of $\MM$. Lemma \ref{lemma6} states its influence on the tangent plane estimation by the local PCA. 

\begin{lem}\label{lemma6}
Suppose $h_{\text{pca}}\asymp n^{-\frac{2}{d+1}}$. Then, if ${x}\in \MM\backslash\MM_{\sqrt{h}}$, the eignvectors $\{\vu_l(x)\}_{l=1}^d$ corresponding to the $d$ largest eigenvalues of the sample covariance matrix $\Sigma_x$ formed in the local PCA differ from {\em an} orthonormal basis $\{\partial_k({x})\}_{k=1}^d$ to $T_{{x}}\MM$ by:
\begin{equation}\label{lemma6:statement1}
\vu_l(x) = \iota_*\partial_l({x})+O_p(\hpca^{5/4})\vw_l+O_p(\hpca^{3/4})\vw^\perp_l\quad\mbox{ for }l=1,\ldots,d,
\end{equation}
where $\vw_l\in\iota_*T_{{x}}\MM$, $\vw^\perp_l\perp \iota_*T_{{x}}\MM$, and $\|\vw_l\|=\|\vw^\perp_l\|=1$, and, if ${x}\in\MM_{\sqrt{h}}$, 
\begin{equation}\label{lemma6:statement2}
\vu_l(x) = \iota_*\partial_l({x})+O_p(\hpca^{3/4})\vw_l+O_p(\hpca^{1/2})\vw^\perp_l \quad\mbox{ for }l=1,\ldots,d,
\end{equation}
where $\vw_l\in\iota_*T_{{x}}\MM$, $\vw^\perp_l\perp \iota_*T_{{x}}\MM$, and both $\vw_l$ and $\vw^\perp_l$ are of $O(1)$.

Suppose $h_{\text{pca}}\asymp O(n^{-\frac{2}{d+2}})$ and ${x}\in \MM\backslash\MM_{\sqrt{h}}$, then a better convergence rate is achieved. Indeed, (\ref{lemma6:statement1}) becomes 
\begin{equation*}
\vu_l(x) = \iota_*\partial_l({x})+O_p(\hpca^{3/2})\vw_l+O_p(\hpca)\vw^\perp_l \quad\mbox{ for }l=1,\ldots,d.
\end{equation*}
\end{lem}

The proof of this lemma follows the same lines as those in \cite{vdm} except some wrinkles caused by the two differences mentioned above. We now detail these wrinkles and refer the readers to \cite{vdm} for the detailed proof.

\begin{proof}
{\allowdisplaybreaks
Fix ${x}\in \MM\backslash\MM_{\sqrt{h}}$. Choose a normal coordinate $\{\partial_k ({x})\}_{k=1}^d$ around ${x}$ and assume $\MM$ is properly rotated and translated so that $x=\mathbf{0}_{p\times 1}$ and $\ve_i=\iota_*\partial_i({x})$, for $i=1,\ldots,d$, where $\mathbf{0}_{p\times 1}$ is the $p\times 1$ zero vector and $\ve_i$ is the unit length $p\times 1$ vector with the $i$-th entry $1$. Denote $Z_x:=\chi_{B^{\RR^p}_{\sqrt{\hpca}}(x)\cap\iota(\MM)}(X)X$, where $\chi$ is the indicator function. 

For later use, we prepare some calculations. First, since $f\in C^2(\MM)$ and $\MM$ is compact, by plugging $\ell=1$ and $v_1=\ve_l$ into Lemma \ref{lemma5} and taking Taylor's expansion, we have
\begin{align}
&\EE\langle Z_x,\ve_l\rangle=\int_{\tilde{B}^{\MM}_{\sqrt{\hpca}}({x})} \langle y,\ve_l\rangle f({y})\ud V({y})\label{lpca:prelim1}\\
=&\,\int_{S^{d-1}}\int_0^{\sqrt{\hpca}} \Big\langle t\iota_*\theta+\frac{t^2}{2}\II_{{x}}(\theta,\theta),\ve_l\Big\rangle\left(f({x})+t\nabla_\theta f({x})\right) t^{d-1}\ud t\ud\theta+O(\hpca^{\frac{d}{2}+3/2})\nonumber\\
=&\, O(\hpca^{\frac{d}{2}+1}). \nonumber
\end{align}
Similar calculation leads to:
\begin{align}\label{lpca:prelim3}
&\EE\langle Z_x,\ve_k\rangle\langle Z_x,\ve_l\rangle=\bigg\{\begin{array}{ll}
                           \frac{|S^{d-1}|}{d}f({x})\hpca^{d/2+1}+O(\hpca^{d/2+2}) & \mbox{for } 1\leq k=l \leq d \\
                           O(\hpca^{d/2+2}) & \mbox{otherwise.} 
                         \end{array}
\end{align}

With (\ref{lpca:prelim1}) and (\ref{lpca:prelim3}), we can finish the proof. 
Recall that the sample mean of $\mathcal{N}^{\text{true}}_{x,\hpca}$ is denoted by $\mu_x$. Then, it follows from the Central Limit Theorem (CLT) and (\ref{lpca:prelim1}) that
\begin{align}
&\langle \mu_x,\ve_l\rangle=\frac{1}{n}\sum_{k=1}^{N_x} \langle X_{x_k},\ve_l\rangle=\left\{\begin{array}{ll}
O(\hpca^{d/2+1})+O_p\big(n^{\frac{-1}{2}}\hpca^{d/4+1}\big) & \mbox{if }l=1,\ldots,d\\
O(\hpca^{d/2+1})+O_p\big(n^{\frac{-1}{2}}\hpca^{d/4+2}\big) & \mbox{otherwise.}\\
\end{array}\right.\nonumber
\end{align}
Since $\hpca^{d/2+1}$ dominates $n^{-1/2}\hpca^{d/4+1}$ asymptotically, due to the assumption $h_{\text{pca}}\asymp n^{-\frac{2}{d+2}}$, we conclude that
\begin{equation}\label{lpca:mu_by_x}
\mu_x=O_p\big(\hpca^{d/2+1}\big).
\end{equation}
Next we consider the sample covariance matrix $\Sigma_x$. By (\ref{lpca:prelim1}),  (\ref{lpca:prelim3}), (\ref{lpca:mu_by_x}), and similar calculation as in the above, we have
\begin{eqnarray}
&&\Sigma_x(i,j) \label{SigmaX}= \frac{1}{n}\sum_{l=1}^{N_x} \langle X_{x_l}-\mu_x,\ve_i\rangle\langle X_{x_l}-\mu_x,\ve_j\rangle\nonumber\\
&=&\left\{\begin{array}{ll}
\EE \langle Z_x,\ve_i\rangle\langle Z_x,\ve_j\rangle+O_p\left(\hpca^{d+2}\right)+O_p\big(n^{\frac{-1}{2}}\hpca^{d/4+1}\big)&\mbox{if }1\leq i,j\leq d\\
\EE \langle Z_x,\ve_i\rangle\langle Z_x,\ve_j\rangle+O_p\left(\hpca^{d+2}\right)+O_p\big(n^{\frac{-1}{2}}\hpca^{d/4+2}\big) &\mbox{if }d+1\leq i,j\leq p \\
\EE \langle Z_x,\ve_i\rangle\langle Z_x,\ve_j\rangle+O_p\left(\hpca^{d+2}\right)+O_p\big(n^{\frac{-1}{2}}\hpca^{d/4+3/2}\big) &\mbox{otherwise, }
\end{array}
\right.\nonumber
\end{eqnarray}
where 
the second $O_p$ term comes from the finite sample variance. By (\ref{lpca:prelim3}) and the assumption $\hpca\asymp n^{-\frac{2}{d+1}}$, we get
\begin{eqnarray}
\Sigma_x =\frac{|S^{d-1}|f({x})}{d} \hpca^{d/2+1}\left\{\left[
\begin{array}{ll}
I_{d} & \mathbf{0}_{d\times p-d} \\
\mathbf{0}_{p-d\times d} & \mathbf{0}_{p-d\times p-d} \\
\end{array}
\right] +  \left[\begin{array}{ll}
                                                                      O_p(\hpca^{1/2}) & O_p(\hpca) \\
                                                                      O_p(\hpca) & O_p(\hpca)
                                                                    \end{array}
 \right]\right\},\nonumber
\end{eqnarray}
where $\mathbf{0}_{m\times m'}$ is the zero matrix of size $m \times m'$, for any $m,m'\in\NN$.
As a result, we get the equation (B.44) in \cite{vdm}. Then we can analyze $\Sigma_x$ by the perturbation theory  exactly in the same way as in \cite{vdm}, so we skip the details. When ${x}\in \MM_{\sqrt{h}}$, the same calculation applies and we skip the details.
}
\end{proof}

Before proving Theorem \ref{thm:interior} and Theorem \ref{thm:boundary}, we prepare some notation and setups. Fix $x$. Recall that $B_x$ is a $p\times d$ matrix with the $k$-th column $\vu_k(x)$ determined by the local PCA. 
Denote $\vy:=B_x^T(y-x)$ and $\vx_l:=B_x^T(X_l-x)$, where $y\in\MM$ and $X_l\in\mathcal{X}$. To simplify the notation, we denote 
\begin{eqnarray}
&&\HHH:=B_x\Hess m({x})B_x^T,\nonumber \\
&&\SSS_x:=\diag\big(\sigma^2(\iota^{-1}(X_1)),\ldots,\sigma^2(\iota^{-1}(X_n))\big),\nonumber \\
&&\QQQ_m({x}):=\big[\vx^T_1\Hess m({x})\vx_1\quad\ldots\quad \vx^T_n\Hess m({x})\vx_n \big]^T.\nonumber
\end{eqnarray}
For a given function $\phi:\MM\mapsto \RR$, $\ell\in\NN$ and $v\in\RR^p$, we define
\begin{eqnarray}
&&\EEE^\ell_0(\phi):=\EE K^\ell_h(X,x)\phi(X),\nonumber\\
&&\EEE^\ell_1(\phi):=\EE K^\ell_h(X,x)(X-x)\phi(X),\nonumber\\
&&\EEE^\ell_2(\phi):=\EE K^\ell_h(X,x)(X-x)(X-x)^T\phi(X),\nonumber\\
&&\EEE^\ell_{3,v}(\phi):=\EE K^\ell_h(X,x)(X-x)(X-x)^T\langle X-x,v\rangle\phi(X),\nonumber\\
&&\mathfrak{q}_1:=\frac{1}{n}\sum_{l=1}^nK_h(X_l,x)\vx^T_l\Hess m({x})\vx_l,\nonumber\\
&&\mathfrak{q}_2:=\frac{1}{n}\sum_{l=1}^nK_h(X_l,x)\vx^T_l\Hess m({x})\vx_l\vx_l.\nonumber
\end{eqnarray}

\subsection{[Proof of Theorem \ref{thm:interior}]}\label{proof:interior}
\begin{proof}
{\allowdisplaybreaks
Fix $x\in \MM$. Denote by $\{\vu_k(x)\}_{k=1}^d$ the orthonormal set determined by local PCA. 
Choose an orthonormal basis $\{\ve_k\}_{k=1}^p$ of $\RR^p$, where $\ve_k$ is the $p\times 1$ unit norm column vector with the $k$-th entry $1$, and assume $\iota$ is properly rotated and translated so that $x=\mathbf{0}_{p\times 1}$ and $\ve_i=\iota_*\partial_i({x})$ for $i=1,\ldots,d$, where $\mathbf{0}_{p\times 1}$ is the $p$-dimensional zero vector. 

With the notation $\vY$ and $\vm$ defined in (\ref{def:Yandm}), clearly we have 
\begin{eqnarray}
\EE\{\hat{m}(x,h)|\mathcal{X}\}
=\vv_1^T(\XX^T_x\WW_x\XX_x)^{-1}\XX^T_x\WW_x\EE \vY=\vv_1^T(\XX^T_x\WW_x\XX_x)^{-1}\XX^T_x\WW_x\vm.\label{thm:interior:proof:Ehatm}
\end{eqnarray}
Take ${y}=\exp_{{x}}(t\theta)$, where $t=O(h^{1/2})$ and $\|\theta\|=1$. By Lemma \ref{lemma2} we have
\begin{equation}\label{relateexp}
t\iota_*\theta=\iota({y})-x-\frac{t^2}{2}\II_{{x}}(\theta,\theta)+O(t^3),
\end{equation}
which by Lemma \ref{lemma6} leads to 
\begin{equation}\label{thetavu}
\langle \iota_*\theta, \vu_k(x)\rangle=\langle \iota_*\theta, \iota_*\partial_k\rangle+O_p(\hpca^{5/4}),
\end{equation}
since $\vw_k^\perp$ is perpendicular to $\iota_*\theta$, and
\begin{equation}\label{pivu}
\langle \II_{{x}}(\theta,\theta), \vu_k(x)\rangle=O_p(\hpca^{3/4}),
\end{equation}
since the second fundamental form $\II_{{x}}$ is perpendicular to the embedded tangent plane $\iota_*T_{{x}}\MM$.
Therefore, for $j=1,\ldots,d$, we have
\begin{eqnarray}
&&\langle t\iota_*\theta,\,\ve_j\rangle=\langle t\iota_*\theta,\,\vu_j(x)-O_p(\hpca^{5/4})\vw_j\rangle\label{uerelation}\\
&=&\langle y-x,\,\vu_j(x)\rangle- \frac{t^2}{2}\langle \II_{{x}}(\theta,\theta),\,\vu_j(x)\rangle+O_p(h^{1/2}\hpca^{5/4})\nonumber\\
&=&\langle y-x,\,\vu_j(x)\rangle+ O_p(h\hpca^{3/4}+h^{1/2}\hpca^{5/4})\nonumber\\
&=&\vy_j+O_p(h\hpca^{3/4}),\nonumber
\end{eqnarray}
where the first equality holds due to Lemma \ref{lemma6}, the second equality holds due to (\ref{relateexp}), the third equality holds due to (\ref{pivu}), and the last equality holds due to the assumption that $\hpca\leq h$.
By Taylor's expansion on $\MM$, (\ref{uerelation}), and the assumption that $\hpca\leq h$, 
\begin{eqnarray}
&&m({y})-m({x})\label{proof:thm1:m_taylor}\\
&=&t\theta\nabla m({x})+\frac{t^2}{2}\Hess m({x})(\theta,\theta)+O(t^3)\nonumber\\
&=&\sum_{j=1}^d\langle t\iota_*\theta,\ve_j\rangle\nabla_{\partial_j} m({x})+\frac{1}{2}\sum_{i,j=1}^d\langle t\iota_*\theta,\ve_i\rangle\langle t\iota_*\theta,\ve_j\rangle\Hess m({x})(\partial_i,\partial_j)+O(h^{\frac{3}{2}})\nonumber\\
&=&\vy^T\nabla m({x})+\frac{1}{2}\vy^T\Hess m({x})\vy+O_p(h\hpca^{\frac{3}{4}}),\nonumber
\end{eqnarray}
where the second equality is obtained by rewriting $\theta=\sum_{k=1}^dg(\theta,\partial_k({x}))\partial_k({x})=\sum_{k=1}^d\langle \iota_*\theta,\ve_k\rangle \partial_k({x})$, because $\iota$ is isometric. Since the kernel $K$ is compactly supported, $m$ is bounded, and $\MM$ is smooth and compact, (\ref{proof:thm1:m_taylor}) leads to 
\begin{eqnarray}
\WW_x\vm=\WW_x\Big(\XX_x\Big[\begin{array}{c}
m({x})\\
\nabla m({x})
\end{array} 
\Big]+\frac{1}{2}\QQQ_m({x})+O_p(h\hpca^{\frac{3}{4}})\Big),\label{thm:interior:proof:M_Taylor}
\end{eqnarray}
where $\XX_x$ is defined in (\ref{design}) and $\WW_x$ is defined in (\ref{weighted}).
By plugging (\ref{thm:interior:proof:M_Taylor}) into (\ref{thm:interior:proof:Ehatm}), the conditional bias is reduced to
\begin{equation}\label{mfd:interior:estimator}
\EE\{\hat{m}({x},h)-m({x})|\mathcal{X}\}=\vv_1^T(\XX_x^T\WW_x\XX_x)^{-1}\XX_x^T\WW_x(\QQQ_m({x})+O_p(h\hpca^{\frac{3}{4}})).
\end{equation}
Now we evaluate (\ref{mfd:interior:estimator}). By direct expansion, we have
\begin{equation}\label{mfd:interior:XTWX1}
\frac{1}{n}\XX^T_x\WW_x\XX_x=\left[
\begin{array}{cc}
\frac{1}{n}\sum_{l=1}^n K_h(X_l,x) & \frac{1}{n}\sum_{l=1}^n K_h(X_l,x)\vx^T_l \\
\frac{1}{n}\sum_{l=1}^n K_h(X_l,x)\vx_l & \frac{1}{n}\sum_{l=1}^n \vx_l K_h(X_l,x)\vx^T_l
\end{array}
\right].
\end{equation} 
Denote by $\mathbf{1}$ the constant function with value 1. By the CLT, we have
\begin{equation}\label{S0}
\frac{1}{n}\sum_{l=1}^n K_h(X_l,x)=\EEE^1_0(\mathbf{1})+ O_p\Big(\frac{1}{n^{\frac{1}{2}}h^{\frac{d}{4}}}\Big),
\end{equation}
\begin{equation}\label{S1}
\frac{1}{n}\sum_{l=1}^n K_h(X_l,x)\vx_l=B^T_x\EEE^1_1(\mathbf{1}) + O_p\Big(\frac{1}{n^{\frac{1}{2}}h^{\frac{d}{4}-\frac{1}{2}}}\Big),
\end{equation}
and
\begin{equation}\label{S2}
\frac{1}{n}\sum_{l=1}^n \vx_l K_h(X_l,x)\vx^T_l=B^T_x\EEE^1_2(\mathbf{1})B_x + O_p\Big(\frac{1}{n^{\frac{1}{2}}h^{\frac{d}{4}-1}}\Big).
\end{equation}
Note that in (\ref{S1}), the random variables $\{K_h(X_l,x)\vx_l\}_{l=1}^n$ are not independent since $\vx_l=B_x^T(X_l\hspace{-1pt}-\hspace{-1pt}x)$ and $B_x$ is evaluated from the random samples $\{X_l\}_{l=1}^n$, and hence the CLT can not be applied directly. However, once we rewrite the left-hand side of (\ref{S1}) as 
$B_x^T\left(\frac{1}{n}\sum_{l=1}^n K_h(X_l,x)(X_l\hspace{-1pt}-\hspace{-1pt}x)\right),$ 
the summands become independent, and the CLT can be applied. The same comment applies to (\ref{S2}).
The expectation in (\ref{S0}) is clear from Lemma \ref{lemma5}. The expectation in (\ref{S1}) becomes
\begin{eqnarray}
B^T_x\EEE^1_1(\mathbf{1})&=&h\frac{\mu_{1,2}}{d}B^T_x\sum_{j=1}^d\iota_*\partial_j\nabla_{\partial_j} f({x})\nonumber\\
&&\qquad+h \int_{S^{d-1}}\int_0^{1} K(t)\frac{B^T_x\II_{{x}}(\theta,\theta) f({x})}{2}t^{d+1}\ud t\ud \theta+O(h^{\frac{3}{2}})\nonumber\\
&=&h\frac{\mu_{1,2}}{d}B^T_x\sum_{j=1}^d\iota_*\partial_j\nabla_{\partial_j} f({x})+O_p(h\hpca^{\frac{3}{4}})+O(h^{\frac{3}{2}})\nonumber\\
&=&h\frac{\mu_{1,2}}{d}\nabla f({x})+O_p(h^{\frac{3}{2}}),\nonumber
\end{eqnarray}
where the first equality holds due to Lemma \ref{lemma5}, the second equality holds due to (\ref{pivu}) and the third equality holds due to (\ref{thetavu}) and the assumption that $\hpca\leq h$. Similarly, the expectation in (\ref{S2}) becomes
\begin{eqnarray}
B^T_x\EEE^1_2(\mathbf{1})B_x
&=&hf({x})\int_{S^{d-1}}\int_0^{1} K\left(t\right)\theta\theta^Tt^{d+1}\ud t\ud \theta+O_p(h\hpca^{\frac{5}{4}})+O(h^{2})\nonumber\\
&=&h\frac{\mu_{1,2}}{d}f({x})I_{d}+O_p(h^{2}),\nonumber
\end{eqnarray}
where the first equality comes from Lemma \ref{lemma5} and (\ref{thetavu}).
As a result, (\ref{mfd:interior:XTWX1}) becomes
\begin{eqnarray}
&&\frac{1}{n}\XX_x^T\WW_x\XX_x=\left[
\begin{array}{cc}
f({x}) & {h}\frac{\mu_{1,2}}{d}\nabla f({x})^T\\
{h}\frac{\mu_{1,2}}{d}\nabla f({x}) & {h}\frac{\mu_{1,2}}{d}f({x})I_{d}
\end{array}
\right]\nonumber\\
&&\quad+\left[
\begin{array}{cc}
O(h)+O_p\Big(\frac{1}{n^{1/2}h^{d/4}}\Big) & O({h}^{3/2})+O_p\Big(\frac{1}{n^{1/2}h^{d/4-1/2}}\Big)\\
O({h}^{3/2})+O_p\Big(\frac{1}{n^{1/2}h^{d/4-1/2}}\Big)& O({h}^{2})+O_p\Big(\frac{1}{n^{1/2}h^{d/4-1}}\Big)
\end{array}
\right].\nonumber
\end{eqnarray}
Since $h\to 0$ and $nh^{d/2}\to \infty$ as $n\to\infty$, we know $\frac{1}{n}\XX_x^T\WW_x\XX_x$ is invertible with probability tending to 1 as $n\to\infty$. Also, since $f({x})+O(h)+O_p\big(\frac{1}{n^{1/2}h^{d/4}}\big)$ and $h\frac{\mu_{1,2}}{d}f({x})I_{d}+O({h}^{2})+O_p\big(\frac{1}{n^{1/2}h^{d/4-1}}\big)$ are also invertible with probability tending to 1 as $n\to\infty$, by the binomial inverse theorem, 
\begin{eqnarray}
&&\Big(\frac{1}{n}\XX_x^T\WW_x\XX_x\Big)^{-1}=\left[
\begin{array}{cc}
f({x})^{-1} &-f({x})^{-2}\nabla f({x})^T\\
-f({x})^{-2}\nabla f({x}) & h^{-1}\frac{d}{\mu_{1,2}f({x})}I_{d}
\end{array}
\right]\label{mfd:interior:Xinv}\\
&&\qquad\qquad+\left[\begin{array}{cc}
O(h)+O_p\Big(\frac{1}{n^{1/2}h^{d/4}}\Big) & O(h^{1/2})+O_p\Big(\frac{1}{n^{1/2}h^{d/4+1/2}}\Big)\\
O(h^{1/2})+O_p\Big(\frac{1}{n^{1/2}h^{d/4+1/2}}\Big)& O(1)+O_p\Big(\frac{1}{n^{1/2}h^{d/4+1}}\Big)
\end{array}\right].\nonumber
\end{eqnarray}
Next we consider $\frac{1}{n}\XX_x^T\WW_x\QQQ_m({x})$. By a direct calculation, 
\begin{equation}\label{thm1:proof:XTWQ}
\frac{1}{n}\XX_x^T\WW_x\QQQ_m({x})=\left[
\begin{array}{c}
\mathfrak{q}_1\\
\mathfrak{q}_2
\end{array}
\right].
\end{equation}
Note that, for any $n\times n$ matrix $Z$ and any $n\times 1$ column vector $v$,
\begin{equation}\label{vtAv}
v^TZv=\tr(Zvv^T).
\end{equation} 
By (\ref{vtAv}) and the CLT, we have
\begin{eqnarray}
\mathfrak{q}_1&=&\frac{1}{n}\sum_{l=1}^nK_h(X_l,x) (X_l-x)^T\HHH(X_l-x)\nonumber\\
&=&\tr \Big(\HHH\frac{1}{n}\sum_{l=1}^nK_h(X_l,x)(X_l-x)(X_l-x)^T\Big)\nonumber\\
&=&\tr \big(\HHH\EEE^1_2(\mathbf{1})\big) + O_p\Big(\frac{1}{n^{1/2}h^{d/4-1}}\Big)\label{thm1:proof:tentativeKhxTHessx}.
\end{eqnarray}
We evaluate $\tr \big(\HHH\EEE_2\big)$ by
\begin{eqnarray}
\tr \big(\HHH\EEE^1_2(\mathbf{1})\big)&=&hf({x})\tr \Big(\HHH\int_{S^{d-1}}\int_0^{1} K\left(t\right)\iota_*\theta\iota_*\theta^Tt^{d+1}\ud t\ud \theta\Big) +O(h^{2})\nonumber\\
&=&hf({x})\int_{S^{d-1}}\int_0^{1} K(t)\theta^T \Hess m({x})\theta t^{d+1}\ud t\ud \theta+O(h^{2})\nonumber\\
&=& h\frac{\mu_{1,2}}{d}f({x})\Delta m({x})+O_p(h^{2}),\label{thm1:proof:tentativeKhxTHessx1}
\end{eqnarray}
where the first equality comes from Lemma \ref{lemma5}, the second equality comes from (\ref{thetavu}) and (\ref{vtAv}) and the last equality holds due to the symmetry of $S^{d-1}$ and the definition of the Laplace-Beltrami operator. 

Then we evaluate $\mathfrak{q}_2$ in (\ref{thm1:proof:XTWQ}). Choose $\{\veee_k\}_{k=1}^p$ as an orthonormal basis of $\RR^p$ and rewrite $X_l-x=\sum_{k=1}^p\langle X_l-x,\veee_k\rangle \veee_k$. Note that the random variables $K_h(X_l,x)(X_l-x)(X_l-x)^T\langle X_l-x,\veee_k\rangle$ are independent. By (\ref{vtAv}) and the CLT, 
\begin{eqnarray}
\mathfrak{q}_2&=&\frac{1}{n}\sum_{l=1}^nK_h(X_l,x)\tr \Big(\HHH(X_l-x)(X_l-x)^T\Big)B^T_x(X_l-x)\label{proof:thm1:XWQ2}\\
&=&B^T_x\sum_{k=1}^p\tr \Big(\HHH\frac{1}{n}\sum_{l=1}^nK_h(X_l,x)(X_l-x)(X_l-x)^T\langle X_l-x,\veee_k\rangle\Big)\veee_k\nonumber\\
&=&B^T_x\sum_{k=1}^p\tr \Big(\HHH\EEE^1_{3,\veee_k}(\mathbf{1}) \Big)\veee_k + O_p\Big(\frac{1}{n^{\frac{1}{2}}h^{\frac{d}{4}-\frac{3}{2}}}\Big).\nonumber
\end{eqnarray}
By the same arguments as those for $\mathfrak{q}_1$, we have
\begin{eqnarray*}
&&B^T_x\sum_{k=1}^p\tr \Big(\HHH\EEE^1_{3,\veee_k}(\mathbf{1}) \Big)\veee_k\\
&=&h^2\frac{\mu_{1,2}}{|S^{d-1}|}B^T_x\sum_{k=1}^p\tr \Big(\HHH\int_{S^{d-1}}\iota_*\theta\iota_*\theta^T[\langle \iota_*\theta,\veee_k\rangle \nabla_\theta f({x})+\frac{f({x})}{2}\langle \II(\theta,\theta),\veee_k\rangle]\ud \theta \Big)\veee_k\\
&&+h^2\frac{\mu_{1,2}f({x})}{2|S^{d-1}|}B^T_x\sum_{k=1}^p\tr \Big(\HHH\int_{S^{d-1}}[\II_{{x}}(\theta,\theta)\iota_*\theta^T+\iota_*\theta\II_{{x}}(\theta,\theta)^T]\langle \iota_*\theta,\veee_k\rangle \ud \theta\Big)\veee_k\\
&=&h^2\frac{\mu_{1,2}}{|S^{d-1}|}\int_{S^{d-1}}\theta^T\Hess m({x})\theta \theta \nabla_\theta f({x})\ud \theta +O_p(h^{5/2}),
\end{eqnarray*}
where the first equality holds by Lemma \ref{lemma5} and the second equality holds by (\ref{thetavu}), (\ref{pivu}), (\ref{vtAv}) and (\ref{proof:thm1:XWQ2}).

As a result, (\ref{thm1:proof:XTWQ}) becomes
\begin{eqnarray}
\frac{1}{n}\XX_x^T\WW_x\QQQ_m({x})&=&\left[
\begin{array}{c}
h\frac{\mu_{1,2}}{d}f({x})\Delta m({x})\\
h^2\frac{\mu_{1,2}}{|S^{d-1}|}\int_{S^{d-1}}\theta^T\Hess m({x})\theta \theta \nabla_\theta f({x})\ud \theta
\end{array}
\right]\nonumber\\
&&+\left[
\begin{array}{c}
O_p(h^{2})+O_p\Big(\frac{1}{n^{1/2}h^{d/4-1}}\Big)\\
O_p(h^{\frac{5}{2}})+O_p\Big(\frac{1}{n^{\frac{1}{2}}h^{d/4-\frac{3}{2}}}\Big)
\end{array}
\right]\label{mfd:interior:Hess}
\end{eqnarray}
Lastly, since $m\in C^3(\MM)$ and $\MM$ is compact, a simple uniform bound combined with (\ref{mfd:interior:Xinv}) yields that the remainder term in (\ref{mfd:interior:estimator}) is $O_p(h\hpca^{3/4})$. 
Plug (\ref{mfd:interior:Xinv}), (\ref{mfd:interior:Hess}) and this result into (\ref{mfd:interior:estimator}), we conclude that 
\begin{equation}\label{rslt:interior:mfd}
\EE\{\hat{m}({x},h)-m({x})| \mathcal{X}\}
=h\frac{\mu_{1,2}}{2d}\Delta m({x})+O_p(h^2+h\hpca^{3/4})+O_p\Big(\frac{1}{n^{\frac{1}{2}}h^{\frac{d}{4}-1}}\Big).
\end{equation}

Next consider the conditional variance. A direct calculation gives
\begin{equation}\label{mfd:interior:var}
\begin{split}
&\operatorname{Var}\{\hat{m}({x},h)|\mathcal{X}\}\\
=\,&\vv_1^T(\XX_x^T\WW_x\XX_x)^{-1}\XX_x^T\WW_x\SSS_x\WW_x\XX_x(\XX_x^T\WW_x\XX_x)^{-1}\vv_1\\
=\,&\frac{1}{n}\vv_1^T\Big(\frac{1}{n}\XX_x^T\WW_x\XX_x\Big)^{-1}\Big(\frac{1}{n}\XX_x^T\WW_x\SSS_x\WW_x\XX_x\Big)\Big(\frac{1}{n}\XX_x^T\WW_x\XX_x\Big)^{-1}\vv_1.
\end{split}
\end{equation}
By the CLT
\begin{align}
&\frac{1}{n}\XX_x^T\WW_x\SSS_x\WW_x\XX_x\nonumber\\
=&\,\left[
\begin{array}{cc}
\frac{1}{n}\sum_{l=1}^n K^2_h(X_l,x)\sigma^2(X_l) & \frac{1}{n}\sum_{l=1}^n K^2_h(X_l,x)\vx_l\sigma^2(X_l) \\
 \frac{1}{n}\sum_{l=1}^n K^2_h(X_l,x)\vx^T_l\sigma^2(X_l) &  \frac{1}{n}\sum_{l=1}^n K^2_h(X_l,x)\vx_l\vx^T_l\sigma^2(X_l)
\end{array}
\right]\nonumber\\
=\,&\left[
\begin{array}{cc}
\EEE^2_0(\sigma^2) & B_x^T\EEE^2_1(\sigma^2) \\
\EEE^2_1(\sigma^2)^TB_x & B_x^T\EEE^2_2(\sigma^2)B_x
\end{array}
\right]\nonumber\\
&\qquad\qquad\qquad\qquad\qquad+\left[
\begin{array}{cc}
O_p\Big( \frac{1}{n^{1/2}h^{3d/4}}\Big)&O_p\Big( \frac{1}{n^{1/2}h^{3d/4-1/2}}\Big)\\
O_p\Big( \frac{1}{n^{1/2}h^{3d/4-1/2}}\Big)&O_p\Big( \frac{1}{n^{1/2}h^{3d/4-1}}\Big)
 \end{array}
\right].\nonumber
\end{align}
We evaluate the expectations by the same arguments as those above and get
\begin{eqnarray}\label{mfd:interior:XWVWX}
&&\hspace{+18pt}\frac{1}{n}\XX_x^T\WW_x\SSS_x\WW_x\XX_x\nonumber\\
&=& h^{-\frac{d}{2}}\bigg\{\bigg[
\begin{array}{cc}
\mu_{2,0}\sigma^2({x})f({x}) & h\vv_*\\
h\vv^T_* & hd^{-1}\mu_{2,2}\sigma^2({x})f({x})I_{d}
\end{array}
\bigg]\\
&&\,+\bigg[
\begin{array}{cc}
O(h)+O_p\Big( \frac{1}{n^{\frac{1}{2}}h^{\frac{d}{4}}}\Big) &O_p(h^{2}+h\hpca^{\frac{3}{4}})+O_p\Big( \frac{1}{n^{\frac{1}{2}}h^{\frac{d}{4}-\frac{1}{2}}}\Big) \\
O_p(h^{2}+h\hpca^{3/4})+O_p\Big(\frac{1}{n^{\frac{1}{2}}h^{\frac{d}{4}-\frac{1}{2}}}\Big) & O_p(h^{2})+O_p\Big( \frac{1}{n^{\frac{1}{2}}h^{\frac{d}{4}-1}}\Big)
\end{array}
\bigg]\bigg\},\nonumber
\end{eqnarray}
where $\vv_*=\frac{\mu_{2,2}\sigma({x})}{d} \big[2f\nabla \sigma+\sigma\nabla f\big]({x})$.
Due to (\ref{mfd:interior:Xinv}) and (\ref{mfd:interior:XWVWX}), (\ref{mfd:interior:var}) becomes
\begin{eqnarray}
\operatorname{Var}\{\hat{m}({x},h)| \mathcal{X}\}=\frac{1}{nh^{d/2}}\frac{\mu_{2,0}\sigma^2({x})}{f({x})}+O_p\Big(\frac{1}{nh^{d/2-1}}+\frac{1}{n^{3/2}h^{3d/4}}\Big).\label{mfd:interior:XWXinvXWVWXXWXinv}
\end{eqnarray}
Thus, the asymptotic conditional MSE in (\ref{thm:interior:cond_mse}) follows from (\ref{rslt:interior:mfd}) and (\ref{mfd:interior:XWXinvXWVWXXWXinv}).
In conclusion, when $\hpca\leq h$, the minimal asymptotic conditional MSE is achieved when $nh^{d/2}\asymp h^{-2}$, as is claimed. Note that $\hpca$ and $h$ are thus related by $\hpca=h^{(d+4)/(d+1)}<h$.


The conditional bias of the estimator $\widehat{\nabla_{\partial_i}m}({x},h)$, for $i=1,\ldots,d$, are evaluated by following exactly the same lines as in the proof of (\ref{mfd:interior:estimator}):
\begin{eqnarray}
&&\EE\{\widehat{\nabla_{\partial_i}m}({x},h)-\nabla_{\partial_i}m({x})| \mathcal{X}\}=\vv_{i+1}^T(\XX^T_x\WW_x\XX_x)^{-1}\XX^T_x\WW_x\vm\label{thm:interior:proof:EhatDm}\\
&=& \nabla_{\partial_i}m({x})+ \vv_{i+1}^T(\XX_x^T\WW_x\XX_x)^{-1}\XX_x^T\WW_x\QQQ_m({x})/2+O(h^{1/2}\hpca^{3/4}).\nonumber
\end{eqnarray}
By plugging (\ref{mfd:interior:Xinv}) and (\ref{mfd:interior:Hess}) into (\ref{thm:interior:proof:EhatDm}), we obtain
\begin{align}
&\EE\{\widehat{\nabla_{\partial_i}m}({x},h)-\nabla_{\partial_i}m({x})| \mathcal{X}\}\\
=&-h\frac{\mu_{1,2}}{d}\frac{\nabla f({x})^T}{f({x})}\Delta m({x})+h\frac{d\int_{S^{d-1}}\theta^T\Hess m({x})\theta \theta \nabla_\theta f({x})\ud \theta}{|S^{d-1}|f({x})}\nonumber\\
&+O(h^{\frac{3}{2}}+h^{\frac{1}{2}}\hpca^{\frac{3}{4}})+O_p\Big(\frac{1}{n^{\frac{1}{2}}h^{\frac{d}{4}-\frac{1}{2}}}\Big).\nonumber
\end{align}
The conditional variance term of $\widehat{\nabla_{\partial_i}m}({x},h )$ comes from (\ref{mfd:interior:Xinv}) and (\ref{mfd:interior:XWVWX}):
\begin{align}
&\operatorname{Var}\{\widehat{\nabla_{\partial_i}m}({x},h)| \mathcal{X}\}\\
=\,&\vv_{i+1}^T(\XX_x^T\WW_x\XX_x)^{-1}(\XX_x^T\WW_x\SSS_x\WW_x\XX_x)(\XX_x^T\WW_x\XX_x)^{-1}\vv_{i+1}\nonumber\\
=\,&\frac{1}{nh^{d/2+1}}\frac{d\mu_{2,2}\sigma^2({x})f({x})}{\mu_{1,2}}+O_p\Big(\frac{1}{nh^{d/2}}\Big)+O_p\Big(\frac{1}{n^{3/2}h^{3d/4+1}}\Big).\nonumber
\end{align}
The conditional MSE is then obtained directly and it leads to the conclusion that the minimal asymptotic conditional MSE is achieved when $nh^{d/2}\asymp h^{-3}$. 
}
\end{proof}


\subsection{[Proof of Theorem \ref{thm:boundary}]}
\begin{proof}
{\allowdisplaybreaks
The proof is smilier to that of Theorem \ref{thm:interior} except the boundary effect. We use the same notation $\{\vu_k(x)\}_{k=1}^d$, $\{\ve_k\}_{k=1}^p$ as those in the proof of Theorem \ref{thm:interior} and the same assumption for $\iota$. Note that the equalities (\ref{thm:interior:proof:Ehatm}) and (\ref{mfd:interior:var}) still hold. Take ${y}=\exp_{{x}} t\theta\in\MM$, where $t=O(\sqrt{h})$ and $\|\theta\|=1$.
By Lemma \ref{lemma2}, Lemma \ref{lemma6} and (\ref{relateexp}), we have for $j=1,\ldots,d$
\begin{eqnarray}
&&\langle t\iota_*\theta,\ve_j\rangle=\langle t\iota_*\theta,\vu_j(x)+O_p(\hpca^{3/4})\vw_j\rangle\label{uerelation:bdry}\\
&=&\langle\iota({y})-x,\vu_j(x)\rangle- \frac{t^2}{2}\langle \II_{{x}}(\theta,\theta),\vu_j(x)\rangle+O_p(\hpca^{3/4}h^{1/2})+O(h^{3/2})\nonumber\\
&=&\vy_j+O(\hpca^{3/4}h^{1/2}+\hpca^{1/2}h)\nonumber,
\end{eqnarray}
By the same arguments as that in (\ref{proof:thm1:m_taylor}) and by (\ref{uerelation:bdry}), we have
\begin{equation}
\begin{split}
&m({y})-m({x})=t\theta\nabla m({x})+\frac{t^2}{2}\Hess m({x})(\theta,\theta)+O(t^3)\nonumber\\
=\,&\sum_{j=1}^d\langle t\iota_*\theta,\ve_j\rangle\nabla_{\partial_j} m({x})+\frac{1}{2}\sum_{i,j=1}^d\langle t\iota_*\theta,\ve_i\rangle\langle t\iota_*\theta,\ve_j\rangle\Hess m({x})(\partial_i,\partial_j)+O(h^{\frac{3}{2}})\nonumber\\
=\,&\vy^T\nabla m({x})+\frac{1}{2}\vy^T\Hess m({x})\vy+O_p(\hpca^{3/4}h^{1/2}+\hpca^{1/2}h),\nonumber
\end{split}
\end{equation}
which leads to the following equality
\[
\WW_x\vm=\WW_x\Big(\XX_x\Big[\begin{array}{c}
m({x})\\
\nabla m({x})
\end{array} 
\Big]+\frac{1}{2}\QQQ_m({x})+O_p(\hpca^{3/4}h^{1/2}+\hpca^{1/2}h)\Big)
\]
since the kernel $K$ is compactly supported. By a direct calculation, the conditional bias is reduced to 
\begin{equation}\label{mfd:bdry:focus}
\EE\{\hat{m}({x})-m({x})| \mathcal{X}\}=\vv_1^T(\XX_x^T\WW_x\XX_x)^{-1}\XX_x^T\WW_x[\QQQ_m({x})/2+O_p(\hpca^{3/4}h^{1/2}+\hpca^{1/2}h)].
\end{equation}
By taking the boundary effect into consideration and the similar arguments as those in the proof of Theorem \ref{thm:interior}, we have
\begin{eqnarray}
\frac{1}{n}\XX_x^T\WW_x\XX_x&=&f({x})C\nu_{1,x}C\nonumber\\
&&+\left[
\begin{array}{cc}
O_p(\sqrt{h})+O_p\Big(\frac{1}{n^{\frac{1}{2}}h^{\frac{d}{4}}}\Big) & O_p(h)+O_p\Big(\frac{1}{n^{\frac{1}{2}}h^{\frac{d}{4}-\frac{1}{2}}}\Big)\\
O_p(h)+O_p\Big(\frac{1}{n^{\frac{1}{2}}h^{\frac{d}{4}-\frac{1}{2}}}\Big) & O_p(h^{\frac{3}{2}})+O_p\Big(\frac{1}{n^{\frac{1}{2}}h^{\frac{d}{4}-1}}\Big)
\end{array}
\right]\nonumber
\end{eqnarray}
where $\nu_{1,x}$ and $C$ are respectively defined in (\ref{thm1:statement:nuixdef}) and (\ref{thm2:C}). The invertibility of $\frac{1}{n}\XX_x^T\WW_x\XX_x$ follows from the assumption (\ref{conditions:statement:fcond}) and (\ref{conditions:statement:Csqrthx}). Indeed, from (\ref{conditions:statement:fcond}) and (\ref{conditions:statement:Csqrthx}) we know 
$$
f({x})\nu_{1,x,11}=f({x})\int_{h^{-1/2}\exp^{-1}_{{x}}\mathfrak{D}}K(y)\ud y>0,
$$ 
and hence Minkowski's inequality implies that with probability tending to 1, the invertibility holds. 
The binomial inverse theorem yields that
\begin{eqnarray}
\Big(\frac{1}{n}\XX_x^T\WW_x\XX_x\Big)^{-1} &=& \frac{C^{-1}\nu_{1,x}^{-1}C^{-1}}{f({x})}\label{mfd:bdry:XWXinv}\\
&&+\left[
\begin{array}{cc}
O_p(\sqrt{h})+O_p\Big(\frac{1}{n^{\frac{1}{2}}h^{\frac{d}{4}}}\Big) & O_p(1)+O_p\Big(\frac{1}{n^{\frac{1}{2}}h^{\frac{d}{4}+\frac{1}{2}}}\Big)\\
O_p(1)+O_p\Big(\frac{1}{n^{\frac{1}{2}}h^{\frac{d}{4}+\frac{1}{2}}}\Big) & O_p(h^{-\frac{1}{2}})+O_p\Big(\frac{1}{n^{\frac{1}{2}}h^{\frac{d}{4}+1}}\Big)
\end{array}
\right],\nonumber
\end{eqnarray}
where
\begin{eqnarray*}
&& \nu^{-1}_{1,x}:=
\Big[\begin{array}{cc}
\nu^{11}_{1,x} & \nu^{12}_{1,x}\\
(\nu^{12}_{1,x})^T &\nu^{22}_{1,x}
\end{array}
\Big], \quad \nu^{11}_{1,x}:=(\nu_{1,x,11}-\nu_{1,x,12}\nu^{-1}_{1,x,22}\nu^T_{1,x,12})^{-1},\\
&&\nu^{22}_{1,x}:=\left( \nu_{1,x,22}-\nu^T_{1,x,12}\nu_{1,x,11}\nu_{1,x,12}\right)^{-1},
\mbox{ and }\quad
\nu^{12}_{1,x}:=-(\nu^{-1}_{1,x,11}\nu_{1,x,12})\nu^{22}_{1,x}.
\end{eqnarray*}

The term $\frac{1}{n}\XX_x^T\WW_x\QQQ_m({x})$ in (\ref{mfd:bdry:focus}) is evaluated by following the same lines as those in (\ref{thm1:proof:XTWQ}) except for the boundary effect. 
By the same arguments as those used to calculate the term $\mathfrak{q}_1$ in (\ref{thm1:proof:XTWQ}), we have
\begin{eqnarray}
\mathfrak{q}_1&=& \int_{\exp_{{x}}\mathfrak{D}({x})} K_{h}(y,x)(y-x)^T\HHH(y-x) f({y})\ud V({y}) + O_p\Big(\frac{1}{n^{1/2}h^{d/4-1}}\Big)\nonumber\\
&=& hf({x})\int_{\frac{1}{\sqrt{h}}{\mathfrak{D}}({x})} K(\|u\|)u^T\Hess m({x})u \ud u + O_p(h^{3/2})+ O_p\Big(\frac{1}{n^{1/2}h^{d/4-1}}\Big),\nonumber 
\end{eqnarray}
where the first equality comes from the CLT and the second equality comes from Lemma \ref{lemma4} and the change of variable. Choose $\{\veee_k\}_{k=1}^p$ as an orthonormal basis of $\RR^p$. By the same arguments as those in (\ref{proof:thm1:XWQ2}), 
\begin{eqnarray}
\mathfrak{q}_2&=&B^T_x\sum_{k=1}^p\tr \Big(\HHH\int_{\exp_{{x}}\mathfrak{D}({x})} K_{h}(y,x)(y-x)(y-x)^T\langle y-x,\veee_k\rangle\nonumber\\
&&\qquad\qquad\qquad\qquad\qquad\qquad\qquad\qquad \times f({y})\ud V({y}) \Big)\veee_k+ O_p\Big(\frac{1}{n^{\frac{1}{2}}h^{\frac{d}{4}-\frac{3}{2}}}\Big)\nonumber\\
&=&h^{3/2}f({x})\int_{\frac{1}{\sqrt{h}}{\mathfrak{D}}({x})} K(\|u\|)u^T\Hess m({x})uu \ud u+O_p(h^{2})+O_p\Big(\frac{1}{n^{\frac{1}{2}}h^{\frac{d}{4}-\frac{3}{2}}}\Big),\nonumber
\end{eqnarray}
where the first equality comes from (\ref{vtAv}) and the second one comes from the assumption $\hpca\leq h$.
Since $m\in C^3$ and $\MM$ is compact, the remainder term in (\ref{mfd:bdry:focus}) is bounded by  
$O_p(\hpca^{3/4}h^{1/2}+\hpca^{1/2}h)$. 
Thus, since $\hpca\leq h$ by assumption, it follows from (\ref{vtAv}) that
\begin{eqnarray}
&&\EE\{\hat{m}({x},h)-m({x})| \mathcal{X}\}\label{thm:bdry:cond_bias:proof}\\
&=&h\frac{\vv^T_1\nu^{-1}_{1,x}}{2}\int_{\frac{1}{\sqrt{h}}\mathfrak{D}({x})}K(\|u\|)u^T\Hess m({x})u\Big[\begin{array}{c}
1\\ u
\end{array}\Big]\ud u\nonumber\\
&&+O_p(\hpca^{\frac{3}{4}}h^{\frac{1}{2}}+\hpca^{\frac{1}{2}}h)+O_p\Big(\frac{1}{n^{\frac{1}{2}}h^{\frac{d}{4}-1}}\Big)\nonumber\\
&=&h\frac{\tr (\Hess m({x})\nu_{1,x,22})}{2(\nu_{1,x,11}-\nu_{1,x,12}\nu^{-1}_{1,x,22}\nu_{1,x,21})}+O_p(\hpca^{\frac{3}{4}}h^{\frac{1}{2}}+\hpca^{\frac{1}{2}}h)+O_p\Big(\frac{1}{n^{\frac{1}{2}}h^{\frac{d}{4}-1}}\Big).\nonumber
\end{eqnarray}
The conditional variance is evaluated by the same lines as those in (\ref{mfd:interior:XWVWX}):
\begin{eqnarray}
&&\frac{1}{n}\XX_x^T\WW_x\SSS_x\WW_x\XX_x=h^{-\frac{d}{2}}\sigma^2({x})f({x})C\nu_{2,x}C\label{mfd:bdry:XWVWX}\\
&&\qquad+h^{-\frac{d}{2}}\left[
\begin{array}{cc}
O_p(h^{1/2})+O_p\Big( \frac{1}{n^{1/2}h^{d/4}}\Big) &O_p(h)+O_p\Big( \frac{1}{n^{1/2}h^{d/4-1/2}}\Big) \\
O_p(h)+O_p\Big( \frac{1}{n^{1/2}h^{d/4-1/2}}\Big) & O_p(h^{3/2})+O_p\Big( \frac{1}{n^{1/2}h^{d/4-1}}\Big)
\end{array}
\right]\nonumber,
\end{eqnarray}
which when combined with (\ref{mfd:bdry:XWXinv}) leads to
\begin{eqnarray}
&&(\XX_x^T\WW_x\XX_x)^{-1}(\XX_x^T\WW_x\SSS_x\WW_x\XX_x)(\XX_x^T\WW_x\XX_x)^{-1}\label{mfd:bdry:XWXinvXWVWXXWXinv}\\
&=&\frac{1}{nh^{\frac{d}{2}}}\frac{\sigma^2({x})}{f({x})}C^{-1}\nu_{1,x}^{-1}\nu_{2,x}\nu_{1,x}^{-1}C^{-1}\nonumber\\
&&+\frac{1}{nh^{\frac{d}{2}}}\left[
\begin{array}{cc}
O_p(h^{\frac{1}{2}})+O_p\Big( \frac{1}{n^{\frac{1}{2}}h^{\frac{d}{4}}}\Big) &O_p(1)+O_p\Big(\frac{1}{n^{\frac{1}{2}}h^{\frac{d}{4}+\frac{1}{2}}}\Big) \\
O_p(1)+O_p\Big(\frac{1}{n^{\frac{1}{2}}h^{\frac{d}{4}+\frac{1}{2}}}\Big) & O_p(h^{\frac{-1}{2}})+O_p\Big( \frac{1}{n^{\frac{1}{2}}h^{\frac{d}{4}+1}}\Big)
\end{array}
\right].\nonumber
\end{eqnarray}
From (\ref{mfd:bdry:XWXinvXWVWXXWXinv}), since $\vv^T_1C^{-1}=\vv^T_1$, we have
\begin{eqnarray}
\operatorname{Var}\{\hat{m}({x},h)| \mathcal{X}\}
=\frac{\vv^T_1\nu_{1,x}^{-1}\nu_{2,x}\nu^{-1}_{1,x}\vv_1}{nh^{\frac{d}{2}}}\frac{\sigma^2({x})}{f({x})}+O_p\Big(\frac{1}{nh^{\frac{d}{2}-\frac{1}{2}}}+\frac{1}{n^{\frac{3}{2}}h^{\frac{3d}{4}}}\Big).\nonumber
\end{eqnarray}
Putting this together with (\ref{thm:bdry:cond_bias:proof}) we obtain the conditional MSE of $\hat{m}({x},h)$. 

With (\ref{mfd:bdry:XWXinv}), (\ref{mfd:bdry:XWVWX}) and the fact that $\vv^T_{i+1}C^{-1}=h^{-1/2}\vv^T_{i+1}$, the conditional bias and the conditional variance of the estimator of the first order covariance derivative of $m({x})$ are clear by the same calculation. For $i=1,\ldots,d$,
\begin{eqnarray}
&&\EE\{\widehat{\nabla_{\partial_i}m}({x},h)-\nabla_{\partial_i}m({x})| \mathcal{X}\}\\
&=&\sqrt{h}\frac{\vv^T_{i+1}\nu^{-1}_{1,x}}{2}\int_{\frac{1}{\sqrt{h}}\mathfrak{D}({x})}K(\|u\|)u^T\Hess m({x})u\Big[\begin{array}{c}
1\\ u
\end{array}\Big]\ud u\nonumber\\
&&+O_p(\hpca^{3/4}+\hpca^{1/2}h^{\frac{1}{2}})+O_p\Big(\frac{1}{n^{\frac{1}{2}}h^{\frac{d}{4}+1}}\Big)\nonumber
\end{eqnarray}
and
\begin{eqnarray}
\operatorname{Var}\{\widehat{\nabla_{\partial_i}m}({x},h)| \mathcal{X}\}&=&\frac{\vv^T_{i+1}\nu_{1,x}^{-1}\nu_{2,x}\nu^{-1}_{1,x}\vv_{i+1}}{nh^{d/2+1}}\frac{\sigma^2({x})}{f({x})}\\
&&+O_p\Big(\frac{1}{nh^{\frac{d}{2}+\frac{1}{2}}}+\frac{1}{n^{\frac{3}{2}}h^{\frac{3d}{4}}}\Big).\nonumber
\end{eqnarray}
Then the conditional MSE of $\widehat{\nabla_{\partial_i}m}({x},h)$ follows from the above results. 
}
\end{proof}

\subsection{[Proof of Corollary \ref{col:smoothbdry}]}
\begin{proof}
{\allowdisplaybreaks
The proof is finished by simplifying the conditional bias term (\ref{thm:bdry:cond_bias:proof}) when the boundary $\partial\MM$ is smooth. We should show that the conditional bias term is actually the linear combination of second order covariant derivatives of $m$ at ${x}$. 
We first symmetrize the integration domain $\mathfrak{D}({x})$ as follows. Suppose 
$$
x_\partial=\argmin_{y\in \partial\MM}d(y,{x})
$$ 
and 
$$
\tilde{h}({x})=\min_{y\in \partial\MM}d(y,{x})<\sqrt{h}.
$$
Choose a normal coordinate $\{\partial_i\}_{i=1}^d$ on the geodesic ball $B^\MM_{\sqrt{h}}({x})$ around ${x}$ so that $x_\partial=\exp_{{x}}(\tilde{h}({x})\partial_d({x}))$. 
Divide $\mathfrak{D}({x})$ into slices $S_{\eta}\subset\RR^{d-1}$, that is, 
$$
\mathfrak{D}({x})=\cup_{\eta=-\sqrt{h}}^{\sqrt{h}}S_\eta,
$$ 
where 
$$
S_{\eta}:=\{ \mathbf{v}\in \RR^{d-1}: \|(\mathbf{v},\eta)\|_{\RR^d}<\sqrt{h} \},
$$
and $\eta\in[-\sqrt{h},\sqrt{h}]$. Define $\tilde{S}_{\eta}$ so that
$$
\tilde{S}_{\eta}:=\cap^{d-1}_{i=1}(R_iS_{\eta}\cap S_{\eta}),
$$
where $R_i$ is the reflection of $\RR^d$ with respect to the $i$-th coordinate. The symmetrization of $\mathfrak{D}({x})$ is thus defined as 
$$
\tilde{\mathfrak{D}}({x}):= \cup_{{\eta}=-\sqrt{h}}^{\sqrt{h}} \tilde{S}_{\eta}.
$$ 
Since $\partial \MM$ is a smooth $(d-1)$-dimensional manifold, by Lemma \ref{lemma2} we can approximate $\exp^{-1}_{{x}}(\exp_{{x}}\mathfrak{D}({x})\cap \partial\MM)$ by a homogeneous degree $2$ polynomial defined on $T_{\exp^{-1}(x_\partial)}\exp^{-1}_{{x}}(\exp_{{x}}\mathfrak{D}({x})\cap \partial\MM)$, whose graph is symmetric in all coordinates, with error $O(h^{3/2})$. Thus, the error of approximating $S_\eta$ by $\tilde{S}_\eta$ is of order $O(h^{3/2})$ and hence the volume of the set $\tilde{\mathfrak{D}}({x})\Delta \mathfrak{D}({x})$ is
\begin{equation}\label{proof:corollary:symmetrization:difference}
\text{Vol}\Big(\tilde{\mathfrak{D}}({x})\Delta \mathfrak{D}({x})\Big)=O(h^{d/2+1}).
\end{equation}
We also denote
\begin{eqnarray}
\alpha({x})&:=&\int_{\frac{1}{\sqrt{h}}\tilde{\mathfrak{D}}({x})}K(\|u\|)\ud u,\\
\beta({x})&:=&\int_{\frac{1}{\sqrt{h}}\tilde{\mathfrak{D}}({x})}K(\|u\|)u_d\ud u,\label{thm:bdry:cond_bias:smooth:alpha}\\
\Gamma({x})&:=&\diag(\gamma_1({x}),\ldots,\gamma_d({x})),\\
\gamma_i({x})&:=&\int_{\frac{1}{\sqrt{h}}\tilde{\mathfrak{D}}({x})}K(\|u\|)u^2_i\ud u,\, i=1,\ldots,d.\label{thm:bdry:cond_bias:smooth:gamma}
\end{eqnarray}
Thus, since $\tilde{\mathfrak{D}}({x})$ is symmetric in the first $d-1$ directions, by (\ref{proof:corollary:symmetrization:difference}) we have
\begin{eqnarray}
&&\int_{\frac{1}{\sqrt{h}}\mathfrak{D}({x})}K(\|u\|)\ud u=\int_{\frac{1}{\sqrt{h}}\tilde{\mathfrak{D}}({x})}K(\|u\|)\ud u+O(h)=\alpha({x})+O(h),\nonumber\\
&&\int_{\frac{1}{\sqrt{h}}\mathfrak{D}({x})}K(\|u\|)u^T\ud u=\int_{\frac{1}{\sqrt{h}}\tilde{\mathfrak{D}}({x})}K(\|u\|)u^T\ud u+O(h)=\beta\vv^T_d({x})+O(h),\nonumber
\end{eqnarray}
and
\begin{eqnarray}
&&\int_{\frac{1}{\sqrt{h}}\mathfrak{D}({x})}K(\|u\|)uu^T\ud u=\int_{\frac{1}{\sqrt{h}}\tilde{\mathfrak{D}}({x})}K(\|u\|)uu^T\ud u+O(h)=\Gamma({x})+O(h).\nonumber
\end{eqnarray}
Hence, we get the following equations:
\begin{eqnarray}
&&\nu^{11}_{1,x}=\frac{1}{\alpha({x})-\beta({x})^2\gamma_{d}({x})}+O(h),\label{Ninv0}\\
&&\nu^{12}_{1,x}=\frac{-\beta({x})\gamma_{d}({x})}{\alpha({x})-\beta({x})^2\gamma_{d}({x})}\vv^T_d+O(h),\label{Ninv1}\\
&& \nu^{22}_{1,x}=
\Gamma({x})^{-1}+O(h). \, \label{Ninv2}
\end{eqnarray} 
Similarly, by the symmetry of $\tilde{\mathfrak{D}}({x})$, we have
\begin{align}
\int_{\frac{1}{\sqrt{h}}\mathfrak{D}({x})}K(\|u\|)&u^T\Hess m({x})u\left[\begin{array}{c}
1\\ u
\end{array}\right]\ud u\nonumber\\
&=\int_{\frac{1}{\sqrt{h}}\tilde{\mathfrak{D}}({x})}K(\|u\|)u^T\Hess m({x})u\left[\begin{array}{c}
1\\ u
\end{array}\right]\ud u+O(h).\label{smoothbdry}
\end{align}
Plugging (\ref{Ninv0}), (\ref{Ninv1}), (\ref{Ninv2}), and (\ref{smoothbdry}) into (\ref{thm:bdry:cond_bias:proof}) leads to
\begin{eqnarray}
&&\frac{\tr (\Hess m({x})\nu_{1,x,22})}{2(\nu_{1,x,11}-\nu_{1,x,12}\nu^{-1}_{1,x,22}\nu_{1,x,21})}=\frac{\sum_{k=1}^d\gamma_k({x})\gamma_d({x})\nabla^2_{\partial_k,\partial_k}m({x})}{2[\alpha({x})\gamma_d({x})-\beta({x})^2]},
\end{eqnarray}
which finishes the claim. Moreover, by the Cauchy-Schwartz inequality, $\alpha({x})\gamma_d({x})-\beta({x})^2>0$ for all ${x}\in\MM_{\sqrt{h}}$. Since $\MM$ is compact, the uniform boundedness of $\frac{\gamma_k({x})\gamma_d({x})}{\alpha({x})\gamma_d({x})-\beta({x})^2}$ holds as is claimed.
}
\end{proof}

\end{document}